\documentclass[oneside]{amsart}

\usepackage[width=170mm,top=30mm,bottom=32mm]{geometry}
\usepackage[utf8]{inputenc}

\usepackage[style=alphabetic,maxbibnames=99,maxalphanames=5]{biblatex}
\usepackage{newunicodechar}

\newunicodechar{̵}{}
\usepackage{caption}
\usepackage{subcaption}
\usepackage[centertags]{amsmath}
\usepackage{bbm, amsfonts,amssymb, amsthm}
\usepackage{mathrsfs}
\usepackage{fancyhdr}
\usepackage{graphicx}
\usepackage{amsmath}
\usepackage{tikz-cd} 
\usepackage{cancel}
\usepackage{physics}

\usepackage{mathtools}
\usepackage{hyperref}
\usepackage{lipsum}

\usepackage[normalem]{ulem} 
\usepackage[all]{xy}
\usepackage{enumitem}
\usepackage{mathtools}
\usepackage{mathrsfs}
\usepackage{amssymb}
\usepackage{amsmath,amsfonts,amsthm}
\usepackage{graphicx}
\usepackage{upgreek}
\usepackage[toc,page]{appendix}
\usepackage{bbm}
\usepackage{esint}
\usepackage{scalerel}
\usepackage{stackengine,wasysym}
\usepackage{verbatim}
\usepackage{hyperref}
\usepackage[dvipsnames]{xcolor}
\hypersetup{colorlinks=true, linktocpage=true,citecolor=Emerald, linkcolor=RedViolet, urlcolor=RedViolet}
\usepackage{geometry}

\newtheorem*{Theorem*}{Theorem}
\newtheorem*{Problema*}{Problem}

\newtheorem{thm}{Theorem}[section]

\newtheorem{lema}[thm]{Lemma}
\newtheorem{prop}[thm]{Proposition}
\newtheorem{cor}[thm]{Corollary}
\newtheorem{exam}[thm]{Example}
\newtheorem{defn}[thm]{Definition} 
\newtheorem{remark}[thm]{Remark}

\newtheorem{teo}{Theorem}[section]
\newtheorem{corollary}{Corollary}[section]

\newcommand{\cA}{\mathcal{A}}\newcommand{\cB}{\mathcal{B}}

\newcommand{\cG}{\mathcal{G}}\newcommand{\cH}{\mathcal{H}}
\newcommand{\cJ}{\mathcal{J}}
\newcommand{\cL}{\mathcal{L}}
\newcommand{\cM}{\mathcal{M}}

\newcommand{\cV}{\mathcal{V}}
\newcommand{\cX}{\mathcal{X}}

\newcommand{\bC}{\mathbb{C}}

\newcommand{\bN}{\mathbb{N}}

\newcommand{\bR}{\mathbb{R}}
\newcommand{\bS}{\mathbb{S}}

\newcommand{\bZ}{\mathbb{Z}}

\newcommand{\nc}{\newcommand}
\nc{\p}{\partial}
\nc{\ol}{\overline}

\def\@maketitle{%
  \normalfont\normalsize
  \let\@makefnmark\relax  \let\@thefnmark\relax
  \ifx\@empty\@date\else \@footnotetext{\@setdate}\fi
  \ifx\@empty\@subjclass\else \@footnotetext{\@setsubjclass}\fi
  \ifx\@empty\@keywords\else \@footnotetext{\@setkeywords}\fi
  \ifx\@empty\thankses\else \@footnotetext{%
    \def\par{\let\par\@par}\@setthanks}\fi
  \@mkboth{\@nx\shortauthors}{\@nx\shorttitle}%
  \global\topskip42\p@ %
  \@settitle
  \ifx\@empty\authors \else \@setauthors \fi
  \ifx\@empty\@commby
  \else
    \baselineskip18\p@
    \vtop{\centering{\footnotesize\@commby\@@par}%
      \global\dimen@i\prevdepth}\prevdepth\dimen@i
  \fi
  \ifx\@empty\@dedicatory
  \else
    \baselineskip18\p@
    \vtop{\centering{\footnotesize\itshape\@dedicatory\@@par}%
      \global\dimen@i\prevdepth}\prevdepth\dimen@i
  \fi
  \@setabstract
  \normalsize
  \dimen@34\p@ \advance\dimen@-\baselineskip
  \vskip\dimen@\relax
}

\definecolor{myteal}{HTML}{00797d}
\hypersetup{
    colorlinks=false,%
    citebordercolor=green,%
    filebordercolor=red,%
    linkbordercolor=blue,%
    urlbordercolor=red}

\usepackage[T1]{fontenc}
\usepackage{lettrine}

\title{On the Morse index of nonorientable minimal surfaces in $\mathbb{R}^3$}

\author{Carlos Andrés Toro Cardona}
\address{Department of Mathematics, Indiana University Bloomington}
\email{cartoroc@iu.edu}
\author{Carlos Granada-Palacio}
\address{Instituto de Matemática Pura e Aplicada}
\email{carlos.palacio@impa.br
}
\author{Ivan Miranda}
\address{Instituto de Matemática Pura e Aplicada}
\email{ivan.miranda@impa.br}

\begin{document}

\begin{abstract}
        We compute the Morse index of several complete nonorientable minimal surfaces immersed in Euclidean three-space. In particular, we prove that the Meeks minimal M\"obius band has Morse index two, which is the least possible Morse index for such nonorientable minimal surfaces, by previous works of Ros and Chodosh-Maximo. This is the first known minimal surface with index two. We show that the López minimal Klein bottle has Morse index three. As a consequence of this computation, we extend results in the literature and observe that  any complete minimal surface immersed in Euclidean three-space with total curvature $8 \pi$ has Morse index three. We also compute the Morse index of the Oliveira family of minimal M\"obius bands in terms of their total curvature and show that every non-negative integer is the Morse index of a complete minimal surface immersed in Euclidean three-space. We show that five is a sharp lower bound for the Morse index of any minimal oriented double cover in Euclidean three-space, and equality is attained by the double cover of the Meeks minimal M\"obius band. Finally, we prove results towards the classification of complete minimal surfaces with Morse index two immersed in Euclidean three-space, obtaining topological and geometric restrictions on such surfaces.

\end{abstract}

\thanks{I.M. was financed in part by the Coordenação de Aperfeiçoamento de Pessoal de Nível Superior - Brasil (CAPES) – Finance Code  001. C.A.G.P. was supported by CNPq – Conselho Nacional
de Desenvolvimento Científico e Tecnológico. C.A.T.C. thanks FAPERJ (Grant number E26/202.321/2024) for supporting this research project.}
\maketitle

	\setcounter{tocdepth}{1}
	\tableofcontents

\section{Introduction}
    
    We study minimal surfaces immersed in the Euclidean three-space $\mathbb{R}^3$. Minimal surfaces are defined as critical points for the area functional and can be characterized as immersions with mean curvature zero. The Morse index of a minimal surface encodes the maximal number of linearly independent directions of compactly supported variations that decrease its area up to second order. 

        Orientable complete minimal surfaces immersed in $\mathbb{R}^3$ with Morse index at most one have long been classified. They are the plane, the catenoid and the Enneper minimal surface. The plane is stable, i.e. has Morse index zero. The classification of complete orientable stable minimal surfaces was established by do Carmo and Peng \cite{DoCarmo}, Fischer-Colbrie and Schoen \cite{fischerschoen1980structure}, and Pogorelov \cite{pogorelov1981stability}, independently. The catenoid and the Enneper minimal surface both have Morse index one, and they are classified by their Morse index among complete orientable minimal surfaces immersed in $\mathbb{R}^3$ \cite{LopezRosWeakly}.

    In contrast, it was proved by Ros \cite{RossStability} and Chodosh and Maximo \cite{ChodoshMaximotopandindexII} that there exists no complete nonorientable minimal surface immersed in Euclidean three-space with Morse index zero or one, respectively. In \cite{ChodoshMaximotopandindexII}, Chodosh and Maximo also prove the non-existence of complete orientable minimal surfaces with Morse index two immersed in $\mathbb{R}^3$. Moreover, there is no complete nonorientable minimal surface immersed in $\mathbb{R}^3$ whose Morse index is known. This motivates the question: are there complete nonorientable minimal surfaces with index two immersed in Euclidean three-space?

    The main goal of this paper is to answer this question and related ones. For this objective, we compute the Morse index of several nonorientable minimal surfaces immersed in $\mathbb{R}^3$ and study the problem of classifying those of least Morse index. 

    The study of the Morse index of minimal surfaces is closely related to their total curvature. It is a classical result due to Fischer-Colbrie \cite{Fischer-Colbrie1} and Ros \cite{RossStability} that a complete minimal surface immersed in $\mathbb{R}^3$ has finite index if and only if it has finite total curvature. Related to the problem of classifying minimal surfaces by Morse index is the problem of classifying them by total curvature. We mention, for instance, that complete orientable minimal surfaces of total curvature at most $4\pi$ are classified: the plane, the catenoid and the Enneper minimal surface. (See the work of Osserman \cite{Osserman}).

    Up to rigid motions and scaling, there is a unique complete nonorientable minimal surface immersed in Euclidean three-space with total curvature at most $6\pi$ \cite{Meeks}. This surface has the topology of a M\"obius band and was constructed by Meeks, \cite{Meeks}. Our first main result is the computation of its Morse index. 
    
{
		\renewcommand{\theteo}{A}
		\begin{teo} \label{thm-Meeks-indice} 
            The Meeks minimal M\"obius band has Morse index two.
		\end{teo}                       
	}

    To our knowledge, the Meeks minimal M\"obius band is the first known complete minimal surface in $\bR^3$ with Morse index two. See the recent comprehensive book by Urbano \cite{Urbano} about the Morse index of minimal submanifolds. We highlight that the Meeks minimal M\"obius band attains the minimum possible Morse index among nonorientable complete minimal surfaces immersed in $\mathbb{R}^3$.

    The classification of complete nonorientable minimal surfaces with total curvature at most $8\pi$ was established by López. It turns out that, up to rigid motions and scaling, in addition to the Meeks minimal M\"obius band there is a unique other nonorientable complete minimal surface with total curvature at most $8\pi$. This other surface has the topology of a once-punctured Klein bottle and total curvature $8\pi$. This is the López minimal Klein bottle constructed in \cite{LopezConstruction}. Our second main result is the computation of its Morse index.

    	{
		\renewcommand{\theteo}{B}
		\begin{teo} \label{thm-index-lopez} 
            The López minimal Klein bottle has Morse index three.
		\end{teo}                      
	}

    Chen and Gackstatter have constructed an orientable complete minimal surface with total curvature $8\pi$, with the topology of a once-punctured torus \cite{CG}. It is known that the Chen-Gackstatter minimal surface also has Morse index three (see \cite{ChodoshMaximotopandindexII,MontielRos}). Complete minimal surfaces with total curvature $8\pi$ have long been classified. Apart from the López minimal Klein bottle and the Chen-Gackstatter minimal torus, there is a family of genus-zero orientable minimal surfaces, described by López in \cite{Lopezorientable8pi}. We have verified that they all have the same Morse index. 
    
    	{
		\renewcommand{\thecorollary}{C}
		\begin{corollary} \label{coro-index-8pi}
            Every complete minimal surface immersed in $\mathbb{R}^3$ with total curvature $8\pi$ has Morse index three.
		\end{corollary} 
	}

    Our main contribution for the proof of Corollary \ref{coro-index-8pi} lies in the computation of the Morse index of the López minimal Klein bottle given by Theorem \ref{thm-index-lopez}. For completeness, we have also included in Section \ref{indexthreecorollary} a computation of the Morse index of the genus-zero examples with total curvature $8\pi$ described in \cite{Lopezorientable8pi}, building on the work of Montiel and Ros \cite{MontielRos}. We note that it follows from Corollary \ref{coro-index-8pi} and the classification of complete minimal surfaces with low Morse index and low total curvature that the Morse index of a complete minimal surface with total curvature at most $8\pi$ is uniquely determined by its total curvature. This is related to a question posed by Fischer-Colbrie in her seminal work \cite{Fischer-Colbrie1} where she proves that the Morse index of a minimal surface is completely determined by its Gauss map and leaves the question whether there exists an explicit relation between the Morse index and the geometry of the Gauss map of a complete minimal surface immersed in $\bR^3$. 
    
    The problem of classifying complete minimal surfaces with total curvature $10\pi$ is open. These surfaces are necessarily nonorientable, and there are infinitely many of them, by work of \cite{Barros, Oliveira, Zhang}. A distinguished member of this family is a minimal M\"obius band constructed by Oliveira \cite{Oliveira}, which generalizes the construction of Meeks \cite{Meeks}. The Oliveira minimal M\"obius band is the unique minimal M\"obius band with total curvature $10\pi$ whose extended Gauss map has a branch point of maximal order at its end \cite[Corollary]{Ishihara}. 

    More generally, Oliveira \cite{Oliveira} constructed a family of minimal M\"obius bands, indexed by their total curvature, which includes as values of total curvature every odd multiple of $2\pi$ greater than $6\pi$. The members of this family are distinguished among complete minimal M\"obius bands by the behavior of their Gauss map at their end \cite[Corollary]{Ishihara}.
    
    With the tools we have developed to prove Theorem \ref{thm-Meeks-indice}, we compute the Morse index of every member of the Oliveira family of minimal M\"obius bands as a linear function of their total curvature.

    	{
		\renewcommand{\theteo}{D}
		\begin{teo} \label{thm-index-oliveira} 
            The Oliveira minimal M\"obius band of total curvature $2m\pi$ has Morse index $m-1$, for every odd integer $m \ge 5$. 
		\end{teo}                       
	}

    Therefore, every non-negative integer is realized as the Morse index of a complete minimal surface immersed in $\mathbb{R}^3$, as we now explain. As a consequence of Theorem \ref{thm-Meeks-indice} and Theorem \ref{thm-index-oliveira}, every positive even number is attained as the Morse index of a complete minimal M\"obius band immersed in $\mathbb{R}^3$. Recall that the plane is stable and that the catenoid and the Enneper surface both have Morse index $1$. The Jorge-Meeks $3$-noid constructed in \cite{jorge1983topology} has Morse index $3$ (\cite{MontielRos}). Finally, by results of Nayatani for $1 \le g \le 37$ in \cite{Nayatani2}, \cite{Nayatani-Costa} and Morabito for $g \ge 38$ in \cite{morabito2009index}, the Morse index of the Costa-Hoffmann-Meeks surface of genus $g$ is $2g + 3$, and these numbers cover all odd integers larger than $3$. 
    
    We now return to the problem of understanding whether the Meeks minimal M\"obius band is the unique minimal surface immersed in Euclidean three-space with Morse index two. Building on the ideas of Chodosh and Maximo \cite{ChodoshMaximotopandindexI, ChodoshMaximotopandindexII}, and relying on Theorems \ref{thm-Meeks-indice} and \ref{thm-index-lopez}, we now derive topological and geometric restrictions for a complete minimal surface with Morse index two. Recall that by work of Chodosh and Máximo \cite{ChodoshMaximotopandindexII}, every index two complete minimal surface immersed in $\mathbb{R}^3$ is necessarily nonorientable.

    	{
		\renewcommand{\theteo}{E}
		\begin{teo} \label{thm-classification-index-two} 
            Let $M$ be a complete minimal surface with Morse index two immersed in $\mathbb{R}^3$. Then $M$ has total curvature at most $10\pi$. Moreover, $M$ is nonorientable, has only one end, and its end has multiplicity three. In addition, the oriented double cover of $M$ has genus at most two.   
		\end{teo}                                  
	}

    The total curvature of a complete nonorientable minimal surface is an integer multiple of $2\pi$ and, as mentioned before, the ones with total curvature at most $8\pi$ are classified as either the Meeks minimal M\"obius band or the López minimal Klein bottle (\cite{LopezUniqueness}, \cite{Meeks}). Therefore, using Theorem \ref{thm-index-lopez} and Theorem \ref{thm-classification-index-two} we conclude that if there exists a complete minimal surface immersed in Euclidean three-space with Morse index two that is not the Meeks minimal M\"obius band up to rigid motions and scaling, then it is necessarily nonorientable with total curvature $10\pi$, has only one end, its end has multiplicity three, and its oriented double cover has genus two.

    We now turn our attention to oriented minimal immersions that are obtained from the pullback of the immersion of a nonorientable minimal surface by its oriented double cover. We refer to such a pullback immersion as a minimal oriented double cover.

	{
		\renewcommand{\theteo}{F}
		\begin{teo} \label{thm-classification-oriented-double-cover} 
            Every complete minimal oriented double cover immersed in Euclidean three-space has Morse index at least five. Moreover, the oriented double cover of the Meeks minimal M\"obius band has Morse index five and the oriented double cover of the López minimal Klein bottle has Morse index seven. 
		\end{teo}                                  
	}

    Theorem \ref{thm-classification-oriented-double-cover} shows that no complete orientable minimal surface with Morse index at most four is a minimal oriented double cover. In particular, the Weierstrass data of a possible example cannot satisfy the conditions described by Meeks \cite{Meeks} for oriented double covers. Note that the classification of immersed orientable minimal surfaces with Morse index three is open \cite{ChodoshTaiwan}, and there are no embedded minimal surfaces in this class by work of Chodosh and Maximo \cite{ChodoshMaximotopandindexII}.

            All of our main theorems are underpinned by Theorem \ref{thm:indexofnonorientable} and its Corollary \ref{cor:indexofnonorientable}, which computes the Morse index of a nonorientable minimal surface in terms of its total curvature, when the surface has finite total curvature and its extended Gauss map has branch values lying in an equator of the complex sphere. This is an adaptation to the nonorientable case of Corollary $15$ of Montiel and Ros \cite{MontielRos}, taking into account the odd-symmetry of the Gauss map of the orientable double cover with respect to the change of sheets involution.
            
            Theorems \ref{thm-Meeks-indice} and \ref{thm-index-oliveira} follow from the study of the Gauss map of the Meeks minimal M\"obius band and of each member of the Oliveira family of minimal M\"obius bands together with applications of Corollary \ref{cor:indexofnonorientable}. 

            To prove Theorem \ref{thm-index-lopez}, we follow a different strategy. This is because the branch values of the extended Gauss map of the López minimal Klein bottle do not lie in an equator of the round sphere. (In fact, they fall in a pair of orthogonal equators). We have used the so-called continuity method. This is inspired by \cite[Theorem 5]{MontielRos} by Montiel and Ros, and the computation of the Morse index of the Costa minimal surface by Nayatani, \cite{Nayatani-Costa}. 

            In order to compute the Morse index of the López minimal Klein bottle, we construct a path of odd meromorphic functions that connect the Gauss map $\mathcal{G}_1$ of the López minimal Klein bottle to a meromorphic function $\mathcal{G}_0$ whose branch values lie in an equator of the complex sphere. The Morse index and nullity of $\mathcal{G}_0$ can be computed with Theorem \ref{thm:indexofnonorientable}. (See Definitions \ref{indexholomorphic} and \ref{indexholomorphicodd} for the definition of index and nullity for a holomorphic map).  We then show that the nullity of the meromorphic functions $\mathcal{G}_t$ are all equal to three. The constancy of the nullity and the continuity in the $C^1$ topology of the path $\{\mathcal{G}_t\}_{t \in [0,1]}$ allow us to conclude that the odd Morse index of $\mathcal{G}_1$ is three, just as that of $\mathcal{G}_0$. 

            To study the nullity of the maps $\mathcal{G}_t$ we rely on Theorem $5$ by Montiel and Ros \cite{MontielRos}.  This result computes the nullity as $3$ plus the dimension of some spaces of meromorphic forms satisfying certain residue and period conditions that we refer to as residue and period problems, see \eqref{eq:spaceresidues} and \eqref{eq:spaceperiods}. In light of this result, to prove that $\mathcal{G}_t$ has nullity three, we prove that the associated residue and period problems have no solution. Our study of these residue and period problems is technical and reminiscent of the work of Nayatani \cite{Nayatani2}, Sarenhu \cite{Sarenhu} and Morabito \cite{morabito2009index}.

                We now highlight some aspects of our approach to this technical part. To study the residue problem, one has to understand the structure of the space of meromorphic one-forms of the underlying Riemann surface structure. We made systematic use of the symmetries of the surface for that. A symmetric choice of basis revealed symmetry in the Residue Matrix (see Theorem \ref{residuematrix}) and this allowed us to exhibit an explicit symmetric basis for the kernel of the Residue Matrix for all $t \in [0,1]$ (in Theorem \ref{Hchapelsymmetricbasis}). This choice of basis was particularly suited to the analysis of the period problem, because it turned out that the Period Matrix is essentially diagonal with respect to it. Having a diagonal Period Matrix reduced the problem of showing that there is no period solution to the problem of showing that six independent integrals are nonzero, for all $t \in [0,1]$ (see Theorem \ref{determinantperiodmatrix}). 
                
                This kind of non-vanishing problem also appears in \cite{morabito2009index, Nayatani2,Sarenhu} and one clever idea that is present in these works is to deform a generator for the homology into simpler paths, so that one can apply tools from the classical theory of elliptic integrals. However, in this process one needs to take care of the poles on the integrand in the region of deformation. It is then typical to search for exact meromorphic one-forms that can be added to the expression to remove the poles inside the region of deformation (see \cite[Appendix]{Sarenhu}). However, this tends to complicate the integrand. We have explored a different idea at this step.
                
                We decompose each of these six integrals as a sum of simpler integrals, which we analyzed independently. We recognized that we could deform the generators of the homology in different directions, ad hoc to each of these simpler integrands, avoiding the usage of new summands. This simplified the expressions. Nonetheless, we had to face difficulties with elliptic integrals and the fact that our path of meromorphic functions is more complicated than what had been studied in the literature. We have developed techniques to deal with these complicated expressions, adapted to each specific case.  We hope these ideas are flexible and can be adapted to other contexts.
                
            The proof of Theorems \ref{thm-classification-index-two} and \ref{thm-classification-oriented-double-cover} build on the ideas of Chodosh and Maximo \cite{ChodoshMaximotopandindexI, ChodoshMaximotopandindexII}. We use their inequalities to bound topology in terms of index. We use the nonorientable version of the Jorge-Meeks theorem due to F. Martin \cite{MartinLopezSurvey, MartinThesis} to bound total curvature in terms of topology. We use classification results in terms of total curvature and topology to rule out some possibilities. In the proof of Theorem \ref{thm-classification-oriented-double-cover} we relate the topological and geometric properties of the oriented double cover with those of the nonorientable surface, and explore the restrictions that the results of Chodosh and Maximo impose on both. We also rely on a theorem of Schoen \cite{schoen1983uniqueness} and its nonorientable adaptation by Kusner \cite{kusner1987conformal}. \\

    \textbf{Structure of the paper.}
 The paper is organized as follows. Section \ref{section::Definitions and preliminary results} contains preliminary results and definitions that are used in the paper. The main goal of Section \ref{section::On the index of Schr\"odinger operators associated to even potentials} is to prove Theorem \ref{thm:indexofnonorientable}. In Section \ref{section::The Meeks minimal M\"obius band and the Oliveira family of minimal M\"obius bands} we prove Theorem \ref{thm-Meeks-indice} and Theorem \ref{thm-index-oliveira}. Section \ref{section::The López minimal Klein bottle} is devoted to the study of the López minimal Klein bottle and culminates with the proof of Theorem \ref{thm-index-lopez}. In Section \ref{section::Minimal surfaces with Morse index two immersed in Euclidean three-space} we prove Theorem \ref{thm-classification-index-two}. The proof of Theorem \ref{thm-classification-oriented-double-cover} is presented in Section \ref{section::The index of a minimal oriented double cover immersed in Euclidean three-space}. The paper contains four appendices. Appendix \ref{section::Riemann surfaces results} contains basic results about Riemann surfaces. Appendix \ref{section::Auxiliary computations with elliptic integrals} contains some technical computations with elliptic integrals needed in the proof of Theorem \ref{thm-index-lopez}. Appendix \ref{section::Explicit estimates} contains some explicit estimates used in Section \ref{section::The López minimal Klein bottle}. Appendix \ref{section::EstimatesPolynomials} contains the description of a framework to prove the positivity of some functions.

\subsection{AI disclosure} We did not make use of LLMs to conduct the research that led to this paper. This article does not contain AI-generated text. 

\subsection{Acknowledgments} We started working on this project when all three of us were PhD students of Lucas Ambrozio at IMPA. We would like to thank him for his support and encouragement during the development of this project.                                               
\section{Definitions and preliminary results} \label{section::Definitions and preliminary results}

\subsection{Morse index of a minimal immersion}

As mentioned in the introduction, minimal immersions into $\mathbb{R}^3$ are defined to be critical points of the area functional under compactly supported variations. By the first variation formula, this is equivalent for the immersion to have zero mean curvature, see \cite{coldingminicozzi}. The Morse index appears when we study the second variation of the area. For a two-sided immersion $X: M \to \mathbb{R}^3$ with Gaussian curvature $\kappa$, the second variation of the area along a deformation in the normal direction with amplitude $u \in C^{\infty}_0(M)$ is given by:
\begin{equation*}
    \eval{\frac{d^2}{dt^2}}_{t=0}\text{area}(M_t) = \int_{M} \abs{\nabla u}^2 + 2 \kappa u^2.
\end{equation*}
This motivates the introduction of the symmetric bilinear form $Q(u,u)$ given by the expression above. This bilinear form is usually called the index form of the minimal immersion. When restricted to test functions supported on a fixed bounded domain $\Omega \subset M$, the bilinear form $Q$, which is associated to the elliptic operator $-\Delta + 2 \kappa$, has finite index due to the compact embedding of $H^1_0(\Omega)$ into $L^2(\Omega)$. Now, since the index is monotonic with respect to the domain, following Fischer-Colbrie \cite{Fischer-Colbrie1}, the Morse index of the immersion is defined as:
\begin{equation}\label{eq:defindex}
    \text{index}(M) = \sup_{\Omega \subset M} \text{index}(Q\lvert_{H^1_0(\Omega)}).
\end{equation}
Fischer-Colbrie proved that the Morse index is finite if and only if the total curvature is finite, see \cite{Fischer-Colbrie1}. Also, recall that by a result of Osserman \cite{Osserman}, when the minimal immersion $X$ is complete, two-sided and has finite total curvature, then $M$ is conformally equivalent to a compact Riemann surface $\Sigma$ minus a finite number of points, and the Gauss map extends as a holomorphic map $\phi:\Sigma\to \bS^2$. 
\subsection{Index and Nullity of a meromorphic map}\label{sec:indexandgaussmap}

The Morse index of a minimal immersion can be computed entirely in terms of its Gauss map. This fact was observed in many pioneering works, for instance \cite{YuenTysk}, \cite{Fischer-Colbrie1}, \cite{MontielRos}. Following \cite{MontielRos}, we briefly describe this relation. Given a complete, two-sided minimal immersion with finite total curvature $X: M \to \mathbb{R}^3$, by Osserman \cite{Osserman}, we know that $M$ is conformal to a compact Riemann surface minus a finite number of points $\Sigma\setminus \{p_i \}$ and the Gauss map extends as a holomorphic map $\phi: \Sigma \to \bS^2$. 

Comparing bilinear forms we have that the Morse index of $X$, defined in \eqref{eq:defindex} coincides with the number of eigenvalues of $-\Delta_{\phi}$ which are less than $2$, where this Laplacian is taken with respect to the branched metric $g_{\phi} = \phi^*(g_{\bS^2})$. The eigenvalues and eigenfunctions are well defined and follow the same classical variational characterization for singular metrics of this kind, \cite{Kokarev}. 

\begin{defn}[{Morse index and nullity of a holomorphic map}]\label{indexholomorphic}
\normalfont
    \textit{The Morse index and the nullity of a holomorphic map} $\phi: \Sigma \to \bS^2$ are defined as the Morse index and the nullity of the operator $L_{\phi} = -\Delta_{\phi} - 2$. We denote them by $\text{index}(\phi)$ and $\text{nul}(\phi)$.
\end{defn}
As a remark, when we have a meromorphic function $\mathcal{G}: M \to \overline{\bC}$ we refer to its Morse index and nullity as the ones corresponding to the associated map taking values in $\bS^2$.
We recall that if the holomorphic map $\phi$ corresponds to the Gauss map of some complete two-sided minimal immersion, then the Morse index of $\phi$ coincides with the Morse index of the associated minimal immersion, \cite{Fischer-Colbrie1}. Also, in this case, the nullity of $\phi$ coincides with the dimension of the space of bounded Jacobi fields of the immersion, see \cite{MontielRos}. 

Observe that the three coordinate functions of $\phi$ are eigenfunctions of $-\Delta_{\phi}$  with eigenvalue $2$, thus, the nullity of $\phi$ is always at least three. In the following, we suppose that the set of branch points of $\phi$ is disjoint to its set of poles. This can be obtained by composing $\phi$ with a rotation, which does not change its Morse index nor nullity. Denoting by $N(\phi)$ the null-space and $L(\phi)$ the space generated by these three trivial eigenfunctions, in \cite{MontielRos} it is shown that $N(\phi)/L(\phi)$ is isomorphic to a space of meromorphic quadratic differentials in $\Sigma$ that we describe below. Consider the meromorphic function $\cG:\Sigma\to \overline{\bC}$ obtained by composing $\phi$ with the stereographic projection. Now, we consider the divisor $R(\phi) = \sum_{i=1}^k p_i$ given by the formal sum of the branch points of $\phi$ without multiplicities and then, define the spaces introduced in \cite{MontielRos}:
\begin{equation}\label{eq:spacedivisors}
     H^{0,2}(\phi) = \{ f\omega_0\otimes\omega_0: \text{div}(f)+2K_{\Sigma}+R(\phi)\geq 0\},
\end{equation}
\begin{equation}\label{eq:spaceresidues}
    \hat{H}(\phi) = \{ \sigma \in  H^{0,2}(\phi): \text{Res}_{p_i} \frac{\sigma}{d \mathcal{G}} = 0, \ i = 1, ..., r \},
\end{equation}
\begin{equation}\label{eq:spaceperiods}
    H(\phi) = \{ \sigma \in \hat{H}(\phi): \text{Re}\int_{\alpha} (1 - \mathcal{G}^2, i(1 + \mathcal{G}^2), 2 \mathcal{G}) \frac{\sigma}{d \mathcal{G}} = 0, \ \forall \alpha \in H_1(\Sigma,\mathbb{Z}) \}.
\end{equation}
Where $K_\Sigma=\text{div}(\omega_0)$ is the canonical divisor of a fixed non-identically zero meromorphic one-form $\omega_0$ in $\Sigma$. Thus we can enunciate the result that we want to use, which is Theorem $5$ in \cite{MontielRos}, saying that $N(\phi)/L(\phi) \cong H(\phi)$.

 Following \cite{MontielRos}, we now describe this isomorphism. Let $\{p_1,\dots,p_k\}$ be the branch points of $\phi$. Given a function $u\in N(\phi)/L(\phi)$, Montiel and Ros construct a branched minimal immersion $X_u:\Sigma \setminus \{p_1,\ldots p_k\}\to \mathbb{R}^3$ with extended Weierstrass data $(\cG,\eta)$ defined by 
     \begin{equation}\label{montielrosmap1}
     X_u :=u\phi+\frac{u_z\phi_{\overline{z}}+u_{\overline{z}}\phi_z }{|\phi_z|^2}=\text{Re}\int_{p_*}^p\left(\frac{1}{2}(1-\cG^2),\frac{i}{2}(1+\cG^2),\cG\right)\eta,
     \end{equation}
     for some $p_*\in M$. Then, $u=\langle X_u,\phi\rangle$ and $\sigma_u\coloneq \eta\ d\cG$ is an element of $H(\phi)
    $.
    
    Conversely, given $\sigma\in H(\phi)$, Montiel and Ros construct a branched minimal immersion of finite total curvature $X:\Sigma\setminus \{p_1,\ldots p_k\}\to \mathbb{R}^3$ defined by
    \begin{equation}\label{montielrosmap2}
    X^{\sigma}(p) := \text{Re}\int_{p_*}^p\left(\frac{1}{2}(1-\cG^2),\frac{i}{2}(1+\cG^2),\cG\right)\frac{\sigma}{d\cG},
    \end{equation}
    for some $p_*\in M$ fixed. Then $u := \langle X^{\sigma},\phi\rangle$ is an element of $ N(\phi)/L(\phi)$. These constructions define the desired isomorphism between $H(\phi)$ and $ N(\phi)/L(\phi)$.

\begin{remark}\label{JorgeMeeksMultiplicitiesvsBranchOrders}
\normalfont
    Given $\sigma \in H(\phi)$, the map introduced in \eqref{montielrosmap2} defines a complete branched  minimal immersion $X: \Sigma \setminus \{ p_{i_1}, ..., p_{i_l} \} \to \bR^3$, where $p_{i_1},..., p_{i_l}$ are the branch points of $\phi$ such that the order of $\sigma$ at $p_{i_s}$ that we denote by $l_{i_s}$ and the branch order of $\mathcal{G}$ at $p_{i_s}$ that we denote by $m_{i_s}$, satisfy $m_{i_s} - l_{i_s} >0$. This condition guarantees that the map defined in \eqref{montielrosmap2} is a branched immersion whose conformal factor has poles only at the points $\{ p_{i_1},...,p_{i_l} \}$ and the Weierstrass data is holomorphic in $\Sigma \setminus \{ p_{i_1},...,p_{i_l} \}$, therefore it defines a complete branched minimal immersion into $\bR^3$. In \cite{JorgeMeeksBranched}, Fang defines a non-negative integer number $J_{i_s}$ as the order the pole of the conformal factor of the immersion at the end $p_{i_s}$. In the above context, we have $ J_{i_s} = m_{i_s}-l_{i_s}$. Finally, observe that $\sigma \in H(\phi) \subset H^{0,2}(\phi)$ implies $l_{i_s}\geq -1$. 
\end{remark}
    
\subsection{Index of nonorientable minimal surfaces}\label{subsec:indexonesided}

For a nonorientable minimal immersion $X: M \to \mathbb{R}^3$, we can define its Morse index by means of its orientable double cover. If $\pi: \Tilde{M} \to M$ is its orientable double cover with induced immersion $\tilde{X}$, then we can consider a variation of $M$ with variational vector field $V \in \Gamma(T^{\perp} M)$ with compact support in $\Omega$. This variation induces a variation in $\tilde{M}$ with magnitude $u_V$ in the normal direction, where $u_V$ is a smooth function with compact support in $\pi^{-1}(\Omega)$ which is odd with respect to the change of sheets involution. Then we can compare the index forms of $X$ and $\tilde{X}$ as:
\begin{equation*}
    Q(V,V) = \int_{M} \lvert \nabla^{\perp} V \lvert^2 - \lvert A \lvert^2 \lvert V \lvert^2 = \frac{1}{2} \int_{\tilde{M}} \lvert \nabla u_V \lvert^2 - \lvert \tilde{A} \lvert^2 u_V^2 = \frac{1}{2} \tilde{Q}(u_V,u_V).
\end{equation*}
Conversely, an odd function with compact support in $\tilde{M}$ induces a variation with compact support of $M$ that verifies the same relation as the one above. This means that over a compact domain $\Omega \subset M$, the Morse index of $M$ equals to the Morse index of $\tilde{M}$ when we consider the index form to be restricted to the space of smooth functions supported in $\pi^{-1}(\Omega)$, that are odd with respect to the change of sheets involution. We call this number \textit{odd Morse index} of $M$ restricted to $\Omega$ and denote it by $\text{index}_{\mathcal{O}}(\pi^{-1}(\Omega))$. Thus, it makes sense to define the Morse index of a complete nonorientable minimal immersion as: 
\begin{equation*}   \text{index}_{\mathcal{O}}(\tilde{M}) = \sup_{\Omega \subset M} \text{index}_{\mathcal{O}} (\pi^{-1}(\Omega)).
\end{equation*}
It is known that for $X: M \to \bR^3$ being a complete nonorientable minimal immersion into $\bR^3$ with finite total curvature, the change of sheets involution in its orientable double cover extends to an anti-holomorphic involution in the compactification of the two-sided immersion. Moreover, this involution permutes the ends of $\tilde{M}$. For these facts, see \cite{MartinLopezSurvey}, \cite{MartinThesis}. This motivates the following definition:

\begin{defn}[{Odd Morse index and odd nullity of a holomorphic map}]\label{indexholomorphicodd}
\normalfont
    Given a holomorphic map $\phi: \Sigma \to \bS^2$, defined on a compact Riemann surface together with an anti-holomorphic involution $\tau: \Sigma \to \Sigma$ such that $\phi \circ \tau = - \phi$, we define \textit{the odd Morse index and the odd nullity} of $\phi$ as being the Morse index and the nullity of the operator $L_{\phi} = -\Delta_{\phi} - 2$ when restricted to the space of functions that are odd with respect to $\tau$. We denote these quantities by $\text{index}_{\mathcal{O}}(\phi)$ and $\text{nul}_{\mathcal{O}}(\phi)$.
\end{defn}

If we denote by $(H^1)_{\mathcal{O}}$ the space of functions in the Sobolev space $H^1(\Sigma)$ that are odd with respect to $\tau: \Sigma \to \Sigma$, then the odd Morse index of $\phi$ is the number of negative eigenvalues of $-\Delta_{\phi} - 2$, where these eigenvalues follow the usual variational characterization:
\begin{equation*}
    (\lambda_k)_{\mathcal{O}} = \inf_{V \in (\mathcal{V}^k)_{\mathcal{O}}} \sup_{u \in V \setminus \{ 0 \}} \frac{\int_{\Sigma} \lvert \nabla^{\phi}u \lvert^2 - 2 u^2}{\int_{\Sigma} u^2}.
\end{equation*}
Where $(\mathcal{V}^k)_{\mathcal{O}}$ is the set of $k-$dimensional subspaces of $(H^1)_{\mathcal{O}}$. these notations are inspired by those introduced by Ambrozio, Buzano, Carlotto and Sharp in \cite{LucasIndex}, where they studied Morse index Schr\"odinger operadors in a complete Riemannian manifold in the presence of an ambient involution.

Now, we describe how to compute the Morse index of a complete nonorientable minimal immersion using the compactification of its orientable double cover. We summarize this in the following theorem, whose proof follows by an adaptation of the arguments of Fischer-Colbrie in \cite{Fischer-Colbrie1} to the nonorientable setting: 

\begin{thm}\label{thm:fischercolbrienaoorientavel}

Let $X: M \to \mathbb{R}^3$ be a complete nonorientable minimal immersion with finite total curvature. Denote by $\pi: \tilde{M} \to M$ the orientable double cover of $M$, and let $\tilde{\Sigma}$ be the compactification of the immersion $\tilde{X}: \tilde{M} \to \mathbb{R}^3$, with extended Gauss map $\phi: \tilde{\Sigma} \to \bS^2$. Then the Morse index of $M$ equals to the odd Morse index of the holomorphic map $\phi$.
    
\end{thm}

\begin{proof}
    As observed above, we have to prove that the odd Morse index of $\tilde{M}$ coincides with the odd Morse index of $L_{\phi}$. By Osserman's theorem and conformal invariance of the energy, if $u \in C^{\infty}_0(\tilde{M})$, then we have that:
    \begin{equation*}
        \tilde{Q}(u,u) = \int_{\tilde{M}} (\lvert \nabla u\lvert^2 - \lvert \tilde{A} \lvert^2 u^2) d\tilde{M}=\int_{\tilde{M}} (\lvert \nabla u\lvert^2 - \lvert \nabla \phi \lvert^2 u^2) d\tilde{M}= \int_{\tilde{\Sigma}} (\lvert \nabla^{\phi} u \lvert^2 - 2 u^2) dv_{g_{\phi}}.
    \end{equation*}
     Therefore, since any odd test function in $\tilde{M}$ extends to zero as an odd function in $\tilde{\Sigma}$, we have that $\text{index}_{\mathcal{O}} \tilde{M} \leq \text{index}_{\mathcal{O}} (\phi)$. 
     
     To prove the opposite inequality, suppose that $\tilde{M} = \tilde{\Sigma} \setminus \{ p_1, ..., p_r, \tau(p_1), ..., \tau(p_r) \}$. Given $n$ odd eigenfunctions $f_1,...,f_n$ of $L_{\phi}$, associated to negative eigenvalues, following the idea of Corollary $2$ in \cite{Fischer-Colbrie1}, we want to multiply $f_i$ by a cut-off function vanishing near the ends of $\tilde{M}$, in order to estimate $\text{index}_{\mathcal{O}} \tilde{M}$. We only have to take care to produce test functions that are odd with respect to $\tau$, however, due to the fact that the anti-holomorphic involution extends to the compactification and permutes the ends \cite{MartinThesis}, the following modification of the test functions used by Fischer-Colbrie in Corollary $2$ of \cite{Fischer-Colbrie1}, works in this case. In fact, given $1 \leq j \leq r$, consider the function:
     \begin{equation*}
        \eta_j(x) = 
        \begin{cases}
        0 & \text{if } d(x,\{p_j, \tau(p_j)\}) < \varepsilon^2\\
        \log (\frac{1}{\varepsilon^2}d(x,\{p_j, \tau(p_j) \}))/\log (\frac{1}{\varepsilon}) & \text{if } \varepsilon^2 < d(x,\{p_j, \tau(p_j) \}) < \varepsilon\\
        1 & \text{if } \varepsilon < d(x,\{p_j, \tau(p_j)\})
        \end{cases}.
    \end{equation*}
Observe these functions are even with respect to $\tau$. Now, we define $\eta$ to be equal to $\eta_j$ in a small neighborhood of each $p_j$ and $\tau(p_j)$, in such a way that this function $\eta$ is even with respect to $\tau$. Then we take $h_i = \eta f_i$. This function is odd with respect to $\tau$ and the same estimates in \cite{Fischer-Colbrie1} holds to conclude that for small enough $\varepsilon>0$, we have that the index form $\tilde{Q}$ is negative definite on the span of the odd functions $h_1,...,h_n$. This proves that $\text{index}_{\mathcal{O}}(\phi) \leq \text{index}_{\mathcal{O}} \tilde{M}$ and then the result follows.
\end{proof}

\subsection{Previous estimates for the Morse index}
An important result in the theory of minimal surfaces is the Jorge-Meeks formula, deduced in \cite{jorge1983topology}. Denote by $Y_s$ the scaling by $\frac{1}{s}$ of the surface intersected by the sphere of radius $s$ centered at the origin. Then for big enough $s$ we have that $Y_s = \{ \gamma^s_1,...,\gamma^s_r \}$ and $X^{-1}(Y_s)$ consists of $r$-immersed Jordan curves in $M$ bounding each end, such that $\gamma^s_i$ converges $C^{\infty}$ as $s \to +\infty$ to a geodesic of $\bS^2$, with multiplicity $d_i$. Moreover, if $\phi$ denotes the extended Gauss map in the compactification, then we have:
\begin{equation}\label{eq:jorgemeeks1}
    2\text{deg}(\phi)=-\mathcal{X}(\Sigma)+\sum_{i=1}^r(d_i+1).
\end{equation}

In the case of a complete nonorientable minimal immersion of finite total curvature $X:M\to \bR^3$, we have that $M=\Sigma\setminus\{p_1,\ldots,p_r\}$ where $\Sigma$ is compact, by the generalized Osserman's theorem due to F. Martin \cite{MartinThesis}. Moreover, the nonorientable version of the Jorge-Meeks formula due to F. Martin \cite[Eq. 25]{MartinLopezSurvey} states that
\begin{equation}\label{jogemeeksformulaonesided}
        \frac{\int_{M} (-\kappa)}{2\pi}=g-1+\sum_{j=1}^r(d_j+1),
 \end{equation}
where $g$ is the genus of the orientable double cover of $\Sigma$.
Moreover, W. Meeks found a formula relating the total curvature of a nonorientable minimal surface with its topology \cite{Meeks}. More precisely,
\begin{equation}\label{eq:meeksmod2}
    \frac{\int_{M} (-\kappa)}{2 \pi} \equiv \mathcal{X}(\Sigma) \ (\text{mod } 2),
\end{equation}
where the Euler characteristic of $\Sigma$ is defined as a half of the Euler characteristic of its orientable double cover. 

Now, we recall the estimates of Chodosh-Maximo in \cite{ChodoshMaximotopandindexII} bounding the index of a minimal surface with the total curvature. Let $X:M\to \bR^3$ be a complete, non-planar, orientable minimal immersion of finite total curvature, where $M=\Sigma\setminus\{p_1,\ldots p_r\}$ and $\Sigma$ compact. Then we have:
\begin{equation}\label{eq:chodosh-maximo1}
        \frac{1}{3} + \frac{1}{6\pi} \int_M (-\kappa) \le \text{index}(M) \le -3 + \frac{3}{2\pi} \int_M (-\kappa).
\end{equation}
Also, they proved under the same hypothesis that if $\Sigma$ has genus $g$ and $r$ ends $p_1, ..., p_r$ with multiplicities $d_1, ..., d_r$, then we have:
\begin{equation}\label{eq:chodosh-maximo3}
        \text{index}(M) \ge \frac{1}{3} \big ( 2g + 2\sum_{j=1}^r (d_j +1) - 5\big).
\end{equation}
For nonorientable minimal immersions, Chodosh-Maximo proved
\begin{equation}\label{eq:chodosh-maximo2}
        \frac{1}{3} + \frac{1}{6\pi} \int_M (-\kappa) \le \text{index}(M) \le -6 + \frac{3}{\pi} \int_M (-\kappa),
\end{equation}
together with
\begin{equation}\label{eq:chodosh-maximo4}
        \text{index}(M) \ge \frac{1}{3} \big ( g + 2\sum_{j=1}^r (d_j +1) - 4\big),
\end{equation}
where $g$ is the genus of the orientable double cover of $M$. We now state an adaptation of the equivalence contained in equations \eqref{montielrosmap1} and \eqref{montielrosmap2} to the nonorientable setting
\begin{thm}\label{nonorientablemontielros}
Let $\Sigma$ be a compact Riemann surface endowed with an anti-holomorphic involution without fixed points $\tau:\Sigma\to \Sigma$. Let $\phi:\Sigma\to \mathbb{S}^2$ be a holomorphic map on $\Sigma$ such that $\phi\circ \tau = -\phi$. We denote by $\cG:\Sigma\to \overline{\bC}$ the map obtained from $\phi$ by composition with the stereographic projection. Let $u\in N(\phi)/L(\phi)$ and consider $\sigma_u \in H(\phi)$ as constructed by the Montiel-Ros isomorphism described in equations \eqref{montielrosmap1} and \eqref{montielrosmap2}.

Then $u \circ \tau = - u$ if and only if $\overline{\tau^{*}\sigma_u}=-\sigma_u$. Moreover, the complete branched minimal immersion $$X_u : \Sigma \backslash \{p_{i_1},\dots,p_{i_l}\}\to \bR^3$$ described in Remark \ref{JorgeMeeksMultiplicitiesvsBranchOrders} is the oriented double cover of a branched complete nonorientable minimal surface of finite total curvature. 

\end{thm}

\begin{proof}
The complete branched minimal immersion $X_u:\Sigma\setminus \{p_{i_1},\ldots,p_{i_l}\}\to \bR^3$ satisfies by equation \eqref{montielrosmap2} for all $p\in \Sigma\setminus\{p_{i_1},\ldots, p_{i_l}\}$
\begin{align*}
\begin{split}
X_u(\tau(p)) &= \text{Re}\int_{p_*}^{\tau(p)} \left(1-\cG^2,i(1+\cG^2),2\cG\right)\frac{\sigma_u}{2d\cG} = \text{Re}\int_{\tau(p_*)}^p \left(1-\frac{1}{\overline{\cG}^2},i(1+\frac{1}{\overline{\cG}^2}),-\frac{2}{\overline{\cG}}\right)\frac{1}{2}\frac{\tau^*\sigma_u}{d \overline{\cG}/\overline{\cG}^2}\\& = \text{Re}\int_{\tau(p_*)}^p \overline{ \left(1-\cG^2,i(1+\cG^2),2\cG\right)\frac{1}{2}\frac{ (-\overline{\tau^*\sigma_u}) }{d  \cG}} = \text{Re}\int_{\tau(p_*)}^p (1-\cG^2,i(1+\cG^2),2\cG)\frac{1}{2} \frac{ (-\overline{\tau^*\sigma_u}) }{d  \cG}
\end{split}
\end{align*}
First, assume that $\sigma_u\in H(\phi)$ is antisymmetric with respect to $\tau$ \textit{i.e.} $-\overline{\tau^{*}\sigma_u}=\sigma_u$. Then by the previous equation we have that for all $p\in \Sigma\setminus\{p_{i_1},\ldots, p_{i_l}\}$
\begin{equation*}
    X_u(\tau(p))=X_u(p) - X_u(\tau(p_*)).
\end{equation*} 
This forces $X_u(\tau(p_*))=0$ and consequently $X_u \circ \tau = X_u$. Therefore, as $u = \langle X_u, \phi \rangle$, we have that $u \circ \tau = - u$.

Conversely, fix $u\in N(\phi)/L(\phi)$ such that $u\circ \tau=-u$. If $z$ and $w$ denote holomorphic coordinates around $p$ and $\tau(p)$, respectively, we have
\begin{align*}
\begin{split}
&u(\tau(p))=-u(p),\quad u_w(\tau(p))=-u_{\overline{z}}(p),\quad  u_{\overline{w}}(\tau(p))=-u_{z}(p),\\ &\phi(\tau(p))=-\phi(p),\quad \phi_w(\tau(p))=-\phi_{\overline{z}}(p),\quad  \phi_{\overline{w}}(\tau(p))=-\phi_{z}(p).
\end{split}
\end{align*}
This implies by \eqref{montielrosmap1} that
\begin{align*}
\begin{split}
X_u(\tau(p))&=u(\tau(p))\phi(\tau(p))+\frac{u_w(\tau(p))\phi_{\overline{w}}(\tau(p))+u_{\overline{w}}(\tau(p))\phi_w (\tau(p))}{|\phi_w(\tau(p))|^2}\\&=u(p)\phi(p)+\frac{u_{\overline{z}}(p)\phi_z(p) +u_z(p)\phi_{\overline{z}}(p)}{|\phi_{\overline{z}}(p)|^2}\\&=X_u(p).
\end{split}
\end{align*}
Therefore, we obtain
\begin{equation*}
   \text{Re}\int_{\tau(p_*)}^p (1-\cG^2,i(1+\cG^2),2\cG)\frac{1}{2} \frac{ (-\overline{\tau^*\sigma_u}) }{d  \cG}=X_u(\tau(p))=X_u(p)= \text{Re}\int_{p_*}^p (1-\cG^2,i(1+\cG^2),2\cG)\frac{1}{2} \frac{\sigma_u}{d\cG}
\end{equation*}
Upon differentiation and the Cauchy-Riemann equations we conclude that
\begin{equation*}
    \overline{\tau^* \sigma_u} = - \sigma_u
\end{equation*}
as we wanted to prove.
\end{proof}

\section{On the Morse index of Schr\"odinger operators associated to odd holomorphic maps} \label{section::On the index of Schr\"odinger operators associated to even potentials}

We extend some of the ideas from the work of S. Montiel and A. Ros \cite{MontielRos} about the index of Schr\"odinger operators. The results of this Section are specialized for holomorphic maps that are odd with respect to some ambient involution. We divide this Section in two parts. In Section \ref{subsection::abstract-setting} we develop these abstract results. In Section \ref{subsection::applications-even-potentials} we prove theorems about the Morse index of nonorientable minimal surfaces immersed in Euclidean three-space. 

\subsection{Abstract setting} \label{subsection::abstract-setting}

We consider the general case of a holomorphic map $\phi: \Sigma \to \bS^2$ defined on a compact Riemann surface in the presence of an involution $\tau: \Sigma \to \Sigma$ in such a way that $\phi \circ \tau = - \phi$. In this section, we establish abstract results about the Morse index and nullity of $\phi$ (see Section \ref{subsec:indexonesided}) in the same vein of the work of S. Montiel and A. Ros \cite{MontielRos}. We denote by $\Delta_{\phi}$ the Laplacian with respect to the branched metric $g_{\phi}=\phi^*(g_{\bS^2})$. 

As before, given such a holomorphic map $\phi$, we consider the elliptic operator $L_{\phi} = -\Delta_{\phi} - 2$, where the Laplacian is taken with respect to the branched metric $g_{\phi}$. Recall that the index of this operator is the same as the one of $-\Delta_g -\lvert \nabla^g \phi \lvert_g^2$ for any Riemannian metric $g$ in the conformal class of the Riemann surface $\Sigma$. 

A simple observation that will be used in the following is that $\tau$ is an isometry in the branched metric $g_{\phi}$. In fact, since $\phi \circ \tau = - \phi$, we have that for any tangent vector $v \in T_p \Sigma$
\begin{align*}
    g_{\phi}(d \tau_p v, d \tau_pv) &= g_{\bS^2}(d\phi_{\tau(p)} d\tau_p v, d\phi_{\tau(p)} d\tau_p v) = g_{\bS^2}(-d\phi_p v, -d\phi_p v)\\
    &= g_{\phi}(v,v).
\end{align*}

In the following, given a compact domain $\Omega \subset \Sigma$, we denote by $\Delta^D \lvert_{\Omega}$ and $\Delta^N \lvert_{\Omega}$ the Dirichlet and Neumann Laplacians. In other words, they are the usual Laplacian operator restricted to the spaces $H^1_0(\Omega)$ and $H^1(\Omega)$, respectively. 

\begin{lema}\label{lem:adaptmontielros1}
    Let $\Sigma$ be a compact Riemann surface with an involution $\tau$ and a holomorphic map $\phi: \Sigma \to \bS^2$. Let $\Omega_1, ..., \Omega_m$ be open sets on $\Sigma$ such that together with $\tau(\Omega_1), ..., \tau(\Omega_m)$ give a partition of $\Sigma$ by open sets with piecewise smooth boundary. Then we have 

    \begin{align}
    \begin{split}\label{eq:adaptmontielros1}
        &\text{index}_{\mathcal{O}} (-\Delta_{\phi} - k(k+1) ) + \text{nul}_{\mathcal{O}}(-\Delta_{\phi} - k(k+1))  \\
        \leq \sum_{i=1}^{m-1} \text{index}(-\Delta^N \lvert_{\Omega_i} &- k(k+1)) + \text{index}(-\Delta^N \lvert_{\Omega_m} - k(k+1)) + \text{nul}(-\Delta^N \lvert_{\Omega_m} - k(k+1)).    
    \end{split}        
    \end{align}
\end{lema}

\begin{remark}
\normalfont
    In the orientable version of this lemma, which is proven in \cite{MontielRos}, the bound in the right hand side of \eqref{eq:adaptmontielros1} contains the sum over all domains of the partition, not only a half, and the estimate is for the total index and nullity.
\end{remark}

\begin{proof}
    Define the following sets for $i = 1, ..., m-1$
\begin{align*}
\begin{split}
    V ={}& \text{span} \{ \text{eigenfunctions of } -\Delta_{\phi} \text{ with eigenvalue } \lambda \leq k(k+1) \\
        &\text{ that are odd with respect to } \tau \}.
\end{split}\\
\begin{split}
    V_i ={}& \text{span} \{ \text{eigenfunctions of } -\Delta_{\phi} \text{ on } \Omega_i, \\
        &\text{ with Neumann boundary condition and eigenvalue } \lambda < k(k+1) \}.
\end{split}\\
\begin{split}
    V_m = {}& \text{span} \{ \text{eigenfunctions of } -\Delta_{\phi} \text{ on } \Omega_m, \\
        &\text{ with Neumann boundary condition and eigenvalue } \lambda \leq k(k+1) \}.
\end{split}
\end{align*}
We can regard the sets $V_i$ as subsets of $L^2(\Sigma,g)$ extending to zero outside $\Omega_i$. Observe that:
\begin{equation*}\label{dimensions}
    \text{dim} V \leq  \sum_{i=1}^m \text{dim} V_i.
\end{equation*}
Otherwise, if this is not true, there would be a nonzero $f \in V \cap \bigoplus_{i=1}^m V_i^{\perp}$. In particular,
since $f$ is orthogonal to $\bigoplus_{i=1}^m V_i$, by the variational characterization of eigenvalues, we have that for any $i = 1,..., m$.
\begin{equation*}
    \int_{\Omega_i} \abs{\nabla f}^2 dv_{g_{\phi}} \geq k(k+1) \int_{\Omega_i} f^2 dv_{g_{\phi}}.
\end{equation*}
Now, since $f$ is odd with respect to $\tau$, which is an isometry in the branched metric $g_{\phi}$, and the symmetrization by $\tau$ of the sets $\Omega_1,...,\Omega_m$ cover all $\Sigma$, we have:
\begin{align*}
    \int_{\Sigma} \abs{\nabla f}^2 dv_{g_{\phi}} &= \sum_{i = 1}^m (\int_{\Omega_i} \abs{\nabla f}^2 dv_{g_{\phi}} + \int_{\tau(\Omega_i)} \abs{\nabla f}^2dv_{g_{\phi}}) = 2 \sum_{i = 1}^m \int_{\Omega_i} \abs{\nabla f}^2 dv_{g_{\phi}} \\
    &\geq 2 k(k+1) \sum_{i = 1}^m \int_{\Omega_i} f^2 dv_{g_{\phi}} = k(k+1) \int_{\Sigma} f^2 dv_{g_{\phi}}.
\end{align*}
Equality holds in the equation above if and only if $f\big|_{\Omega_i}$ is an eigenfunction for the Neumann problem in $\Omega_i$ with eigenvalue $k(k+1)$, for $i = 1, ..., m$. In particular, since $f\big|_{\Omega_m}$ is orthogonal to all eigenfunctions for the Neumann problem with eigenvalue $k(k+1)$, it follows that $f\big|_{\Omega_m}$ must be zero. 

Finally, the fact that $f \in V$ means that 
\begin{equation*}
    \int_{\Sigma} \abs{\nabla f}^2 dv_{g_{\phi}} \leq k(k+1) \int_{\Sigma} \abs{f}^2 dv_{g_{\phi}}.
\end{equation*}
In particular we have equality in the last equation, meaning that $f\big|_{\Omega_m} = 0$, but the unique continuation principle implies that $f$ is identically zero, which is a contradiction.
\end{proof}

\begin{lema}\label{lem:adaptmontielros2}
    Let $\Sigma$ be a compact Riemann surface with an involution $\tau$ and a holomorphic map $\phi:\Sigma \to \bS^2$. Let $\Omega_1, ..., \Omega_m$ be open sets on $\Sigma$ such that together with $\tau(\Omega_1), ..., \tau(\Omega_m)$ give a partition of $\Sigma$ by open sets with piecewise smooth boundary. Then we have 

    \begin{align*}
    \begin{split}       &\text{index}_{\mathcal{O}} (-\Delta_{\phi} - k(k+1) )  \\
        \geq \sum_{i=1}^{m-1} (\text{index}(-\Delta^D \lvert_{\Omega_i} - k(k+1&)) + \text{nul}(-\Delta^D \lvert_{\Omega_i} - k(k+1))) + \text{index}(-\Delta^D \lvert_{\Omega_m} - k(k+1)).     
    \end{split}
    \end{align*}
\end{lema}

\begin{proof}

 Define as usual, for $i=1,...,m-1$.
\begin{align*}
\begin{split}
    V ={}& \text{span} \{ \text{eigenfunctions of } -\Delta_{\phi} \text{ with eigenvalue } \lambda < k(k+1) \\
        &\text{ that are odd with respect to } \tau \}.
\end{split}\\
\begin{split}
    V_i ={}& \text{span} \{ \text{eigenfunctions of } -\Delta_{\phi} \text{ on } \Omega_i, \\
        &\text{ with Dirichlet boundary condition and eigenvalue } \lambda \leq k(k+1) \}.
\end{split}\\
\begin{split}
    V_m = {}& \text{span} \{ \text{eigenfunctions of } -\Delta_{\phi} \text{ on } \Omega_m, \\
        &\text{ with Dirichlet boundary condition and eigenvalue } \lambda < k(k+1) \}.
\end{split}
\end{align*}
Note that the spaces $V_i$ for $i=1,...,m$ can be regarded as subspaces of $W^{1,2}(\Sigma)$ by extension to zero. We want to prove that
\begin{equation*}
    \text{dim}V \geq \text{dim} \bigoplus_{i=1}^m V_i.
\end{equation*}
Proceeding again by contradiction, suppose there is $f_0 \in \bigoplus_{i=1}^m V_i \cap V^{\perp}$. By definition, $f_0$ would be zero on the sets $\tau(\Omega_i)$. We can define an odd version $f$ of this function just as taking the value $-f_0\circ \tau$ on the sets $\tau(\Omega_i)$, and being equal to $f_0$ on the sets $\Omega_i$. 

Observe that $f$ also belongs to $V^{\perp}$. In fact, for any odd eigenfunction $g$, we have
\begin{align*}
    \int_{\Sigma} f g dv_{g_{\phi}}&= \sum_{i=1}^m (\int_{\Omega_i} f_0g dv_{g_{\phi}} + \int_{\tau(\Omega_i)} (-f_0\circ \tau) g dv_{g_{\phi}}) \\
    &=\sum_{i=1}^m (\int_{\Omega_i} f_0 g dv_{g_{\phi}} - \int_{\Omega_i} f_0 (g\circ \tau) dv_{g_{\phi}}) \\
    &= 2 \int_{\Sigma} f_0 g dv_{g_{\phi}}\quad \\
    &= 0.
\end{align*}
Where we have used that $\tau$ is an isometry of $g_{\phi}$. Then, we have that $f$ is an odd function which is orthogonal to all odd eigenfunctions with eigenvalue less than $k(k+1)$. By the variational characterization of the eigenvalues, we have
\begin{equation*}\label{geq2}
    \int_{\Sigma} \abs{\nabla f}^2 dv_{g_{\phi}} \geq k(k+1) \int_{\Sigma} f^2 dv_{g_{\phi}}.
\end{equation*}
Now, since $f_0 \in \bigoplus_{i=1}^m V_i$, we have that 
\begin{equation*}
    \int_{\Omega_i} \abs{\nabla f}^2 dv_{g_{\phi}} \leq k(k+1) \int_{\Omega_i} f^2 dv_{g_{\phi}}.
\end{equation*}
With equality if and only if $f \lvert_{\Omega_i}$ is a linear combination of Dirichlet eigenfunctions with eigenvalue $k(k+1)$ on $\Omega_i$ with eigenvalue $k(k+1)$. In particular, equality for $i = m$ would imply that $f\lvert_{\Omega_m}$ is identically zero. But, dividing into all $\Omega_i$ and $\tau(\Omega_i)$, and using that $f$ is odd, we have that
\begin{equation*}
    \int_{\Sigma} \abs{\nabla f}^2 dv_{g_{\phi}} \leq k(k+1) \int_{\Sigma} f^2 dv_{g_{\phi}}.
\end{equation*}
But this implies that we have equality above, then by the last observations $f \lvert_{\Omega_m}$ is identically zero, but this means that $f$ is identically zero by the unique continuation property. Then the lemma follows by contradiction.
    
\end{proof}

\subsection{Application about the Morse index of nonorientable surfaces in Euclidean three-space} \label{subsection::applications-even-potentials}

We derive a tool to compute the Morse index of complete nonorientable minimal surfaces based on information about their Gauss map. The result should be compared to Corollary 15 in \cite{MontielRos}, from where it was adapted.

\begin{thm}\label{thm:indexofnonorientable}
    Suppose we have a holomorphic map of degree $d$ defined on a compact Riemann surface $\phi: \Sigma \to \bS^2$ together with an involution $\tau: \Sigma \to \Sigma$ such that $\phi \circ \tau = - \phi$. If all branch values of $\phi$ are contained in some equator of $\bS^2$, then the odd Morse index of $\phi$ is $d-1$ and the odd nullity of $\phi$ is $3$.
\end{thm}
Theorem \ref{thm:indexofnonorientable} is an important ingredient of the proof of Theorems \ref{thm-Meeks-indice}, \ref{thm-index-lopez} and \ref{thm-index-oliveira}.

\begin{proof}[Proof of Theorem \ref{thm:indexofnonorientable}]

    Since $\phi$ is odd with respect to $\tau$, note that 
\begin{equation*}
    \phi^{-1}(\bS^2-C) = \bigcup_{i=1}^{d} (\Omega_i \cup \tau(\Omega_i)),
\end{equation*}

\noindent where every $\Omega_i$ and $\tau(\Omega_i)$ is sent by $\phi$ isometrically, with respect to the metric $g$, onto a half sphere $D$. Thus, the Dirichlet and Neumann eigenvalues on these domains are equal to those of a half sphere. Recall that

\begin{itemize}
    \item The Neumann eigenvalues on a half sphere are $k(k+1)$ with multiplicities $k+1$ and $k=0,1,2,...$. 
    \item The Dirichlet eigenvalues on a half sphere are $k(k+1)$ with multiplicity $k$ and $k=1,2,3,...$.
\end{itemize}
Combining Lemmas \ref{lem:adaptmontielros1} and \ref{lem:adaptmontielros2}, we have:
\begin{align*}
    \text{index}_{\mathcal{O}}(-\Delta_{\phi} - 2) + \text{nul}_{\mathcal{O}}(-\Delta_{\phi} - 2) &\leq d+2,\\
    \text{index}_{\mathcal{O}}(-\Delta_{\phi} - 2) &\geq d-1.
\end{align*}
Moreover, $\text{nul}_{\mathcal{O}}(-\Delta_{\phi} - 2) \geq 3$ since the coordinate functions of $\phi$ belong to this eigenspace. Then we obtain $\text{index}_{\mathcal{O}} (-\Delta_{\phi} - 2) = d-1$ and $\text{nul}_{\mathcal{O}}(-\Delta_{\phi} - 2) = 3$, which proves the result.
\end{proof}
Therefore, as a consequence of Theorems \ref{thm:fischercolbrienaoorientavel} and \ref{thm:indexofnonorientable}, we have a tool to compute Morse index of nonorientable minimal surfaces
\begin{cor}\label{cor:indexofnonorientable}
    If the extended Gauss map of the double cover of a complete nonorientable minimal immersion has degree $d$ and all its branch values are contained in some equator $C \subset \bS^2$, then the nonorientable minimal immersion has Morse index $d-1$.
\end{cor}

\section{The Meeks minimal M\"obius band and the Oliveira family of minimal M\"obius bands}\label{section::The Meeks minimal M\"obius band and the Oliveira family of minimal M\"obius bands}

We compute the Morse index of the Meeks minimal M\"obius band and of each element of the Oliveira family of minimal M\"obius bands. These computations are based on the study of the Gauss map of these surfaces, and we rely on the application of Theorem \ref{thm:indexofnonorientable} to compute their index in terms of their total curvature. 

\begin{proof}[Proof of Theorem \ref{thm-Meeks-indice}]

Let us describe a parametrization of the Meeks minimal M\"obius band. We consider its oriented double cover $M=\bC\setminus\{0\}$ with its Weierstrass data  
\begin{equation*}
     \cG(z)=z^{2}\frac{z+1}{z-1}, \text{ and }\quad \eta=i\frac{(z-1)^2}{z^{4}}dz.
\end{equation*}

Note that $\cG$ extends holomorphically as a map from the complex sphere $\overline{\mathbb{C}}$ to itself. By the Riemann-Hurwitz Theorem \ref{RiemannHurwitz}, $\cG$ has $4$ branch points, taking into account their multiplicities. We now explicitly identify these branch points and their branching orders. 

In the chart $\mathbb{C}$, the derivative of the Gauss map $\cG$ is given by
\begin{align*}
    \cG'(z) &= z\frac{(3z+2)(z-1) - (z^2 + z)}{(z-1)^2}= 2z \frac{z^2 - z - 1}{(z-1)^2}.
\end{align*}

Thus, $\cG$ has the origin $0$ as a branch point with branch order $1$. Moreover, there are two real numbers $\zeta_1, \zeta_2$, which are zeros of the polynomial in the numerator above, which are also branch points of $\cG$ with branch order $1$. A similar computation taking the chart at infinity shows that the point $\infty$ is a branch point with branch order $1$. Since the Gauss map $\cG$ preserves the real line, this proves that all branch values of the Gauss map of the Meeks Möbius band belong to the equator determined by the real line after mapping to the sphere by the stereographic projection. 
\end{proof} 

We now study the Oliveira family of minimal M\"obius bands. The computations are similar, and in fact generalize the previous one. 

\begin{proof}[Proof of Theorem \ref{thm-index-oliveira}]
 Let us describe a parametrization for each element of the Oliveira family of minimal M\"obius bands. This family is indexed by their total curvature, which are listed as $2 m \pi$ for some odd integer $m\geq 5$. We consider the parametrization by its oriented double cover $M=\bC\setminus\{0\}$ with Weierstrass data
\begin{equation*}
     \cG_m(z)=z^{m-1}\frac{z+1}{z-1}, \text{ and }\quad \eta_m=i\frac{(z-1)^2}{z^{m+1}}dz.
\end{equation*}
By the Riemann-Hurwitz Theorem \ref{RiemannHurwitz}, the number of branch points of $\mathcal{G}_m$ counting multiplicities is $2m-2$. In the parametrization outside $\infty$, the derivative of the Gauss map is given by:
\begin{align*}
    \cG_m'(z) &= z^{m-2}\frac{(mz+(m-1))(z-1) - (z^2 + z)}{(z-1)^2}= z^{m-2} \frac{(m-1)z^2 - 2z - (m-1)}{(z-1)^2}.
\end{align*}
Thus, this map has $0$ as a branch point with branch order $m-2$ and two real numbers $\zeta_1, \zeta_2$, which are zeros of the polynomial in the numerator above, both with branch order $1$. A similar computation taking the chart at infinity shows that the point $\infty$ is a branch point with branch order $m-2$. Now, we observe that the image of these points under the map $\cG_m$ are $0$, $\infty$ and two real numbers $\cG_m(\zeta_1)$, $\cG_m(\zeta_2)$, proving that all branch values of $\cG_m$ lie in a common equator. 

Then, using Theorem \ref{thm:indexofnonorientable} we conclude that the Morse index of the Oliveira's Möbius bands with total curvature $2 \pi m$ equals $m-1$. 
\end{proof}

Notice that the Oliveira family of minimal Möbius bands of degree $m\geq 5$, and the Meeks Möbius band $m=3$, have a branch point at the origin $z=0$ of branch order $m-2$. We present an alternative way to prove that such family have branch values in a single equator as a consequence of the following result
\begin{prop}\label{criteriaequator}
    Let $\tau:\bS^2\to \bS^2$ be an anti-holomorphic involution on the sphere and consider a holomorphic map $\cG:\bS^2\to \overline{\bC}$ of degree $m$ such that $\cG\circ \tau=-1/\overline{\cG}$. If $\cG$ has a branch point $p\in \bS^2$ with branch order $r_p(\cG)\geq m-2$, then all the branch values of $\cG$ lie on a single equator.
\end{prop}

\begin{proof}
    Indeed, if there exists such a point $p\in \bS^2$ with $r_p(\cG)\geq m-2$, then $r_{\tau(p)}(\cG)\geq m-2$ by Proposition \ref{BranchPointsComeInPairs}. Therefore by the Riemann-Hurwitz Theorem \ref{RiemannHurwitz} we have
    \begin{equation*}
        2=2m-\left(r_p(\cG)+r_{\tau(p)}(\cG)+\sum_{q\notin \{p, \tau(p) \}} r_q(\cG)\right)\leq 2m-(2m-4)-\sum_{q\notin \{p, \tau(p) \}} r_q(\cG),
    \end{equation*}
    which implies that
    \begin{equation*}
       0\leq  \sum_{q\notin \{p, \tau(p)\} }r_q(\cG)\leq 2.
    \end{equation*}
    Since branch points of $\cG$ come in pairs of points related by the anti-holomorphic involution $\tau$ by Proposition \ref{BranchPointsComeInPairs}, it follows that either $p$ and $\tau(p)$ are the only branch points of $\cG$ or there exists only one additional pair of branch points $\tilde{p}$ and $\tau(\tilde{p})$ of $\cG$. Therefore, the meromorphic function $\cG$ has at most two pairs of branch points related by the anti-holomorphic involution, which are mapped into at most two antipodal pairs of points. It then follows that the branch values of $\cG$ lie on a single equator.
\end{proof}
As a consequence of Theorem \ref{thm:indexofnonorientable} and Proposition \ref{criteriaequator} the following application is immediate
\begin{cor}\label{cor:nulmeeks}
    The extended Gauss map of the Meeks minimal Möbius band and the Oliveira family of minimal Möbius bands has odd nullity three. 
\end{cor} 
\begin{remark}
\normalfont
     Let us give an alternative proof of Corollary \ref{cor:nulmeeks} in the case of Meeks Möbius band, which illustrates a combination of ideas that could be adapted to other contexts. Let $\phi:\mathbb{S}^2\to \mathbb{S}^2$ be the extended Gauss map associated with the double cover of the Meeks minimal M\"obius band. Assume by contradiction that $\dim Nul_{\text{odd}}(\phi)>3$. Then there exists $u\in N(\phi)\setminus L(\phi)$ such that $u\circ \tau =-u$. Since the holomorphic map $\phi$ has degree $3$ and four branch points with branch order one by the proof of Theorem \ref{thm-Meeks-indice}, the Montiel-Ros isomorphism constructs a branched complete minimal immersion $X_u:\mathbb{S}^2\setminus\{p_{i_1}, ..., p_{i_l}\}\to \mathbb{R}^3$, as described in Remark \ref{JorgeMeeksMultiplicitiesvsBranchOrders}, which has finite total curvature $12\pi$ and at most four ends $\{ p_{i_1},...,p_{i_l} \}$ located at a subset of the branch points of $\phi$. 

Moreover since the branch orders of $\phi$ are one, by Remark \ref{JorgeMeeksMultiplicitiesvsBranchOrders} the numbers $J_{i_s}$ associated to each end of $X_u$ satisfy $J_i \leq 2$.

By the generalized Jorge-Meeks formula in \cite{JorgeMeeksBranched}, given a complete branched minimal immersion $Y:\Sigma\setminus\{r_1,\ldots r_k\}\to \mathbb{R}^3$ with extended Gauss map $\cG:\Sigma\to \overline{\mathbb{C}}$ it holds
\begin{equation*}
     2\deg(\cG)=-\mathcal{X}(\Sigma)+\sum_{j=1}^k J_j-\sum_{j=1}^m K_j,
\end{equation*}
where $K_j$ are the orders of the branch points of the immersion $Y$. For the specific case of the branched complete minimal immersion $X_u:\mathbb{S}^2\setminus\{p_{i_1},...,p_{i_l}\}\to \mathbb{R}^3$ we have
\begin{equation*}
    6 = -2 + \sum_{s = 1}^{l} J_{i_s} - \sum_{j=1}^m K_j \leq -2 + 8 - \sum_{j=1}^m K_j,
\end{equation*}
which implies that the sum of the branch orders of the immersion is equal to zero, hence $X_u$ is a complete minimal immersion without branch points. Moreover, equality in the equation above implies that $l = 4$ and $J_i=2$, which means that $X_u: \bS^2 \setminus \{p_1,p_2,p_3,p_4\}\to \bR^3$ has four ends.

By Theorem \ref{nonorientablemontielros}, and because $u\circ \tau=-u$, the minimal immersion $X_u$ is the oriented double cover of a nonorientable complete minimal immersion $X'_u:\mathbb{RP}^2\setminus\{q_1,q_2\}\to \mathbb{R}^3$ of finite total curvature $6\pi$. But this contradicts the Meeks uniqueness theorem \cite{Meeks}. In conclusion \begin{equation*}
        \dim Nul_{\text{odd}}(\phi)=3,
    \end{equation*}
holds for the Meeks minimal M\"obius band.
\end{remark}

\section{The López minimal Klein bottle}\label{section::The López minimal Klein bottle}

This Section is divided into four parts. In Section \ref{sec:deformationfamily} we describe a family of meromorphic functions that deform the Gauss map of the orientable double cover of the López minimal Klein bottle to a meromorphic map whose branch values lie on an equator of the complex sphere. In Section \ref{sec:localizationbranchs} we localize the branch points of each map of the constructed family of meromorphic functions. In Section \ref{sec:nullityofdeformation} we study the nullity of the Schr\"odinger operator associated to each of these meromorphic functions. Section \ref{sec:proofB} contains the proof of Theorem \ref{thm-index-lopez}. Finally, Section \ref{indexthreecorollary} contains a proof of Corollary \ref{coro-index-8pi}

We start describing a parametrization of the López minimal Klein bottle, introduced in \cite{LopezConstruction}, and its basic properties. For this, consider:
\begin{equation}
    q(z):=z\left(z-\frac{1}{r}\right)(z+r)
\end{equation}
\noindent and let
\begin{equation}
    \overline{M}_r :=\left\lbrace(z,w)\in\overline{\mathbb{C}}^2: w^2=q(z)\right\rbrace.
\end{equation}
 The compact Riemann surface $\overline{M}$ has genus one as can be verified using the Riemann-Hurwitz Theorem \ref{RiemannHurwitz}. We denote $M_r := \overline{M}_r \backslash \{(0,0), (\infty,\infty)\}$.
There exists a unique real number $r >2$ such that the Weierstrass data on $M_r$
\begin{align*}
    \mathcal{G}(z,w) &:=\sqrt{r}\frac{(z+1)w}{(z-1)(z+r)}, \text{ and }\quad 
    \eta(z,w):=i\frac{(z-1)^2(z+r)}{z^2w}dz
\end{align*}
is free of real periods and compatible in the sense of \cite[Proposition 1]{Meeks} with the anti-holomorphic involution without fixed points of $M_r$ given by
\begin{equation}\label{tauinvolution}
    \tau(z,w) :=\left(-\frac{1}{\overline{z}},\frac{\overline{w}}{\overline{z}^2}\right).
\end{equation} 
The minimal immersion of the nonorientable quotient $M'_r=M_r/\langle\tau\rangle$ into $\mathbb{R}^3$ constructed with this data defined a punctured Klein bottle with total curvature $8 \pi$. This is the López minimal Klein bottle constructed in \cite{LopezConstruction}. It is the unique complete nonorientable minimal surface with total curvature $8\pi$ immersed in Euclidean three-space, see \cite{LopezUniqueness}. Note that $M_r$ is the oriented double cover of this surface. From now on, we drop the subscript $r$ from $M$ and $\overline{M}$ and assume that $r$ stands for the conformal parameter of the López minimal Klein bottle. Useful estimates for the parameter $r$, that will be used later, are presented in Appendix \ref{section::Explicit estimates}, Lemma \ref{rLopezNumerical}. 

We highlight two other anti-holomorphic involutions of $\overline{M}$:
\begin{equation}\label{eq:simetriasA}
    A_1(z,w) := (\overline{z},-\overline{w}), \quad \text{and } \quad A_2(z,w) := (\overline{z}, \overline{w}).
\end{equation}

It was shown by López \cite{LopezConstruction} that $A_1$ and $A_2$ are intrinsic isometries of the minimal oriented double cover of the López minimal Klein bottle with the pullback metric of its immersion into $\mathbb{R}^3$ which have extrinsic counterparts given by rotation in Euclidean three-space with respect to a pair of perpendicular lines that are contained in the surface. 

Let us conclude the presentation of the López minimal Klein bottle with a description of an atlas for the Riemann surface structure of its oriented double cover $\overline{M}$ and the computation of the divisors of some meromorphic functions and one-forms defined on $\overline{M}$. 

\begin{lema}[{\textit{Cf}. \cite[Example 1.7]{RiemannSurfacesGirondo}}] \label{holomorphic-charts}
    Let $(a_1,a_2,a_3) = (-r,0,\frac{1}{r})$. The following holomorphic charts define an atlas for $\overline{M}$.

    \begin{itemize}
        \item Chart at $(x_0,y_0)$ with $x_0\neq a_i$
        \begin{equation}\label{chartgenericpoints}
            \phi^{-1}(\xi)=\left(\xi+x_0, \sqrt[]{\Pi_i(\xi+x_0-a_i)}\right)\quad  \quad 0<|\xi|<\epsilon,
        \end{equation}
        \item Chart at $(a_i,0)$
        \begin{equation}
            \phi_i^{-1}(\xi)=\left(\xi^2+a_i,\xi\sqrt[]{\Pi_{j\neq i}(\xi^2+a_i-a_j)}\right)\quad  \quad 0<|\xi|<\epsilon,
        \end{equation}
        \item Chart at $(\infty,\infty)$
        \begin{equation}
            \psi^{-1}(\xi)=\left(\frac{1}{\xi^2},\frac{1}{\xi^3}\sqrt[]{\Pi_i(1-a_i\xi^2)}\right)\quad  \quad 0<|\xi|<\epsilon.
        \end{equation}
    \end{itemize}
\end{lema}

 We now introduce the notation
    \begin{equation} \label{notation-P}
    \begin{cases}
         P_0: =(0,0), \text{ } P_{\frac{1}{r}}:=(\frac{1}{r}, 0),  \text{ and } P_{-r}:=(-r,0),\\
        P_{\infty} := (\infty, \infty),\\
        P_{- t}^+ := \left(-t, \sqrt{q(-t)}\right),  \text{ and } P_{- t}^- := \left(-t,- \sqrt{q(-t)}\right), \text{ for every } t \in (0,1],\\
        P_{\frac{1}{t}}^+ := \left(\frac{1}{t},  \sqrt{q(\frac{1}{t})}\right),  \text{ and } P^-_{\frac{1}{t}} = \left(\frac{1}{t}, - \sqrt{q(\frac{1}{t})}\right), \text{ for every } t \in (0,1].
    \end{cases}
    \end{equation} 
\begin{lema} \label{some-divisors}
    The following meromorphic functions and one-forms have explicitly computable divisors:
    \begin{equation}
        \begin{cases}
            \operatorname{div}(w) = P_0 + P_{-r}+P_{\frac{1}{r}} - 3 P_\infty, \\
            \operatorname{div}(z) = 2P_0-2P_\infty, \\
            \operatorname{div}(z+r) = 2P_{-r} - 2 P_\infty, \\
            \operatorname{div}(z+t) = P_{-t}^+ + P_{-t}^- -2P_\infty \text{ for every } t \in (0,1], \\
            \operatorname{div}(z-\frac{1}{t}) = P_{\frac{1}{t}}^+ + P_{\frac{1}{t}}^- -2 P_\infty \text{ for every } t \in (0,1], \\
            \operatorname{div}(z-\frac{1}{r}) = 2 P_{\frac{1}{r}} - 2 P_\infty, \\
            \operatorname{div}(dz) = P_0 + P_{-r}+P_{\frac{1}{r}} - 3 P_\infty.
        \end{cases}
    \end{equation}
\end{lema}

\begin{proof}
        The result follows directly from the expression of the corresponding meromorphic functions and one-forms in the holomorphic charts of $\overline{M}$ described in Lemma \ref{holomorphic-charts}. We use that $ \{-r,0,\frac{1}{r}\} \cap [-1,0) = \varnothing$ and $\{-r,0,\frac{1}{r}\} \cap [1,\infty) = \varnothing$, since $r>2$ by López work \cite{LopezConstruction} or using the more precise estimate of Lemma \ref{rLopezNumerical}.

\end{proof}

We highlight that the meromorphic one-form $\frac{dz}{w}$ defined on $\overline{M}$ is free of zeros and poles, because the divisors of $w$ and $dz$ are equal, in fact:
\begin{gather}\label{Lopezcanonicaldivisortrivial}
    \operatorname{div}\left(\frac{dz}{w}\right)=0.
\end{gather}

\subsection{The deformation family and its properties}\label{sec:deformationfamily}

We now define a family of meromorphic functions $\{\mathcal{G}_t\}_{t \in [0,1]}$ from $\overline{M}$ to the complex sphere $\overline{\bC}$ given by
\begin{equation}\label{eq:deformapadegauss}
    \mathcal{G}_t(z,w):=\sqrt{r} \frac{(z+t) w}{(tz-1)(z+r)}, \text{ for $t \in (0,1]$, \quad and } \quad \mathcal{G}_0(z,w) := -\sqrt{r} \frac{z w}{(z+r)}.
\end{equation}
The map $\mathcal{G}_1$ is equal to the Gauss map $\mathcal{G}$ of the López minimal Klein bottle. Let us study the zeros and poles of the meromorphic functions $\mathcal{G}_t$. Recall the notation introduced in \eqref{notation-P}. 

\begin{lema}\label{zeroespolos}
     
 The divisors of the meromorphic functions in the family $\{\mathcal{G}_t\}_{t \in [0,1]}$ are
 \begin{equation*}
     \begin{cases}
         \operatorname{div}(\mathcal{G}_t) = P_0 + P_{-t}^+ + P_{-t}^- + P_{\frac{1}{r}} - P_{\infty} - P_{\frac{1}{t}}^+ - P_{\frac{1}{t}}^- - P_{-r} \text{ for every }t \in (0,1], \\
         \operatorname{div}(\mathcal{G}_0) = 3 P_0 + P_{\frac{1}{r}} - 3 P_{\infty} - P_{-r}.
     \end{cases}
 \end{equation*}
\end{lema}

\begin{proof}
    The result follows directly from Lemma \ref{some-divisors}.
\end{proof}
We now highlight important properties of the family $\{\mathcal{G}_t\}_{t \in [0,1]}$ of meromorphic functions. First, observe that each map of the family is odd with respect to the anti-holomorphic involution $\tau$ when considered having image on the sphere (recall that in the identification of the extended complex plane with the sphere $\overline{\bC}\simeq \bS^2$ by the stereographic projection, the antipodal map sends $x \in \overline{\bC}$ to $-\frac{1}{\bar x} \in \overline{\bC}$).

\begin{lema} \label{div-Gt}
    The meromorphic function $\mathcal{G}_t$ satisfies $\mathcal{G}_t \circ \tau =- \frac{1}{\overline{\mathcal{G}_t}}$, for every $t \in [0,1]$.
\end{lema}

\begin{proof}
    If $t \in (0,1]$ and $(z,w) \not\in \{(0,0),(\infty,\infty)\}$ then
    \begin{align*}
        \mathcal{G}_t \circ \tau (z,w) &= \mathcal{G}_t  \left(-\frac{1}{\bar z},\frac{\bar w}{\bar z^2}\right) = \sqrt{r} \frac{((-\frac{1}{\bar z})+t) \frac{\bar w}{\bar z^2}}{(t(-\frac{1}{\bar z})-1)((-\frac{1}{\bar z})+r)} \\&=\sqrt{r} \frac{(-1+t \bar z) \bar w^2}{\bar z(-t-\bar z)(-1+r \bar z)\bar w} = \sqrt{r} \frac{(-1+t \bar z) \bar z( \bar z - \frac{1}{r})(\bar z +r)}{\bar z(-t-\bar z)(-1+r \bar z)\bar w} \\ &= -\frac{1}{\sqrt{r}} \frac{(-1+t \bar z) (\bar z +r)}{(t+\bar z)\bar w} = - \frac{1}{\overline{\mathcal{G}_t(z,w)}}. 
    \end{align*}

    Moreover, if $(z,w) \not \in \{(0,0),(\infty,\infty)\}$ then
    \begin{align*}
        \mathcal{G}_0 \circ \tau (z,w) &= \mathcal{G}_0  \left(-\frac{1}{\bar z},\frac{\bar w}{\bar z^2}\right) =- \sqrt{r} \frac{(-\frac{1}{\bar z}) \frac{\bar w}{\bar z^2}}{((-\frac{1}{\bar z})+r)} \\&=-\sqrt{r} \frac{-\bar w^2}{\bar z^2(-1+r \bar z)\bar w} = \sqrt{r} \frac{\bar z( \bar z - \frac{1}{r})(\bar z +r)}{\bar z^2(-1+r \bar z)\bar w} \\ &= \frac{1}{\sqrt{r}} \frac{(\bar z +r)}{\bar z\bar w} = - \frac{1}{\overline{\mathcal{G}_0(z,w)}}. 
    \end{align*}

    Finally, it follows from Lemma \ref{zeroespolos} that $\mathcal{G}_t(0,0) = 0$ and $\mathcal{G}_t(\infty,\infty) = \infty$, for every $t \in [0,1]$. This completes the proof. 
\end{proof}

We now study the continuity of the family of meromorphic functions $\{\mathcal{G}_t\}_{t \in [0,1]}$.

\begin{lema}\label{deformationpath}
    The family $\{\mathcal{G}_t\}_{t \in [0,1]}$ is continuous in the $C^1$ topology.
\end{lema}

\begin{proof}
    
We first prove that $\mathcal{G}_t \to \mathcal{G}_0$ smoothly when $t \to 0^+$. The case $t \to t_0$ for $t_0 \in (0,1]$ is analogous.

Consider the finite family of holomorphic charts of $\overline{M}$ given by

\begin{itemize}
    \item $(U_1,x_1) = \left(B(0,R_1),\ \xi \mapsto \Big( \xi^2-r,\xi\sqrt{(\xi^2 - r - \frac{1}{r})(\xi^2-r)}\,\Big)\right)$,
    \item $(U_2,x_2) = \left(B(0,R_2),\ \xi \mapsto \Big( \frac{1}{\xi^2},\frac{1}{\xi^3}\sqrt{(1-\frac{1}{r}\xi^2)(1+r\xi^2)}\,\Big)\right)$,
    \item $(U_3,x_3) = \left(B(0,R_3),\ \xi \mapsto \Big( \xi^2,\xi\sqrt{(\xi^2+r)(\xi^2-\frac{1}{r})}\Big)\right)$,
    \item $(U_4,x_4) = (B(0,R_4),\ \xi \mapsto \Big( \xi^2+\frac{1}{r},\xi\sqrt{(\xi^2+\frac{1}{r})(\xi^2+\frac{1}{r}+r)}\,\Big)$,
    \item $(U_j,x_j) = \left(B(0,R_j),\ \xi \mapsto \Big( \xi+\theta_j,\sqrt{(\xi+\theta_j)(\xi+\theta_j+r)(\xi+\theta_j-\frac{1}{r})}\,\Big)\right)$,\ $j=5,\dots,N$, 
\end{itemize}

\noindent where we assume that $N \ge 5$ is an integer, the real numbers $R_k$ satisfy $R_k^2 \le \min\{r-1,\frac{1}{2}\}$, for every $k=1,\dots,N$, and $R_j < \frac{R_1}{2}$ for every $j=5,\dots,N$. We assume that the points $\theta_j$ satisfy $\theta_j \not \in \{0,-r,\frac{1}{r},\infty\}$ and $|\theta_j + r| \ge \frac{R_1}{2}$, for every $j=5,\dots,N$. Finally, we assume that the collection of open sets $U_j$ covers $\overline{M}$. We also consider charts of $\overline{\mathbb{C}}$: $(V_1,y_1) = (\mathbb{C},\ \xi \mapsto \frac{1}{\xi})$, and $(V_2,y_2) = (\mathbb{C},\ \xi \mapsto \xi)$.

    We now assume that $0 \le t<1$ satisfies $t|\theta_j|< \frac{1}{2}$, for every $j=5,\dots,N$. Then
\begin{equation*}
\mathcal{G}_t(U_j) \subset y_2(V_2),\ \text{for}\ j \ge 5,\quad \mathcal{G}_t(U_1) \subset y_1(V_1),\quad \mathcal{G}_t(U_2) \subset y_1(V_1),\quad \mathcal{G}_t(U_3) \subset y_2(V_2), \text{ and  } \mathcal{G}_t(U_4) \subset y_2(V_2).
\end{equation*}
 
In what follows, we let $P$ denote a holomorphic function defined on an open ball that never vanishes. We compute $$y_1^{-1} \circ \mathcal{G}_t|_{U_1} \circ x_1 : \xi \in B(0,R_1) \mapsto  (\mathcal{G}_t(\xi^2-r,\xi P(\xi)))^{-1} = \frac{1}{\sqrt{r}}\frac{[t(\xi^2-r)-1]\xi}{(\xi^2-r+t)P(\xi)} \in \mathbb{C}.$$ 
    Since $R_1^2 < r-1$, using equation \eqref{eq:deformapadegauss}, it is clear that $y_1^{-1} \circ \mathcal{G}_t|_{U_1} \circ x_1$ converges smoothly on compact subsets of $B(0,R_1)$ to $y_1^{-1} \circ \mathcal{G}_0|_{U_1} \circ x_1$. Analogously, $y_2^{-1} \circ \mathcal{G}_t|_{U_3} \circ x_3$ converges smoothly on compact subsets of $B(0,R_3)$ to $y_2^{-1} \circ \mathcal{G}_0|_{U_3} \circ x_3$; $y_2^{-1} \circ \mathcal{G}_t|_{U_4} \circ x_4$ converges smoothly on compact subsets of $B(0,R_4)$ to $y_2^{-1} \circ \mathcal{G}_0|_{U_4} \circ x_4$; $y_1^{-1} \circ \mathcal{G}_t|_{U_2} \circ x_2$ converges smoothly on compact subsets of $B(0,R_2)$ to $y_1^{-1} \circ \mathcal{G}_0|_{U_2} \circ x_2$; $y_2^{-1} \circ \mathcal{G}_t|_{U_j} \circ x_j$ converges smoothly on compact subsets of $B(0,R_j)$ to $y_2^{-1} \circ \mathcal{G}_0|_{U_j} \circ x_j$, for every $j = 5,\dots,N$. This proves that $\mathcal{G}_t \to \mathcal{G}_0$ smoothly when $t \to 0^+$, as desired.

\end{proof}
As a remark, it follows from Lemma \ref{deformationpath} that all maps $\mathcal{G}_t$ have degree four. This follows since $C^1$ continuity preserves the degree, and $\mathcal{G}_1$ has degree four, since it is the Gauss map of the minimal orientable double cover of the López Klein bottle, with total curvature $16 \pi$.

We highlight that the map $\mathcal{G}_0$ has a special configuration of branch values. 

\begin{lema}\label{branchvaluesG0}
    The branch values of the map $\mathcal{G}_0$ lie in an equator of the complex sphere $\overline{\mathbb{C}}$.
\end{lema}

\begin{proof}
    A direct computation shows that:
    \begin{equation}
    d\mathcal{G}_0(z,w) =  \sqrt{r} \frac{k_0(z)}{2(z+r)} \frac{dz}{w},
\end{equation}
where $k_0(z) = z(3 + (\frac{2}{r} - 4r)z - 3z^2)$. If we evaluate the quadratic polynomial given by $k_0(z)/z$ at $-r$, a basic algebraic manipulation gives $r^2 + 1$. This means that there is a zero in $\rho_1 \in (-\infty, -r)$. Similarly, evaluating this quadratic polynomial at $1/r$ gives $-(1+r^2)/r^2$, which is negative, therefore, since it is evidently positive at zero, there is another root $\rho_2 \in (0,1/r)$. Since $\rho_1, \rho_2$ do not belong to $\{ -r, 0, 1/r \}$, we have found four branch points with branch order one for $\mathcal{G}_0$, namely, $(\rho_1, \pm i \sqrt{\lvert q(\rho_1) \lvert})$ and $(\rho_2, \pm i \sqrt{\lvert q(\rho_2) \lvert})$. 
    
Now, consider the holomorphic chart around the point $(0,0)$ given by $\phi_0^{-1}(z) = \left(z^2, z \sqrt{(z^2 - \frac{1}{r})(z^2 + r)}\right)$. In this chart:
    \begin{equation*}
        \mathcal{G}_0 \circ \phi_0^{-1}(z) = -\sqrt{r} \frac{z^3 \sqrt{(z^2 - \frac{1}{r})(z^2 + r)}}{z^2 + r}. 
    \end{equation*}

Thus, the point $P_0 = (0,0)$ is a branch point with branch order two for $\mathcal{G}_0$. Finally, the point $P_{\infty}=(\infty,\infty)$ is a branch point with branch order two as an application of Lemma \ref{div-Gt} and Proposition \ref{BranchPointsComeInPairs}. Since by the Riemann-Hurwitz Theorem \ref{RiemannHurwitz}, $\mathcal{G}_0$ has total branching order $8$, it follows by the previous computations that all the branch values of $\mathcal{G}_0$ lie in the equator of the complex sphere $\overline{\bC}$ determined by the purely imaginary line. 
    
\end{proof} 

\subsection{The branch points of the meromorphic functions $\mathcal{G}_t$ for every $t \in (0,1]$}\label{sec:localizationbranchs}

A direct computation of the differential of $\mathcal{G}_t$ gives
\begin{equation}\label{eq:derivativeofgaussmap}
    d\mathcal{G}_t(z,w) = \sqrt{r} \frac{k_t(z)}{2(tz - 1)^2 (z+r)} \frac{dz}{w},
\end{equation}
\noindent where 
\begin{equation}\label{eq:kt}
    k_t(z)\coloneq k(z,t) = t + (3 - 2rt + t^2) z + \left(\frac{2}{r} - 4r - 2t + \frac{2t^2}{r}\right)z^2 + (-3 + 2rt - t^2) z^3 + tz^4.
\end{equation}
Let us introduce new real-valued parameters $L_i = L_i(t)$, $i=1,2,3$, defined so that $$k_t(z) = t -L_1 z + L_2 z^2 + L_1 z^3 + t z^4.$$ 
\begin{lema}\label{lem:defkt}
    The roots of the polynomial $k_t$ lie on the real line, for any $t \in [0,1]$. When $t \in (0,1]$, the polynomial $k_t$ can be factored as $k_t(z) = t (z - a_t)(z-b_t)(z+\frac{1}{a_t})(z+\frac{1}{b_t})$.
\end{lema}

\begin{proof}
    When $t \in (0,1]$, $z = 0$ is not a root of $k_t$. After division by $z^2$, introducing $\lambda := \frac{1}{z} - z$, we obtain
    \begin{equation}\label{eq:ktdivididoz2}
        \frac{k_t(z)}{z^2} = t \lambda^2 - L_1 \lambda + L_2 + 2t.
    \end{equation}
    The discriminant $\delta_t$ of the polynomial above satisfies: 
    \begin{equation}\label{discriminant}
        \delta_t=L_1^2 - 4t(L_2 + 2 t) =(3 - 2 r t + t^2)^2 -\frac{8 t (1 - 2 r^2 + t^2)}{r} >0.
    \end{equation}
    
     Thus, the polynomial in $\lambda$ defined in \eqref{eq:ktdivididoz2} has two real roots, meaning that $k_t$ has four real roots, since every real number is assumed twice by the function $\frac{1}{z} - z$. The structure of the factorization comes from the fact that if $z$ is a root of $k_t$, by the expression above, we have that $-\frac{1}{z}$ is also a root of $k_t$.

     Now we write an explicit expression for $a_t$ and $b_t$. Consider the roots of the quadratic polynomial in \eqref{eq:ktdivididoz2} denoted by 
\begin{align}\label{xtyt}
    x_t &= \frac{L_1 + \sqrt{L_1^2 - 4 t(L_2 + 2 t)}}{2 t}, \quad y_t = \frac{L_1 - \sqrt{L_1^2 - 4 t(L_2 + 2t)}}{2 t}.
\end{align}
The roots of $k_t$ are roots of the equations $\frac{1}{z} - z = x_t$ and $\frac{1}{z} - z = y_t$. We choose 
\begin{align}\label{eq:atbt}
    a_t = \frac{-y_t - \sqrt{y_t^2 + 4}}{2},\quad
    b_t = \frac{-x_t - \sqrt{x_t^2 + 4}}{2}. 
\end{align}
\end{proof}
We now describe the branch points of $\mathcal{G}_t$. First, from Lemma \ref{lem:defkt} and equation \eqref{eq:atbt}, we know that $a_t$ and $b_t$ are the two negative zeroes of $k_t$. Moreover, since the function $z \mapsto \frac{1}{z}-z$ is decreasing and $y_t < x_t$, we observe that $a_t > b_t$. Now we give a precise description of the branch points.
\begin{lema}\label{BranchpointsDeformationmap}
      For $t \in (0,1]$, $\mathcal{G}_t$ has $8$ different branch points with branch order one of the form $B_1(t) = (a_t, \sqrt{\abs{q(a_t)}})$, $B_2(t) = (b_t, i \sqrt{\abs{q(b_t)}})$, $B_3(t) = A_1(B_1(t))$, $B_4(t) = A_2(B_2(t))$ and $B_{4+i}(t) = \tau(B_i(t))$ for $i=1,2,3,4$. Moreover, the branch values of $\mathcal{G}_t$ are contained in two fixed orthogonal equators for all $t\in (0,1]$.
\end{lema}
\begin{proof}
    By Lemma \ref{lem:posicoesdosatbt} we know that $q(a_t) > 0 > q(b_t)$ for any $t \in (0,1]$. Thus we have that
    \begin{align*}
        &B_1(t) = (a_t, \sqrt{\abs{q(a_t)}}),\hspace{0.6cm}B_5(t)=\left(-1/a_t,\sqrt{\abs{q(a_t)}}/a_t^2\right) \\ &B_2(t) = (b_t, i \sqrt{\abs{q(b_t)}}),\hspace{0.55cm}  B_6(t)=\left(-1/b_t,-i\sqrt{\abs{q(b_t)}}/b_t^2\right)\\& B_3(t)=(a_t, -\sqrt{\abs{q(a_t)}}),\quad B_7(t)=\left(-1/a_t,-\sqrt{\abs{q(a_t)}}/a_t^2\right)\\&B_4(t)=(b_t, -i \sqrt{\abs{q(b_t)}}),\quad B_8(t)=\left(-1/b_t,i\sqrt{\abs{q(b_t)}}/b_t^2\right),
    \end{align*}
    are branch points of $\mathcal{G}_t$ using equation \eqref{eq:derivativeofgaussmap} and Lemma \ref{lem:defkt}. On the other hand, notice that we can write $B_3(t) = A_1(B_1(t))$, $B_4(t) = A_2(B_2(t))$ and $B_{4+i}(t)=\tau(B_i(t))$ using the symmetries \eqref{tauinvolution} and \eqref{eq:simetriasA}. Since by the Riemann-Hurwitz Theorem \ref{RiemannHurwitz}, $\mathcal{G}_t$ has exactly $8$ branch points counting multiplicities, then we have found all the branch points of $\mathcal{G}_t$. 
\end{proof}

\subsection{On the nullity of the meromorphic functions $\mathcal{G}_t$}\label{sec:nullityofdeformation}
Our aim in this Section is to compute the nullity of the maps $\mathcal{G}_t$ for $t > 0$. For this, we use the theory developed by Montiel and Ros in \cite{MontielRos} to compute the nullity of a meromorphic function as we presented in Section \ref{sec:indexandgaussmap}.

Our first result is about the dimension and characterization of the space $H^{0,2}(\mathcal{G}_t)$, which is a space of meromorphic quadratic differentials satisfying certain conditions on its divisors. Recall equation \eqref{eq:spacedivisors} for the precise definition of this space.

\begin{thm}\label{lema:formadeH0}
 For all $t\in (0,1]$ the space $ H^{0,2}(\mathcal{G}_t)$ has complex dimension $8$ and it is characterized by
\begin{equation*}
    H^{0,2}(\mathcal{G}_t) = \left\lbrace \left( \frac{P(z)}{k_t(z)} + \frac{Q(z)}{k_t(z)} w \right) \left( \frac{dz}{w} \right)^2: P, Q\in \bC[z],\ \text{deg} (P) \leq 4, \text{deg} (Q) \leq 2 \right\rbrace.
\end{equation*}    
\end{thm}

\begin{proof}
 The space $H^{0,2}(\cG_t)$ is isomorphic to $\cL(2K_{\overline{M}}+R(\cG_t))$, by \eqref{isomorphismLspaceH0space}. The divisor $2K_{\overline{M}}+R(\cG_t)$ has degree $8$ since the Gauss map has exactly $8$ different branch points for each $t\in (0,1]$ and the canonical divisor of $\overline{M}$ is trivial, see Lemma \ref{BranchpointsDeformationmap} and \eqref{Lopezcanonicaldivisortrivial}. Therefore we find that
 \begin{align*}
     \dim\left(H^{0,2}(\cG_t)\right)=l(2K_{\overline{M}}+R(\cG_t))
     &=1-\operatorname{genus}(\overline{M})+\deg\left(2K_{\overline{M}}+R(\cG_t)\right)+i(2K_{\overline{M}}+R(\cG_t))\\
     &=1-1+8+0=8.
 \end{align*}
 as a consequence of the Riemann-Roch Theorem \ref{RiemannRoch} and Lemma \ref{highdegreedivisor}. For the characterization of $H^{0,2}(\cG_t)$, consider the space $\cA$ of all the quadratic forms \begin{equation*}
     \sigma=\left(\frac{P}{k_t(z)}+\frac{Q(z)}{k_t(z)}w\right)\left(\frac{dz}{w}\right)^2,
 \end{equation*}
 where $P, Q$ are polynomials of degrees $\deg P\leq 4$, $\deg Q\leq 2$. Notice that the quotient $\frac{P}{k_t}$ can only have poles at the roots of $k_t$, and they have order at most one. On the other hand, the function $w$ has only one pole at $P_{\infty}$ which have order three and the order of the pole for the function $z$ at this point is two, by Lemma \ref{some-divisors}. Hence, we have that $\frac{Q(z)w}{k_t(z)}$ does not have a pole at $P_{\infty}$ for $\text{deg}(Q) \leq 2$. Moreover, this expression can only have poles at the roots of $k_t$, and they have order at most one. 
    Then, since the canonical divisor $K_{\overline{M}}$ is trivial by \eqref{Lopezcanonicaldivisortrivial}, we have that the space $\cA$ is contained in $H^{0,2}(\mathcal{G}_t)$ and since their dimensions are the same, we have the conclusion. 
\end{proof}
Now we introduce some notation to distinguish between this space and the associated space of meromorphic one-forms obtained after division by $d \mathcal{G}_t$.
\begin{defn}\label{spaceoneforms}
\normalfont
Let $\{\sigma_j\}_{j=1}^8$ be a complex basis of $H^{0,2}(\cG_t)$. We define $\cV(\cG_t)$ as the complex eight dimensional vector space generated by the associated one-forms $\eta_j\coloneq\frac{\sigma_j}{d\cG}$.
\end{defn}

\subsubsection{The residue problem}\label{residueproblem} 

In this Section, we study the space:
\begin{equation*}
    \hat{H}(\mathcal{G}_t) = \{ \sigma \in H^{0,2}(\mathcal{G}_t): \operatorname{Res}_{B_i(t)}\frac{\sigma}{d \mathcal{G}_t} = 0,\ 1\leq i\leq 8 \}.
\end{equation*}

\begin{defn}[{Residue Matrix}]
\normalfont
Let $\cB=\{\eta_j(t)=\frac{\sigma_j}{d\cG_t}\}_{j=1}^8$ be a complex basis of the space $\cV(\cG_t)$. We define the \textit{Residue Matrix} $\operatorname{Res}_{\cB}(\cG_t)$ with respect to $\cB$ as the square complex matrix
\begin{gather*}
   \operatorname{Res}_{\cB}(\cG_t)\coloneqq (\text{Res}_{B_i}\eta_j(t))\in \cM_{8\times 8}(\bC).
\end{gather*}
\end{defn}
\begin{lema}\label{simplificationresidueeta}
    For all $t\in(0,1]$ the induced one-form $\eta(t)=\frac{\sigma}{d\cG_t}\in \cV(\cG_t)$ is
    \begin{equation*}
        \eta(t)=\frac{2(tz-1)(z+t)P(z)}{\cG_tk_t^2}dz+\frac{2(tz-1)^2(z+r)Q(z)}{\sqrt{r}k_t^2}dz,\quad \sigma=\left(\frac{P+Qw}{k_t}\right)\left(\frac{dz}{w}\right)^2.
    \end{equation*}
    Moreover, the residue of $\eta(t)$ at any branch point $B_i(t)=(p_i(t),w(p_i(t)))$, $1\leq i\leq 8$ is
    \begin{equation}\label{residuecalculationoneforms}
        \operatorname{Res}_{B_i(t)}\eta(t)=\frac{2}{\cG_t(B_i(t))}\operatorname{Res}_{p_i(t)}\left(\frac{(tz-1)(z+t)P(z)}{k_t^2(z)} dz\right)+\frac{2}{\sqrt{r}}\operatorname{Res}_{p_i(t)}\left(\frac{(tz-1)^2(z+r)Q(z)}{k_t^2(z)} dz\right).
    \end{equation}
\end{lema}
\begin{proof}
    The first part of the Lemma follows directly by Definition \ref{spaceoneforms} and an application of equations \eqref{eq:deformapadegauss}, \eqref{eq:derivativeofgaussmap}. For the second part, we find the residue of $\eta(t)$ at each branch point $B_i(t)$ using the complex chart $\phi^{-1}(\xi)$ of equation \eqref{chartgenericpoints}, since for all $t\in (0,1]$ the branch points $B_j(t)$ of $\cG_t$ do not intersect the set of points $\{P_0,\ P_{\frac{1}{r}},\ P_{-r}\}$ by Lemma \ref{lem:posicoesdosatbt}. Moreover, the one-form 
    \begin{align*}
        &\left(\phi^{-1}\right)^{*}\left(\frac{(tz-1)(z+t)P(z)}{k_t^2(z)} dz\right )=\left(\frac{a_{-2}}{\xi^2}+\frac{a_{-1}}{\xi}+h(\xi)\right)d\xi,\\
        &a_{-1}=\operatorname{Res}_{p_i(t)}\frac{(tz-1)(z+t)P(z)}{k_t^2(z)},\ h\ \text{holomorphic},
    \end{align*}
    has at most a pole of order two at $\xi=0$, while the Gauss map in the $\phi^{-1}(\xi)$ chart around any branch point $B_i(t)$ satisfies $(G_t\circ \phi^{-1})'(0)=0$. Therefore, locally around any $B_i(t)$ we have
    \begin{equation*}
\frac{1}{\cG_t\circ \phi^{-1}(\xi)}=\frac{1}{\cG_t(B_i(t))}+O(\xi^2)
    \end{equation*}
    and the result follows.
\end{proof}
 To compute the dimension of the space $\hat{H}(\cG_t)$ it is enough to compute the nullity of the Residue Matrix $\operatorname{Res}_\cB(\cG_t)$ as the next proposition shows. We omit its straightforward proof.
 
\begin{prop}\label{isomorphismHchapeuspace}
   Let $\{e_j\}_{j=1}^8$ be the canonical basis of $\bC^8$. For each $t\in(0,1]$, the linear map
    \begin{equation*}
T:\text{Nul}(\operatorname{Res}_{\cB}(\cG_t))\to \hat{H}(G_t),\quad T\left(\sum_{j=1}^8c_je_j\right)=\sum_{j=1}^{8}c_j\sigma_j,
    \end{equation*} 
    is an isomorphism. 
\end{prop}

Now, we introduce a notion of symmetry for quadratic differentials:
\begin{defn}\label{symmetricquadraticdifferentials}
\normalfont
    Given an anti-holomorphic involution $\gamma$ we say that a quadratic differential $\sigma$ is $\gamma$-\textit{symmetric} if $\overline{\gamma^* \sigma} = \sigma$. Conversely, we say that it is $\gamma$-\textit{antisymmetric} if $\overline{\gamma^* \sigma} = - \sigma$. 
\end{defn}

Observe that a subspace of quadratic differentials which are either symmetric or antisymmetric is not a complex subspace, but a real one. The space $H^{0,2}(\cG_t)$ decomposes into eight real subspaces composed by quadratic differentials which are symmetric or antisymmetric under the anti-holomorphic involutions $A_1$, $A_2$, $\tau$ respectively
\begin{equation}\label{decompositionH0}
    H^{0,2}(\cG_t)=H^{0}_{SSS}\oplus H^{0}_{SSA}\oplus H^{0}_{ASA}\oplus  H^{0}_{ASS}\oplus H^{0}_{AAA}\oplus H^{0}_{AAS}\oplus H^{0}_{SAS}\oplus H^{0}_{SAA},
\end{equation}
with the last four subspaces being just $i$ times the first four. More explicitly we can write those subspaces as
\begin{align*}
    \begin{split}
         &H^0_{SSS}=\left\lbrace \frac{c_0z^4-c_1z^3+c_2z^2+c_1z+c_0}{k_t}\left(\frac{dz}{w}\right)^2: c_j\in \bR\right\rbrace ,\quad \dim_{\bR} H^0_{SSS}=3,\\& H^0_{SSA}=\left\lbrace \frac{c_0z^4-c_1z^3-c_1z-c_0}{k_t}\left(\frac{dz}{w}\right)^2: c_j\in \bR\right\rbrace,\quad \dim_{\bR} H^0_{SSA}=2,\\&H^0_{ASA}=\left\lbrace  \frac{c_0z^2+c_1z-c_0}{k_t}w\left(\frac{dz}{w}\right)^2: c_j\in \bR\right\rbrace,\quad \dim_{\bR}H^0_{ASA}=2,\\&H^0_{ASS}=\left\lbrace c_0\frac{z^2+1}{k_t}w\left(\frac{dz}{w}\right)^2: c_0\in \bR\right\rbrace,\quad \dim_{\bR}H^0_{ASS}=1.
    \end{split}
\end{align*}
We prove this in the case of the subspace $H^0_{SSS}$. Suppose that $\sigma \in H^0_{SSS}\subset H^{0,2}(\cG_t)$. There exist $P,Q\in \bC[z]$ with $\deg P\leq 4$ and $\deg Q\leq 2$ such that $\sigma$ has the form presented in Theorem \ref{lema:formadeH0}. By hypothesis, the quadratic form $\sigma$ is symmetric with respect to the anti-holomorphic involution $A_1$ of equation \eqref{eq:simetriasA}. Since $k_t(z)\in \bR[z]$ by equation \eqref{eq:kt}, we have that
\begin{equation*}
    \left(\frac{P(z)}{k_t(z)}+\frac{Q(z)}{k_t(z)}w\right)\left(\frac{dz}{w}\right)^2=\sigma=\overline{A_1^*\sigma}= \left(\frac{\overline{P(\overline{z})}}{\overline{k_t(\overline{z})}}-\frac{\overline{Q(\overline{z})}}{\overline{k_t(\overline{z})}}w\right)\left(-\frac{dz}{w}\right)^2=\left(\frac{\overline{P(\overline{z})}}{k_t(z)}-\frac{\overline{Q(\overline{z})}}{k_t(z)}w\right)\left(\frac{dz}{w}\right)^2.
\end{equation*}
Hence, it follows that $P\in \bR_4[z]$ and $Q\in i\bR_2[z]$. However, the quadratic form is also symmetric with respect to $A_2$ of equation \eqref{eq:simetriasA} by hypothesis, which implies that
\begin{equation*}
    \left(\frac{P(z)}{k_t(z)}+\frac{Q(z)}{k_t(z)}w\right)\left(\frac{dz}{w}\right)^2=\sigma=\overline{A_2^*\sigma}= \left(\frac{\overline{P(\overline{z})}}{\overline{k_t(\overline{z})}}+\frac{\overline{Q(\overline{z})}}{\overline{k_t(\overline{z})}}w\right)\left(\frac{dz}{w}\right)^2=\left(\frac{\overline{P(\overline{z})}}{k_t(z)}+\frac{\overline{Q(\overline{z})}}{k_t(z)}w\right)\left(\frac{dz}{w}\right)^2,
\end{equation*}
and therefore $Q\in \bR_2[z]$. Hence, it follows that $Q$ is the zero polynomial and $P(z)=\sum_{j=0}^4 c_jz^j$ where $c_j\in \bR$. Finally, we use the hypothesis of $\sigma$ being symmetric with respect to the $\tau$ anti-holomorphic involution of equation \eqref{tauinvolution}. We find that
\begin{equation*}
    \frac{P(z)}{k_t(z)}\left(\frac{dz}{w}\right)^2=\sigma=\overline{\tau^{*}\sigma}= \frac{\overline{P(-\frac{1}{\overline{z}})}}{\overline{k_t(-\frac{1}{\overline{z}})}}\left(\frac{d(-\frac{1}{z})}{\frac{w}{z^2}}\right)^2=\frac{z^4P(-\frac{1}{z})}{k_t(z)}\left(\frac{dz}{w}\right)^2,
\end{equation*}
where we have used that $P\in \bR_4[z]$ and the transformational property $\overline{z}^4k_t(-\frac{1}{\overline{z}})=\overline{k_t(z)}$ which follows from equation \eqref{eq:kt}. Hence, $z^4P(-\frac{1}{z})=P(z)$, which implies that $c_4=c_0$ and $c_3=-c_1$. In conclusion, we have reduced the quadratic form $\sigma\in H^0_{SSS}$, using the symmetries $A_1, A_2$ and $\tau$ to
\begin{equation*}
    \sigma=\frac{c_0z^4-c_1z^3+c_2z^2+ c_1z +c_0}{k_t(z)}\left(\frac{dz}{w}\right)^2
\end{equation*}
which proves the characterization of the $H^0_{SSS}$ space. The other subspaces are handled in a similar way.
Let $\{\sigma^1_{SSS},\ \sigma^2_{SSS},\ \sigma^3_{SSS}\}$, $\{\sigma^1_{SSA},\ \sigma^2_{SSA}\}$, $\{\sigma^1_{ASA},\ \sigma^2_{ASA}\}$ and $\{\sigma^1_{ASS}\}$ be real bases of the vector spaces $H^{0}_{SSS}$, $H^{0}_{SSA}$, $H^{0}_{ASA}$, $H^{0}_{ASS}$ respectively. Then the eight quadratic differentials
\begin{equation*}
    \{\sigma^1_{SSS},\ \sigma^2_{SSS},\ \sigma^3_{SSS},\ \sigma^1_{SSA},\ \sigma^2_{SSA},\ \sigma^1_{ASA},\ \sigma^2_{ASA},\ \sigma^1_{ASS}\}
\end{equation*}
form a complex basis of the space $H^{0,2}(\cG_t).$ We now calculate the nullity of the Residue Matrix with respect to the following complex basis $\cB_{\text{sym}}$ of $\cV(\cG_t)$ (recall Definition \ref{spaceoneforms} for the notation below)
\begin{equation*}
    \cB_{\text{sym}}\coloneq \{\eta^1_{SSS},\ \eta^2_{SSS},\ \eta^3_{SSS},\ \eta^1_{SSA},\ \eta^2_{SSA},\ \eta^1_{ASA},\ \eta^2_{ASA},\ \eta^1_{ASS}\}.
\end{equation*}
Notice that the name of the elements of this basis comes from the symmetries of the associated quadratic differentials, but these symmetries change after we divide by $d \mathcal{G}_t$. We clarify the symmetry properties of the basis $\cB_{\text{sym}}$ in the following lemma 
 \begin{lema}\label{symmetrypropertiesBsym}
     For all $t\in(0,1]$, the complex basis $\cB_{\text{sym}}$ satisfies that for all $j=1,2,3$ and $k=1,2$, $l=1$
     \begin{align*}
         \begin{split}
             &\overline{A_1^{*}\eta_{SSS}^j}=-\eta_{SSS}^j,\quad \overline{A_1^{*}\eta_{SSA}^k}=-\eta_{SSA}^k,\quad \overline{A_1^{*}\eta_{ASA}^k}=\eta_{ASA}^k,\quad \overline{A_1^{*}\eta_{ASS}^l}=\eta_{ASS}^l,\\\\ &\overline{A_2^{*}\eta_{SSS}^j}=\eta_{SSS}^j,\quad \overline{A_2^{*}\eta_{SSA}^k}=\eta_{SSA}^k,\quad \overline{A_2^{*}\eta_{ASA}^k}=\eta_{ASA}^k,\quad \overline{A_2^{*}\eta_{ASS}^l}=\eta_{ASS}^l,\\\\ &\overline{\tau^{*}\eta_{SSS}^j}=\cG_t^2\eta_{SSS}^j,\quad \overline{\tau^{*}\eta_{SSA}^k}=-\cG_t^2\eta_{SSA}^k,\quad \overline{\tau^{*}\eta_{ASA}^k}=-\cG_t^2\eta_{ASA}^k,\quad \overline{\tau^{*}\eta_{ASS}^l}=\cG_t^2\eta_{ASS}^l.
         \end{split}
     \end{align*}
 \end{lema}
 \begin{proof}
     The proof follows from a straightforward application of Definitions \ref{spaceoneforms},  \ref{symmetricquadraticdifferentials} and equations \eqref{tauinvolution}, \eqref{eq:simetriasA}, \eqref{eq:deformapadegauss}. For instance, we have
     \begin{equation*}
         \overline{A_1^{*}\eta^j_{SSS}}=\frac{\overline{A_1^{*}\sigma_{SSS}^j}}{\overline{d\cG_t\circ A_1}}=\frac{\sigma_{SSS}^j}{-d\cG_t}=-\eta_{SSS}^j,\quad \overline{A_2^{*}\eta^k_{SSA}}=\frac{\overline{A_2^{*}\sigma_{SSA}^k}}{\overline{d\cG_t\circ A_2}}=\frac{\sigma_{SSA}^k}{d\cG_t}=\eta_{SSA}^k,
     \end{equation*}
     and
     \begin{equation*} \overline{\tau^{*}\eta^k_{ASA}}=\frac{\overline{\tau^{*}\sigma_{ASA}^k}}{\overline{d\cG_t\circ \tau}}=\frac{-\sigma_{ASA}^k}{d(-1/\cG_t)}=-\cG_t^2\eta_{ASA}^k,\quad \overline{\tau^{*}\eta^l_{ASS}}=\frac{\overline{\tau^{*}\sigma_{ASS}^l}}{\overline{d\cG_t\circ \tau}}=\frac{\sigma_{ASS}^l}{d(-1/\cG_t)}=\cG_t^2\eta_{ASS}^l.
     \end{equation*}
 \end{proof}
The symmetry properties of the elements of the basis $\cB_{\text{sym}}$ allows us to deduce the following observations about the residues at the branch points
\begin{lema}\label{realimaginaryresidues}
    For all $t\in (0,1]$ the residues of the one-forms of $\cB_{\text{sym}}$ at the branch points $B_1(t)$ and $B_2(t)$ satisfy for all $j=1,2,3$ and $k=1,2$, $l=1$
    \begin{align*}
        \begin{split}
              \operatorname{Res}_{B_1(t)}\eta^j_{SSS}\in \bR,\quad , \operatorname{Res}_{B_1(t)}\eta^k_{SSA} \in \bR,\quad \operatorname{Res}_{B_1(t)}\eta^k_{ASA} \in \bR,\quad \operatorname{Res}_{B_1(t)}\eta^l_{ASS} \in \bR, \\ \operatorname{Res}_{B_2(t)}\eta^j_{SSS}\in i\bR ,\quad \operatorname{Res}_{B_2(t)}\eta^k_{SSA} \in i\bR,\quad \operatorname{Res}_{B_2(t)}\eta^k_{ASA} \in \bR ,\quad \operatorname{Res}_{B_2(t)}\eta^l_{ASS} \in \bR. 
        \end{split}
    \end{align*}
\end{lema}

\begin{proof}
    The result follows from an application of Lemma \ref{BranchpointsDeformationmap}, Lemma \ref{simplificationresidueeta}, the decomposition of the $H^{0,2}(\cG_t)$ into eight real subspaces \eqref{decompositionH0} and Proposition \ref{residuetool}. 
\end{proof}
The Residue Matrix in the complex basis $\cB_{\text{sym}}$ has several symmetries due to the fact that all branch points are determined by $B_1(t)$ and $B_2(t)$ by applying the $A_1,\ A_2$ or $\tau$ symmetries. Therefore, the relevant information of the Residue Matrix is contained in its first two rows. To exploit these symmetries, we introduce the following terminology
\begin{defn}\label{residuevectors}
\normalfont
    We adopt the following notation for the elements of the Residue Matrix
    \begin{align*}
    \begin{split}
    &\Vec{\mu}_a\coloneq \left(\text{Res}_{B_1}\eta^1_{SSS},\ \text{Res}_{B_1}\eta^2_{SSS},\ \text{Res}_{B_1}\eta^3_{SSS}\right)\in \bR^3,\quad i\Vec{\mu}_b\coloneq \left(\text{Res}_{B_2}\eta^1_{SSS},\ \text{Res}_{B_2}\eta^2_{SSS},\ \text{Res}_{B_2}\eta^3_{SSS}\right)\in i\bR^3,\\&
    \Vec{\nu}_a\coloneq \left(\text{Res}_{B_1}\eta^1_{SSA},\text{Res}_{B_1}\eta^2_{SSA}\right)\in \bR^2,\quad i \Vec{\nu}_b\coloneq \left(\text{Res}_{B_2}\eta^1_{SSA},\text{Res}_{B_2}\eta^2_{SSA}\right)\in i\bR^2,\\& \Vec{\kappa}_{a}\coloneq \left(\text{Res}_{B_1}\eta^1_{ASA},\text{Res}_{B_1}\eta^2_{ASA}\right)\in \bR^2,\quad \Vec{\kappa}_{b}\coloneq \left(\text{Res}_{B_2}\eta^1_{ASA},\text{Res}_{B_2}\eta^2_{ASA}\right)\in \bR^2,\\&    \lambda_{a}\coloneq \text{Res}_{B_1}\eta^1_{ASS}\in \bR \quad \lambda_{b}\coloneq \text{Res}_{B_2}\eta^1_{ASS}\in \bR.
        \end{split}
    \end{align*}
    The fact that some residues are purely real and some others are purely imaginary comes from Lemma \ref{realimaginaryresidues}.
\end{defn}
We present the symmetries of the Residue Matrix in the following theorem

\begin{thm}\label{residuematrix}
    The Residue Matrix with respect to the basis $\cB_{\text{sym}}$ has the form
    \begin{gather*}
       \operatorname{Res}_{\cB_{\text{sym}}}(\cG_t)= \begin{pmatrix}
I_4 & 0 \\
0 & D_4
\end{pmatrix}_{8\times 8} \left(\begin{array}{cccc}           \Vec{\mu}_a&\Vec{\nu}_a&\Vec{\kappa}_a&\lambda_a\\ i\Vec{\mu}_b&i\Vec{\nu}_b&\Vec{\kappa}_b&\lambda_b\\ -\Vec{\mu}_a&-\Vec{\nu}_a&\Vec{\kappa}_a&\lambda_a\\ -i\Vec{\mu}_b&-i\Vec{\nu}_b&\Vec{\kappa}_b&\lambda_b\\\Vec{\mu}_a&-\Vec{\nu}_a&-\Vec{\kappa}_a&\lambda_a\\ -i\Vec{\mu}_b&i\Vec{\nu}_b&-\Vec{\kappa}_b&\lambda_b\\ -\Vec{\mu}_a&\Vec{\nu}_a&-\Vec{\kappa}_a&\lambda_a\\ i\Vec{\mu}_b&-i\Vec{\nu}_b&-\Vec{\kappa}_b&\lambda_b
        \end{array}\right)_{8\times 8},
    \end{gather*}
    where
    \begin{equation*}
        D_4=\operatorname{diag}(\overline{\cG_t(B_1)^2},\overline{\cG_t(B_2)^2},\overline{\cG_t(B_3)^2},\overline{\cG_t(B_4)^2}).
    \end{equation*}
\end{thm}

\begin{proof}
    We know from Lemma \ref{BranchpointsDeformationmap} that $B_3(t) = A_1(B_1(t))$. Using Definition \ref{residuevectors}, Lemma \ref{symmetrypropertiesBsym} and Lemma \ref{Residuerelatedpoints} we relate the first and third rows of the Residue Matrix
    \begin{align*}
        \operatorname{Res}_{\cB_{\text{sym}}}(\cG_t)_{3,\cdot}&=(\operatorname{Res}_{B_3}\eta^j_{SSS},\operatorname{Res}_{B_3}\eta^k_{SSA},\operatorname{Res}_{B_3}\eta^k_{ASA},\operatorname{Res}_{B_3}\eta^l_{ASS})\\&=\overline{(\operatorname{Res}_{B_1}\overline{A_1^{*}\eta^j_{SSS}},\operatorname{Res}_{B_1}\overline{A_1^{*}\eta^k_{SSA}},\operatorname{Res}_{B_1}\overline{A_1^{*}\eta^k_{ASA}},\operatorname{Res}_{B_1}\overline{A_1^{*}\eta^l_{ASS})}}\\&=\overline{(-\operatorname{Res}_{B_1}\eta_{SSS}^j,-\operatorname{Res}_{B_1}\eta_{SSA},\operatorname{Res}_{B_1}\eta^k_{ASA},\operatorname{Res}_{B_1}\eta^l_{ASS})}\\&=\overline{(-\Vec{\mu}_a,-\Vec{\nu}_a,\Vec{\kappa}_a,\lambda_a)}\\&=(-\Vec{\mu}_a,-\Vec{\nu}_a,\Vec{\kappa}_a,\lambda_a).
    \end{align*}
    The form of the fourth row of the Residue Matrix follows from the observation that $B_4(t)=A_2(B_2(t))$ and a similar analysis as before. The last half of the Residue Matrix is obtained through the action of the $\tau$ symmetry on the first four rows, since $B_{i+4}(t)=\tau(B_i(t))$. We point out that the one-forms $\eta\in \cB_{\text{sym}}$ have poles of order at most two at the branch points $B_i(t)=(p_i(t),w(p_i(t)))$, therefore, in a complex chart $\phi^{-1}(\xi)$ of equation \eqref{chartgenericpoints} centered at $B_i(t)$ we have
    \begin{align*}
    &\left(\phi^{-1}\right)^{*}\eta=\left(\frac{a_{-2}}{\xi^2}+\frac{a_{-1}}{\xi}+h(\xi)\right)d\xi, \ h\ \text{holomorphic}, \\
    &a_{-1}=\operatorname{Res}_{p_i(t)}\eta,\\ &\left(\cG_t\circ \phi^{-1}\right)^2(\xi)=\cG_t(B_i(t))^2+O(\xi^2).
    \end{align*}
   It follows that for all $\eta\in \cB_{\text{sym}}$
    \begin{equation*}
\operatorname{Res}_{B_i}\cG_t^2\eta=\cG_t(B_i(t))^2\operatorname{Res}_{B_i}\eta.
    \end{equation*}
    We conclude the proof by showing how it is possible to relate the two halves of the Residue Matrix through the action of the $\tau$ symmetry, in the particular case of the second and sixth rows, the others follow the same principle
    \begin{align*}
           \operatorname{Res}_{\cB_{\text{sym}}}(\cG_t)_{6,\cdot}&=(\operatorname{Res}_{B_6}\eta^j_{SSS},\operatorname{Res}_{B_6}\eta^k_{SSA},\operatorname{Res}_{B_6}\eta^k_{ASA},\operatorname{Res}_{B_6}\eta^l_{ASS})\\&=\overline{(\operatorname{Res}_{B_2}\overline{\tau^{*}\eta^j_{SSS}},\operatorname{Res}_{B_2}\overline{\tau^{*}\eta^k_{SSA}},\operatorname{Res}_{B_2}\overline{\tau^{*}\eta^k_{ASA}},\operatorname{Res}_{B_2}\overline{\tau^{*}\eta^l_{ASS})}}\\&=\overline{(\operatorname{Res}_{B_2}\cG_t^2\eta_{SSS}^j,-\operatorname{Res}_{B_2}\cG_t^2\eta_{SSA},-\operatorname{Res}_{B_2}\cG_t^2\eta^k_{ASA},\operatorname{Res}_{B_2}\cG_t^2\eta^l_{ASS})}\\&=\overline{\cG_t^2(B_2(t))}\ \overline{(\operatorname{Res}_{B_2}\eta_{SSS}^j,-\operatorname{Res}_{B_2}\eta_{SSA},-\operatorname{Res}_{B_2}\eta^k_{ASA},\operatorname{Res}_{B_2}\eta^l_{ASS})}\\&=\overline{\cG_t^2(B_2(t))}\ \overline{(i\Vec{\mu}_b,-i\Vec{\nu}_b,-\Vec{\kappa}_b,\lambda_b)}\\&=\overline{\cG_t^2(B_2(t))}\ (-i\Vec{\mu}_b,i\Vec{\nu}_b,-\Vec{\kappa}_b,\lambda_b),
    \end{align*}
    which completes the proof.
\end{proof}
The following corollary can be established by Proposition \ref{isomorphismHchapeuspace} and straightforward row operations in the Residue Matrix $\operatorname{Res}_{\cB_{\text{sym}}}(\cG_t)$ of Theorem \ref{residuematrix}, since these operations preserve the null space.

\begin{cor}\label{reducedresiduematrix}
    The space $\hat{H}(\cG_t)$ is isomorphic to the nullity space of the reduced Residue Matrix
    \begin{gather*}
       \operatorname{Res}^{\text{red}}_{\cB_{\text{sym}}}(\cG_t) =\begin{pmatrix}
\Vec{\mu}_a&0&0&0\\i\Vec{\mu}_b&0&0&0\\0&0&\Vec{\kappa}_a&0\\0&0&\Vec{\kappa}_b&0\\0&\Vec{\nu}_a&0&0\\0&i\Vec{\nu}_b&0&0\\ 0&0&0&\lambda_a\\0&0&0&\lambda_b
        \end{pmatrix}_{8\times 8}.
    \end{gather*}
\end{cor}

In order to compute the nullity of the reduced Residue Matrix $\operatorname{Res}^{red}_{\cB_{\text{sym}}}(\cG_t)$ we use a specific basis of $\cB_{\text{sym}}$ that depend on careful choices of parameters presented below. We remark that these parameters were introduced ad hoc in order to make the first entries of the vectors $\Vec{\mu}_a, \Vec{\mu}_b, \Vec{\nu}_a, \Vec{\nu}_b, \Vec{\kappa}_a, \Vec{\kappa}_b$ vanish.
\begin{defn}\label{def:coefficientspolynomials}
\normalfont
   Consider the following set of real functions defined on the interval $(0,1]$
   \begin{align*}
       c_{1N}(t) =& -rt^4+(-4+2r^2)t^3-2rt^2+(-4+6r^2)t+3r,\quad
        c_{1D}(t) =-3rt^3+2r^2t^2-rt,\\
       c_{2N}(t) =&-t^4+4rt^3-4r^2t^2+8rt+(1-8r^2),\quad
        c_{2D}(t) =2t^3-5rt^2+(2+2r^2)t-3r,\\
       c_{3N}(t) =& -rt^6+(-4+8r^2)t^5+(23r-14r^3)t^4+(-20r^2+4r^4)t^3\\
       &+(13r-16r^3)t^2+(4-20r^2+12r^4)t+(-3r+6r^3) ,\\
       c_{4N}(t) =&2r^2t^6+(10r-8r^3)t^5+(8-38r^2+8r^4)t^4+(-4r+28r^3)t^3+(8-46r^2)t^2+(-30r+36r^3)t+18r^2,\\
       c_{34D}(t) =&-3rt^5+10r^2t^4+(-2r-10r^3)t^3+(6r^2+4r^4)t^2+(r-2r^3)t.
    \end{align*}
\end{defn}
Using these functions, we define the following polynomials in order to define the basis we want to use
\begin{defn}\label{specialpolynomials}
\normalfont
    We introduce the following set of polynomial functions
    \begin{align*}
        &P_1(z,t)\coloneq (z+t) (tz-1)p_1(z,t),\quad p_1(z,t)\coloneq c_{34D}(t) z^4 - c_{3N}(t) z^3 + c_{4N}(t) z^2 + c_{3N}(t) z + c_{34D}(t),\\ &P_2(z,t)\coloneq (z+t) (tz-1)p_2(z,t),\quad p_2(z,t)\coloneq  c_{1D}(t) z^4 - c_{1N}(t) z^3 - c_{1N}(t) z - c_{1D}(t), \\ &P_3(z,t)\coloneq (tz-1)^2 (z+r)p_3(z,t) ,\quad p_3(z,t)\coloneq c_{2D}(t) z^2 + c_{2N}(t) z - c_{2D}(t).
    \end{align*}
\end{defn}
Now we use this set of polynomials to define a symmetric basis of one-forms as follows:
\begin{lema}\label{specialbasis} 
    The following collection of eight one-forms 
     \begin{align*}
      &\eta_{SSS}^1(t)=  \frac{P_1(z,t)}{\mathcal{G}_t k_t^2} dz,\quad \eta_{SSS}^2(t) = \frac{(z+t) (tz-1) (z^3 - z)}{\mathcal{G}_t k_t^2} dz,\quad
        \eta_{SSS}^3(t) = \frac{(z+t) (tz-1) (z^4 + 1)}{\mathcal{G}_t k_t^2}dz,\\ &\eta_{SSA}^1(t) = \frac{P_2(z,t)}{\mathcal{G}_t k_t^2} dz,\quad \eta_{SSA}^2(t) = \frac{(z+t) (tz-1) (z^4 - z^3 - z - 1)}{\mathcal{G}_t k_t^2} dz, \\
        &\eta_{ASA}^1(t) = \frac{P_3(z,t)}{k_t^2} dz,\quad \eta_{ASA}^2(t) = \frac{(tz-1)^2 (z+r) (z^2 + z - 1)}{k_t^2}dz\\
        &\eta_{ASS}(t) =  \frac{(tz-1)^2 (z+r) (z^2+1)}{k_t^2} dz,
    \end{align*}
    constitutes a complex basis $\cB_{\text{sym}}$ of $\cV(\cG_t)$.
\end{lema}

The proof of Lemma \ref{specialbasis} relies on the following estimates.

\begin{lema}\label{propertiesconstants}
    For all $t\in(0,1]$ it holds that
    \begin{itemize}
        \item[I.] $0<c_{4N}(t)$.
        \item[II.] $0<c_{1N}(t)-c_{1D}(t)$.
        \item[III.] $c_{2N}(t)<c_{2D}(t)<0$.
    \end{itemize}
\end{lema}
\begin{proof}[Proof of Lemma \ref{propertiesconstants}]

Note that $c_{4N}$ is a polynomial in $t$ and $r$, as described in Definition \ref{def:coefficientspolynomials}. We denote this polynomial by $P_{4N}(t, r)$ and compute the $\Theta$ function for this problem, as defined in \eqref{eq:thetageral}, using the estimates for $r$ given by Lemma \ref{rLopezNumerical} and some large enough $n$ representing the number of sub-intervals, to show that
\begin{equation}\label{eq:thetaforfirstlemma}
    \Theta(P_{4N}; r^{\inf},r^{\sup},n) > 0.
\end{equation}
The concrete computations and the other items of this lemma are carried out in the companion Mathematica notebook. In Example \ref{ex:examplec2D}, we have included the computation related to the condition $c_{2D}(t) < 0$ using the method described in Appendix \ref{section::EstimatesPolynomials}.
\end{proof}

\begin{proof}[Proof of Lemma \ref{specialbasis} ]
    Fix a time $t\in (0,1]$. By Theorem \ref{lema:formadeH0} it is enough to prove that the eight one-forms are $\bC$-linearly independent. Suppose that there exists $\lambda_{j}\in \bC$, $j=1,\ldots, 8$ such that
    \begin{equation}\label{linearcombination}
        \lambda_1\eta_{SSS}^1+\lambda_2\eta_{SSS}^2+\lambda_3\eta_{SSS}^3+\lambda_4\eta_{SSA}^1+\lambda_5\eta_{SSA}^2+\lambda_6\eta_{ASA}^1+\lambda_7\eta_{ASA}^2+\lambda_8\eta^1_{ASS}=0.
    \end{equation}
    We claim that $\lambda_j=0$ for all $j=1,\ldots ,8$. In fact, evaluating equation \eqref{linearcombination} at the points $(z,w),\ (z,-w)\in M$ and using Definition \eqref{eq:deformapadegauss}, we have after adding and subtracting the pairs of equations that for all $(z,w)\in M$
    \begin{equation*}
        \lambda_1\eta_{SSS}^1+\lambda_2\eta_{SSS}^2+\lambda_3\eta_{SSS}^3+\lambda_4\eta_{SSA}^1+\lambda_5\eta_{SSA}^2=0,\quad \lambda_6\eta_{ASA}^1+\lambda_7\eta_{ASA}^2+\lambda_8\eta^1_{ASS}=0.
    \end{equation*}
After multiplying by appropriate factors we obtain
\begin{equation}\label{firstpair}
    \lambda_1p_1(z,t)+\lambda_2(z^3-z)+\lambda_3(z^4+1)+\lambda_4p_2(z,t)+\lambda_5(z^4-z^3-z-1)=0,
\end{equation}
\begin{equation}\label{secondpair}
    \lambda_6p_3(z,t)+\lambda_7(z^2+z-1)+\lambda_8(z^2+1)=0.
\end{equation}
On one hand, from the $z^2$ coefficient in equation \eqref{firstpair} we obtain that $\lambda_1c_{4N}(t)=0$ which means by Lemma \ref{propertiesconstants} that $\lambda_1=0$. By analyzing the $z^0$, $z$, $z^3$ and $z^4$ coefficients in equation \eqref{firstpair} we obtain that 
\begin{equation*}
   \lambda_3-\lambda_4c_{1D}(t)-\lambda_5=0 ,\quad-\lambda_2-\lambda_4c_{1N}(t)-\lambda_5=0,\quad \lambda_2-\lambda_4c_{1N}(t)-\lambda_5=0,\quad \lambda_3+\lambda_4c_{1D}(t)+\lambda_5=0,
\end{equation*}
which implies that $\lambda_2=\lambda_3=0$ and $\lambda_4c_{1N}(t)=\lambda_4c_{1D}(t)=-\lambda_5$. From Lemma \ref{propertiesconstants} we obtain $\lambda_4=\lambda_5=0$. On the other hand, from the coefficients $z^0$, $z$ and $z^2$ in equation \eqref{secondpair} we obtain
\begin{equation*}
    -\lambda_6c_{2D}(t)-\lambda_7+\lambda_8=0,\quad \lambda_6c_{2N}(t)+\lambda_7=0,\quad \lambda_6c_{2D}(t)+\lambda_7+\lambda_8=0,
\end{equation*}
which implies that $\lambda_8=0$ and $\lambda_6c_{2N}(t)=\lambda_6c_{2D}(t)=-\lambda_7$. From Lemma \ref{propertiesconstants} we have $\lambda_6=\lambda_7=0$ concluding the proof.
\end{proof}

Now we establish the properties of the residues in the basis of Lemma \ref{specialbasis}
\begin{lema}\label{specialpropertiesbasisBsym}
    The complex basis $\cB_{\text{sym}}$ of Lemma \ref{specialbasis} has the following properties for all $t\in (0,1]$
   \begin{align*}
       &\Vec{\mu}_a=(0,\mu_a^2,\mu_a^3),\quad \Vec{\mu}_b=(0,\mu_b^2,\mu_b^3),\quad \mu_a^2\mu_b^3-\mu_a^3\mu_b^2\neq 0,\\ &\Vec{\nu}_a=(0,\nu_a^2),\quad \Vec{\nu}_b=(0,\nu_b^2),\quad \nu_a^2\neq 0,\quad \nu_b^2\neq0,\\&\Vec{\kappa}_a=(0,\kappa_a^2),\quad \Vec{\kappa}_b=(0,\kappa_b^2),\quad \kappa_a^2\neq 0,\quad \kappa_b^2\neq0,\\ &\lambda_a\neq 0,\quad \lambda_b\neq0 .
   \end{align*}
\end{lema}
\begin{proof}
We start by proving that for all $t\in (0,1] $$$\mu_a^1=0,\quad \mu_b^1=0,\quad \nu_a^1=0,\quad \nu_b^1=0,\quad \kappa_a^1=0,\quad \kappa_b^1=0.$$

These equations are equivalent to prove that for all $t\in (0,1]$ \begin{align*}
   & \text{Res}_{z=a_t}\frac{P_1(z,t)}{k_t^2(z)}=0,\quad \text{Res}_{z=b_t}\frac{P_1(z,t)}{k_t^2(z)}=0,\\
   & \text{Res}_{z=a_t}\frac{P_2(z,t)}{k_t^2(z)}=0,\quad \text{Res}_{z=b_t}\frac{P_2(z,t)}{k_t^2(z)}=0,\\ &\text{Res}_{z=a_t}\frac{P_3(z,t)}{k_t^2(z)}=0,\quad \text{Res}_{z=b_t}\frac{P_3(z,t)}{k_t^2(z)}=0,
\end{align*}
since the polynomial $k_t$ has only simple roots at $a,\ b,\ -\frac{1}{a},\ -\frac{1}{b}$ and the Gauss map $\cG$ has a branch point at the points $B_j$, \textit{i.e.} $\cG_t'(B_j)=0$.

We handle each pair of equations separately using the same strategy. The key idea is that the polynomial $k_t(z)$ of equation \eqref{eq:kt} and the polynomials $P_j(z)$ introduced in Lemma \ref{specialbasis} satisfy ordinary differential equations that have a direct consequence for the calculation of the residue vectors of Definition \ref{residuevectors}.

Let us consider the residues associated to the polynomial $P_1$. A straightforward computation shows that for fixed $t\in (0,1]$, we have
\begin{gather}\label{EDO1}
    P_1'(z)k_t'(z)-k_t''(z)P_1(z)=k_t(z)Q_1(z),
\end{gather}
where 
\begin{equation*}
    Q_1(z)=s_4z^4+s_3z^3+s_2z^2+s_1z+s_0,
\end{equation*}
with
\begin{align*}
s_4=12rt^2\left[1-2r^2+(6r+4r^3)t-(2+10r^2)t^2+10rt^3-3t^4\right],
\end{align*}
\begin{gather*}
    s_3=-16t\left[-r+2r^3+(2-7r^2+8r^4)t+(5r-12r^3)t^2-8r^2t^3+(11r-2r^3)t^4-(2+r^2)t^5+rt^6\right],
\end{gather*}
\begin{gather*}
    s_2=6\left[
    \begin{aligned}
    &-3r+6r^3+(4-5r^2+12r^4)t+(-10r+8r^3)t^2+(4-39r^2-8r^4)t^3\\&+(6r+30r^3)t^4+(4-15r^2+4r^4)t^5+(-10r+4r^3)t^6+(4-5r^2)t^7+rt^8),
    \end{aligned}
    \right]
\end{gather*}
\begin{gather*}
    s_1= 6r\left[
    \begin{aligned}
    -9r&+(12-12r^2)t+(10r+12r^3)t^2+(4-20r^2)t^3-8rt^4\\&+(12+4r^2)t^5+(-10r+4r^3)t^6+(4-4r^2)t^7+rt^8,
    \end{aligned}
    \right]
\end{gather*}
and
\begin{gather*}
     s_0=2\left[
    \begin{aligned}
    &3r-6r^3+(-4+11r^2-12r^4)t+(8r^3)t^2+(-4+13r^2)t^3\\&-(2r+10r^3)t^4+(-4+5r^2-4r^4)t^5+8rt^6+(-4+3r^2)t^7-rt^8.
    \end{aligned}
    \right]
\end{gather*}
Suppose that there exists some $t\in (0,1]$ such that $P_1(a_t,t)=0$. Then by equation \eqref{EDO1} it follows that $P_1'(a_t,t)=0$, since $a_t$ is a simple root of $k_t(z)$. This implies that $z=a_t$ is a zero of $P_1(z,t)$ with order at least two, and therefore 
\begin{equation*}
    \operatorname{Res}_{z=a_t}\frac{P_1(z,t)}{k_t^2(z)}=0.
\end{equation*}
Now, let $t\in (0,1]$ such that $P_1(a_t,t)\neq 0$. Hence, we may apply Proposition \ref{residuetool} and Lemma \ref{polynomialidentity} to obtain
\begin{equation*}
    \operatorname{Res}_{z=a_t}\frac{P_1(z,t)}{k_t^2(z)}=\frac{P_1(a_t,t)}{t^2(a_t-b_t)^2(a_t+\frac{1}{a_t})^2(a_t+\frac{1}{b_t})^2}\left(\frac{P_1'(a_t,t)}{P_1(a_t,t)}-\frac{k_t''(a_t)}{k_t'(a_t)}\right)=0,
\end{equation*}
by the differential equation \eqref{EDO1} satisfied by the polynomial $P_1(z,t)$. The proof for $z=b_t$ is completely analogous. Similarly the polynomials $P_2(z,t)$ and $P_3(z,t)$ satisfy the differential equation
\begin{gather}\label{EDO23}
    P_2'(z)k_t'(z)-k_t''(z)P_2(z)=k_t(z)Q_2(z),\quad P_3'(z)k_t'(z)-k_t''(z)P_3(z)=k_t(z)Q_3(z), 
\end{gather}
where
\begin{align*}
\begin{split}
Q_2(z)=&12rt^2(-1+2rt-3t^2)z^4-16t(r+(-2+4r^2)t-2rt^2-2t^3+rt^4)z^3\\&+6(t^2-1)\left[-3r+(4-6r^2)t+2rt^2+(4-2r^2)t^3+rt^4\right]z^2\\&2\left[3r+(-4+6r^2)t-7rt^2-5rt^4+(4-2r^2)t^5+rt^6\right],
\end{split}
\end{align*}
and
\begin{align*}
    Q_3(z)=&-\left[24rt^2-(16+16r^2)t^3+40rt^4-16t^5\right]z^3\\&-\left[-33rt+(18+72r^2)t^2-(102r+12r^3)t^3+(24+48r^2)t^4-33rt^5+6t^6\right]z^2\\&- \left[9r-60r^2t+(57r+36r^3)t^2-72r^2t^3+(27r+12r^3)t^4-12r^2t^5+3rt^6\right]z\\& +2-4r^2+(-5r+24r^3)t+(8-24r^2)t^2+(-6r+12r^3)t^3+(6-12r^2)t^4+3rt^5.
\end{align*}
The rest of the argument is completely analogous. 

To conclude the proof of this lemma, it remains to prove that for every $t \in (0,1]$ we have
\begin{equation}\label{eq:residuesnonzero}
\begin{split}
    \mu_a^2\mu_b^3-\mu_a^3\mu_b^2 \neq 0, \ \nu^2_a \neq 0, \ \nu^2_b \neq 0, \\
    \kappa_a^2 \neq 0, \ \kappa_b^2 \neq 0, \ \lambda_a \neq 0, \ \lambda_b \neq 0.
\end{split}
\end{equation}
We now reduce \eqref{eq:residuesnonzero} to the verification of the sign of some polynomials in several variables in order to apply the strategy presented in Appendix \ref{section::EstimatesPolynomials}. Let us describe how to prove the non-vanishing condition $\mu_a^2\mu_b^3-\mu_a^3\mu_b^2 \neq 0$ for every $t \in (0,1]$, since the other desired conditions can be proved similarly.

As observed in the beginning of this proof, in each of the residues $\mu_a^2$, $\mu_b^3$, $\mu_a^3$ and $\mu_b^2$ we can factor out the Gauss map evaluated at the branch point where the residue is computed. We then apply Proposition \ref{residuetool} to the remaining expression and obtain that any factor is a rational function in $a_t$, $b_t$, $t$ and $r$. We show in the companion Mathematica notebook that the expression obtained in this way for the condition $\mu_a^2\mu_b^3-\mu_a^3\mu_b^2 \neq 0$ that we are analyzing has a factor
\begin{equation*}
    Q_{\mu}(t)=\frac{2 a_t^4 b_t^4}{i t^4 (a_t - b_t)^3(1 + a_t b_t)^3(1 + a_t^2)^3(1 + b_t^2)^3},
\end{equation*}
so that 
\begin{equation*}
    \mu_a^2\mu_b^3-\mu_a^3\mu_b^2 = \frac{Q_{\mu}(t) \tilde{P}_{\mu}(a_t,b_t,t,r)}{\mathcal{G}_t(B_1) \mathcal{G}_t(B_2)}, 
\end{equation*} where $\tilde{P}_{\mu}(a_t,b_t,t,r)$ is a polynomial with rational coefficients. Using Lemma \ref{lem:posicoesdosatbt} we can verify that $Q_{\mu}(t)$ is nonzero in $(0,1]$, therefore we only have to prove that the polynomial $\tilde{P}_{\mu}(a_t,b_t,t,r)$ is nonzero in this interval. 

In order to apply the strategy described in Appendix \ref{section::EstimatesPolynomials} we introduce positive variables defining the polynomial $P_{\mu}(|a_t|,|b_t|,t,r):=P_{\mu}(-a_t,-b_t,t,r)$. Now we use the construction of the piecewise linear functions described in Lemma \ref{lem:estimatesforabtau} that bound from above and from below the functions $|a_t|$ and $|b_t|$ in $(0,1]$, and we use the estimates for $r$ given by Lemma \ref{rLopezNumerical}, to compute the $\Theta$ function for this problem, as defined in \eqref{eq:thetageral}: there exists some number of sub-intervals $n$ such that
\begin{equation}\label{thetaforresidues}
    \Theta(P_{\mu}; |a|^{\inf}_t,|a|^{\sup}_t,|b|^{\inf}_t,|b|^{\sup}_t,r^{\inf},r^{\sup},n) >0 .
\end{equation}
 This computation is done in the companion Mathematica notebook, where we also include the other conditions in \eqref{eq:residuesnonzero}, which are similar to this.
\end{proof}
Now, we can compute the dimension of the space introduced in equation \eqref{eq:spaceresidues}.
\begin{thm}\label{Hchapelsymmetricbasis}
    For all $t\in (0,1]$, the complex dimension of the space $\hat{H}(\cG_t)$ is equal to three with a complex basis given by $\{\sigma_j=\eta_j(t)d\cG_t\}_{j=1}^3$ where
    \begin{align*}
       \eta_1(t)\coloneq \eta^1_{SSS}(t),\quad \eta_2(t)\coloneq \eta^1_{SSA}(t),\quad \eta_3(t)\coloneq \eta^1_{ASA}(t).
    \end{align*}
\end{thm}
\begin{proof}
Let us fix $t\in (0,1]$ and consider the complex basis $\cB_{\text{sym}}$ of Lemma \ref{specialbasis}. The space $\hat{H}(\cG_t)$ is isomorphic to the null space of the reduced Residue Matrix $\operatorname{Res}^{\text{red}}_{\cB_{\text{sym}}}(\cG_t)$ by Corollary \ref{reducedresiduematrix}. A vector $\mathbf{c}=\sum_{j=1}^8c_je_j$ is in the null space of the reduced Residue Matrix if and only if
\begin{equation*}
    \begin{pmatrix}
        \mu_a^2&\mu_a^3\\ i\mu_b^2& i\mu_b^3
    \end{pmatrix}\begin{pmatrix}
        c_2\\c_3
    \end{pmatrix}=\begin{pmatrix}
        0\\ 0
    \end{pmatrix},\quad \nu_a^2c_5=0,\quad \nu_b^2c_5=0,\quad \kappa_a^2c_7=0,\quad \kappa_b^2c_7=0,\quad \lambda_ac_8=0,\quad \lambda_bc_8=0.
\end{equation*}
Hence, the quadratic form $$\sigma=\left(c_1\eta^1_{SSS}+c_2 \eta^2_{SSS}+c_3\eta^3_{SSS}+c_4 \eta^1_{SSA}+c_5\eta^2_{SSA}+c_6 \eta^1_{ASA}+c_7\eta^2_{ASA}+c_8\eta^1_{ASS}\right)d\cG_t\in H^{0,2}(\cG_t)$$ belongs to the space $\hat{H}(\cG_t)$ if and only if
\begin{equation*}
    c_2=0,\quad c_3=0,\quad c_5=0,\quad c_7=0,\quad c_8=0,
\end{equation*}
where we have used Proposition \ref{isomorphismHchapeuspace} and Lemma \ref{specialpropertiesbasisBsym}.

Therefore, the quadratic differentials $\{\sigma_j=\eta_j(t)d\cG_t\}_{j=1}^3$ where $$\eta_1(t)=\eta_{SSS}^1(t),\quad \eta_2(t)=\eta_{SSA}^1(t),\quad \eta_3(t)=\eta_{ASA}^1(t)$$ generate the complex space $\hat{H}(\cG_t)$, and moreover it constitutes a basis for that space since the $\eta_j(t)$ one-forms are linearly independent by an application of Lemma \ref{specialbasis}.
\end{proof}

\subsubsection{The period problem} 
The main objective of this Section is to prove that for each $t\in (0,1]$ the space $H(\mathcal{G}_t)$ defined in equation \eqref{eq:spaceperiods}, is trivial, \textit{i.e.} it is not possible to find a non-identically zero one-form $\eta(t) = \gamma_1(t) \eta_1(t) + \gamma_2(t) \eta_2(t) + \gamma_3(t) \eta_3(t)$ satisfying the period condition
\begin{equation*}
    \text{Re}\int_{\alpha_j} ((1 - \mathcal{G}_t^2), i(1 + \mathcal{G}_t^2) , 2\mathcal{G}_t)\eta(t)=0,
\end{equation*}
where $\alpha_1, \alpha_2$ are generators of the first homology group of $\overline{M}$.
\begin{defn}[{Period Matrix}]\label{periodmatrix}
\normalfont
    We introduce \textit{the Period Matrix} $\mathcal{P}er(\mathcal{G}_t)$ as the real square matrix
\begin{gather*}
    \mathcal{P}er(\mathcal{G}_t)=\begin{pmatrix}
        \text{Re}\int_{\alpha_1}\Phi_{l_i,m_j}(t)&-\text{Im}\int_{\alpha_1}\Phi_{l_i,m_j}(t)\\ \text{Re}\int_{\alpha_2} \Phi_{l_i,m_j}(t)&-\text{Im}\int_{\alpha_2}\Phi_{l_i,m_j}(t)
    \end{pmatrix}\in \cM_{6\times6}(\bR),
\end{gather*}
where
\begin{equation*}
    \Phi_{l_1,m_j}(t)\coloneq (1-\cG_t^2)\eta_j(t),\quad \Phi_{l_2,m_j}(t)\coloneq i(1+\cG_t^2)\eta_j(t) ,\quad \Phi_{l_3,m_j}(t)\coloneq 2\cG_t\eta_j(t).
\end{equation*}
\end{defn}

The next proposition has a straightforward proof and highlights the relevance of the Period Matrix. 

\begin{prop}
    Let $\{E_j\}_{j=1}^6$ be the canonical basis of $\bR^6$. For all $t\in (0,1]$, the linear map
    \begin{equation*}
        S:\text{Nul}(\mathcal{P}er(\mathcal{G}_t))\to H(\cG_t),\quad S\left(\sum_{j=1}^6c_j(t)E_j\right)=\sum_{j=1}^3(c_j(t)+ic_{j+3}(t))\sigma_j(t)
    \end{equation*}
    is an isomorphism.
\end{prop}

In order to compute the entries of the Period Matrix $\mathcal{P}er(\mathcal{G}_t)$, we introduce explicit generators for the homology group of $\overline{M}$.

\begin{figure}
    \centering
    \includegraphics[width=0.7\linewidth]{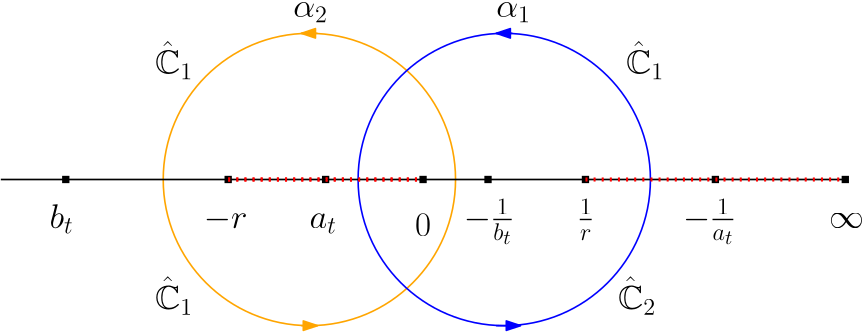}
    \caption{Generators of the first homology group of $\overline{M}$.}
    \label{fig:homological_paths}
\end{figure}
\begin{defn}[\textit{Cf}. \cite{Sarenhu}]
    \normalfont
    The first homology group of $\overline{M}$ is generated by the paths 
\begin{equation*}
    \alpha_1\coloneq \left\lbrace \left(\frac{1}{r}-\frac{r}{4}\right)+\left(\frac{1}{r}+\frac{r}{4}\right)e^{i\theta}\in \hat{\bC}_1| 0\leq \theta\leq \pi \right\rbrace \bigcup \left\lbrace \left(\frac{1}{r}-\frac{r}{4}\right)+\left(\frac{1}{r}+\frac{r}{4}\right)e^{i\theta}\in \hat{\bC}_2| \pi\leq \theta\leq 2\pi \right\rbrace, 
\end{equation*}
\noindent and 
\begin{equation*}
     \alpha_2\coloneq \left\lbrace \left(-r+\frac{1}{4r}\right)+\left(r+\frac{1}{4r}\right)e^{i\theta}\in \hat{\bC}_1| 0\leq \theta\leq 2\pi \right\rbrace.
\end{equation*}
\end{defn}
We compute the determinant of the Period Matrix associated with the López Klein bottle
\begin{thm}\label{determinantperiodmatrix}
    \begin{equation*}
        \det(\mathcal{P}er(\mathcal{G}_t))=64i\int_{\alpha_1}\eta_1(t)\cdot \int_{\alpha_1}\eta_2(t)\cdot \int_{\alpha_1}\cG_t \eta_3(t)\cdot \int_{\alpha_2}\eta_1(t)\cdot \int_{\alpha_2}\eta_2(t)\cdot \int_{\alpha_2}\cG_t\eta_3(t).
    \end{equation*}
\end{thm}
We need to introduce some preliminary lemmas to accomplish this result. In the literature \cite{morabito2009index}, \cite{Nayatani-Costa}, \cite{Sarenhu}, a common technique to calculate integrals of one-forms along the homological generators of a compact Riemann surface is to perturb the one-form by an exact differential, in such a way that the perturbation does not have any poles along the path of integration. In the following calculations we follow a different approach, which is based on the idea that the same homological generator can be deformed in multiple ways in order to avoid the poles.
\begin{lema}\label{lema7}
The homological generators $\alpha_1$ and $\alpha_2$ deform in the following different ways:
\begin{itemize}
    \item[I.] $\alpha_1$ deforms to the path 
    \begin{equation*}
      \left\lbrace\left(s,-i\sqrt{-q(s)}\right),s\in\left[0,\frac{1}{r}\right]\right\rbrace \bigcup \left\lbrace\left(s,i\sqrt{-q\left(s\right)}\right),s\in\left[\frac{1}{r},0\right]\right\rbrace.
    \end{equation*}
    \item[II.] $\alpha_1$ deforms to the path 
    \begin{equation*}
      \left\lbrace\left(s,-i\sqrt{-q(s)}\right),s\in\left[-\infty,-r\right]\right\rbrace \bigcup \left\lbrace\left(s,i\sqrt{-q\left(s\right)}\right),s\in\left[-r,-\infty\right]\right\rbrace.
    \end{equation*}
    \item[III.] $\alpha_2$ deforms to the path
    \begin{equation*}
        \left\lbrace\left(s,\sqrt{q(s)}\right),s\in\left[-r,0\right]\right\rbrace \bigcup \left\lbrace\left(s,-\sqrt{q\left(s\right)}\right),s\in\left[0,-r\right]\right\rbrace.
    \end{equation*}
    \item[IV.] $\alpha_2$ deforms to the path
    \begin{equation*}
        \left\lbrace\left(s,\sqrt{q(s)}\right),s\in\left[\frac{1}{r},\infty\right]\right\rbrace \bigcup \left\lbrace\left(s,-\sqrt{q\left(s\right)}\right),s\in\left[\infty,\frac{1}{r}\right]\right\rbrace.
    \end{equation*}
\end{itemize}
\end{lema}

\begin{proof}
    Recall from \cite[Section 3.1]{Sarenhu} that $w$ is a function that depends on the copy of $\hat{\bC}$
    \begin{gather*}
        w_1:\hat{\bC}_1\to \bC\quad w_1(z)=\sqrt{|q(z)|}e^{i\frac{\theta_1+\theta_2+\theta_3}{2}},\quad  w_2:\hat{\bC}_2\to \bC\quad w_2(z)=-\sqrt{|q(z)|}e^{i\frac{\theta_1+\theta_2+\theta_3}{2}}.
    \end{gather*}
    For item $I$ notice that $\alpha_1$ has half contained in $\hat{\bC}_1$ and the other half in $\hat{\bC}_2$. Moreover along the real segment $\left[0,\frac{1}{r}\right]$ $q(z)<0$, so $|q(z)|=-q(z)$ and we can write that $\alpha_1$ deforms to $C_1\cup C_2$ where
    \begin{equation*}
        C_1=\left\lbrace\left(s,w_2=-\sqrt{|q(z)|}e^{i\frac{\pi+2\pi+2\pi}{2}}=-i\sqrt{-q(s)}\right),s\in\left[0,\frac{1}{r}\right]\right\rbrace,
    \end{equation*}
    and
    \begin{equation*}
        C_2= \left\lbrace\left(s,w_1=\sqrt{|q(z)|}e^{i\frac{\pi+0+0}{2}}=i\sqrt{-q\left(s\right)}\right),s\in\left[\frac{1}{r},0\right]\right\rbrace.
    \end{equation*}

    The other items are proved in the same way.
\end{proof}
Motivated by the solution of the residue problem of Section \ref{residueproblem}, we implement the anti-holomorphic involution symmetries $A_1,\ A_2$ and $\tau$ in order to investigate the period problem.

\begin{lema}\label{periodlema1}
The action of the symmetries $A_1,\ A_2$ and $\tau$ on the homology basis $\alpha_1$ and $\alpha_2$ of $\overline{M}$ is given by
   \begin{gather*}
    (A_1)_{*}\alpha_1=\alpha_1,\quad (A_1)_{*}\alpha_2=-\alpha_2,\quad (A_2)_{*}\alpha_1=-\alpha_1,\quad (A_2)_{*}\alpha_2=\alpha_2,\quad  (\tau)_{*}\alpha_1=-\alpha_1,\quad (\tau)_{*}\alpha_2=\alpha_2.
\end{gather*}
\end{lema}
\begin{proof}
    Since the genus of the Riemann surface $\overline{M}$ is one, the space of holomorphic differentials $\cH$ is generated by the holomorphic one-form without zeros $\frac{dz}{w}$ defined globally on $\overline{M}$. Notice that the integral of this one-form around any of the homological paths $\alpha_1$ or $\alpha_2$ is the same as the integral along the deformation paths of Lemma \ref{lema7} where we can determine whether the one-form $\frac{dz}{w}$ assumes purely real or purely imaginary values.
    
    We next prove that $(A_1)_*\alpha_1=\alpha_1$. By integration we have
    \begin{gather*}
        \int_{(A_1)_*\alpha_1}\frac{dz}{w}=\int_{\alpha_1}(A_1)^*\frac{dz}{w}=\int_{\alpha_1}\frac{d\overline{z}}{-\overline{w}}=-\overline{\int_{\alpha_1}\frac{dz}{w}}=\int_{\alpha_1}\frac{dz}{w},
    \end{gather*}
    which implies by Lemma \ref{holomorphiconeforms} that $(A_1)_*\alpha_1=\alpha_1$ as we wanted to prove. The other statements follow using the same approach.
\end{proof}
The basis that we obtained for the space $\hat{H}(\mathcal{G}_t)$ diagonalizes the Period Matrix, up to reordering of the elements. We prove this remarkable property using the symmetries $A_1, A_2$ and $\tau$, with a series of lemmas below. We start with

\begin{lema}\label{periodlema2}
For all $t\in(0,1]$ we have
    \begin{gather*}
        \forall j \in \{1,2\}\quad 
            \int_{\alpha_1}\cG_t\eta_j(t)\in \bR,\quad  \int_{\alpha_2}\cG_t\eta_j(t)\in i\bR.
    \end{gather*}
\end{lema}

\begin{proof}
    We use the $A_1$ symmetry.
    \begin{align*}
        \begin{split}
            \int_{\alpha_1}\cG_t\eta_1&=\int_{(A_1)_*\alpha_1}\cG_t\eta_1=\int_{\alpha_1}A_1^*(\cG_t\eta_1)=\int_{\alpha_1}-\overline{\cG_t}A_1^{*}\left(\frac{\sigma_1}{d\cG_t}\right)=-\int_{\alpha_1}\overline{\cG_t}\frac{\overline{\sigma_1}}{-d\overline{\cG_t}}=\overline{\int_{\alpha_1}\cG_t\eta_1}
        \end{split}
    \end{align*}
    which implies that
    \begin{equation*}
        \int_{\alpha_1}\cG_t\eta_1(t)\in \bR.
    \end{equation*}
    The other assertion follows with the same idea.
\end{proof}
\begin{lema}\label{periodlema3}
For all $t\in(0,1]$ we have
    \begin{gather*}
        \forall j \in \{1,2,3\}\ \begin{cases}
            \int_{\alpha_1}\cG_t\eta_j(t)\in i\bR\\ \int_{\alpha_2}\cG_t\eta_j(t)\in \bR
        \end{cases},\quad \forall j\in \{1,2\}\ \begin{cases}
            \int_{\alpha_1}\eta_j(t)\in i\bR\\ \int_{\alpha_2}\eta_j(t) \in \bR
        \end{cases}.
    \end{gather*}
\end{lema}

\begin{proof}
    We use the $A_2$ symmetry.
    \begin{align*}
        \begin{split}
            \int_{\alpha_1}\cG_t\eta_1=\int_{-(A_2)_*\alpha_1}\cG_t\eta_1=-\int_{\alpha_1}A_2^*(\cG_t\eta_1)=-\int_{\alpha_1}\overline{\cG_t}A_2^*\left(\frac{\sigma_1}{d\cG_t}\right)=-\int_{\alpha_1}\overline{\cG_t}\frac{\overline{\sigma_1}}{d\overline{\cG_t}}=-\overline{\int_{\alpha_1}\cG_t\eta_1}
        \end{split}
    \end{align*}
    which implies that
    \begin{equation*}
        \int_{\alpha_1}\cG_t\eta_1(t)\in i\bR
    \end{equation*}
    as we wanted to prove. The other assertions follow the same strategy.
\end{proof}
\begin{lema}\label{periodlema4}
For all $t\in (0,1]$ we have
    \begin{gather*}
         \forall j\in \{1,2\} \quad   \int_{\alpha_j}\eta_1(t)\cG_t^2=\int_{\alpha_j}\eta_1(t),\quad \int_{\alpha_j}\eta_2(t)\cG_t^2=-\int_{\alpha_j}\eta_2(t).
    \end{gather*}
\end{lema}
\begin{proof}
    We use the $\tau$ symmetry.
    \begin{align*}
        \begin{split}
            \int_{\alpha_1}\eta_1\cG_t^2&=\int_{-\tau_*\alpha_1}\eta_1\cG_t^2=-\int_{\alpha_1}\tau^{*}\left(\eta_1\cG_t^2\right)=-\int_{\alpha_1}\tau^{*}\left(\frac{\sigma_1}{d\cG_t}\right)\left(-\frac{1}{\overline{\cG_t}}\right)^2\\&=-\int_{\alpha_1}\frac{\overline{\sigma_1}}{d\left(-\frac{1}{\overline{\cG_t}}\right)}\frac{1}{\overline{\cG_t}^2}=-\int_{\alpha_1}\frac{\overline{\sigma_1}}{\left(\frac{d\overline{\cG_t}}{\overline{\cG_t}^2}\right)}\frac{1}{\overline{\cG_t}^2}=-\overline{\int_{\alpha_1}\eta_1}.
        \end{split}
    \end{align*}
    Therefore, using Lemma \ref{periodlema3} we conclude that
    \begin{equation*}
        \int_{\alpha_1}\eta_1(t)\cG_t^2=\int_{\alpha_1}\eta_1(t).
    \end{equation*}
    The other claims are proved in the same way.
\end{proof}

\begin{lema}\label{extralema}
For all $t\in (0,1]$ we have
    \begin{equation*}
        \forall j \in \{1,2\} \quad \int_{\alpha_j}\eta_3(t)=0,\quad \int_{\alpha_j}\cG_t^2\eta_3(t)=0.
    \end{equation*}
\end{lema}
\begin{proof}
    Indeed, consider the symmetry $$A(z,w)\coloneq A_1\circ A_2(z,w)=(z,-w).$$ Its action on the homology basis $\alpha_1,\ \alpha_2$ is given by
    \begin{equation}\label{Atranformationhomologybasis}
        A_{*}\ \alpha_1=-\alpha_1,\quad A_{*}\ \alpha_2=-\alpha_2,
    \end{equation}
    where we have used Lemma \ref{periodlema1}. On the other hand
    \begin{equation}\label{Ainvarianceeta3}
        A^{*} \eta_3=  A^{*} \left(\frac{P_3(z,t)}{k_t^2(z,t)}dz\right)=\frac{P_3(z,t)}{k_t^2(z,t)}dz=\eta_3,\quad A^{*}\left(\cG^2\eta_3\right)=(\cG\circ A)^2A^{*}\eta_3=(-\cG)^2\eta_3=\cG^2\eta_3.
    \end{equation}
    Therefore, from equations \eqref{Atranformationhomologybasis} and \eqref{Ainvarianceeta3} we obtain for all $j=1,2$
    \begin{equation*}
        \int_{\alpha_j}\eta_3=\int_{-A^{*}\alpha_j}\eta_3=-\int_{\alpha_j}A^{*}\eta_3=-\int_{\alpha_j}\eta_3,\quad \int_{\alpha_j}\cG^2\eta_3=\int_{-A^{*}\alpha_j}\cG^2\eta_3=-\int_{\alpha_j}A^{*}\left(\cG^2\eta_3\right)=-\int_{\alpha_j}\cG^2\eta_3
    \end{equation*}
    which implies
    \begin{equation*}
        \int_{\alpha_j}\eta_3(t)=0,\quad \int_{\alpha_j}\cG_t^2\eta_3(t)=0
    \end{equation*}
    as we wanted to prove.
\end{proof}
\begin{lema}\label{periodlema5}
For all $t\in (0,1]$ we have
\begin{gather*}
    \begin{cases}
        \int_{\alpha_1}\Phi_{l_1,m_2}(t)\in i\bR\\ \int_{\alpha_2}\Phi_{l_1,m_2}(t)\in \bR
    \end{cases} \quad \begin{cases}
        \int_{\alpha_1}\Phi_{l_2,m_1}(t)\in \bR\\ \int_{\alpha_2}\Phi_{l_2,m_1}(t)\in i\bR
    \end{cases}.
\end{gather*}    
\end{lema}
\begin{proof}
    Using Lemma \ref{periodlema3} and Lemma \ref{periodlema4}
    \begin{equation*}
        \int_{\alpha_1}\Phi_{l_1,m_2}(t)=\int_{\alpha_1}(1-\cG_t^2)\eta_2(t)=2\int_{\alpha_1}\eta_2(t)\in i\bR.
    \end{equation*}
    The other claims follow an analogous proof.
\end{proof}
Now, we are in the conditions to prove the main result of this Section
\begin{proof}[{Proof of Theorem \ref{determinantperiodmatrix}}]
By Definition \ref{periodmatrix} and Lemma \ref{periodlema4}
\begin{equation}\label{l1m1}
    \int_{\alpha_j}\Phi_{l_1,m_1}(t)=\int_{\alpha_j}(1-\cG_t^2)\eta_1(t)=0,\quad \int_{\alpha_j}\Phi_{l_2,m_2}(t)=\int_{\alpha_j}i(1+\cG_t^2)\eta_2(t)=0.
\end{equation}
By Lemma \ref{periodlema2} and Lemma \ref{periodlema3}
\begin{equation}\label{Geta1Geta2}
    \int_{\alpha_j}\cG_t\eta_{1}(t)=0,\quad \int_{\alpha_j}\cG_t\eta_2(t)=0.
\end{equation}
By Lemma \ref{periodlema3}
\begin{equation}\label{Geta3alphaj}
    \int_{\alpha_1}\cG_t\eta_3(t)\in i\bR\quad \int_{\alpha_2}\cG_t\eta_3(t)\in \bR.
\end{equation}

By Lemma \ref{periodlema4} and Lemma \ref{periodlema5} we can rewrite
\begin{gather*}
    -\text{Im}\int_{\alpha_1}\Phi_{l_1,m_2}(t)=i\int_{\alpha_1}\Phi_{l_1,m_2}(t)=i\int_{\alpha_1}(1-\cG_t^2)\eta_2(t)=2i\int_{\alpha_1}\eta_2(t),
    \end{gather*}
    \begin{gather*}
\text{Re}\int_{\alpha_1}\Phi_{l_2,m_1}(t)=\int_{\alpha_1}\Phi_{l_2,m_1}(t)=\int_{\alpha_1}i(1+\cG_t^2)\eta_1(t)=2i\int_{\alpha_1}\eta_1(t),\end{gather*}
\begin{gather*} \text{Re}\int_{\alpha_2}\Phi_{l_1,m_2}(t)=\int_{\alpha_2}\Phi_{l_1,m_2}(t)=\int_{\alpha_2}(1-\cG_t^2)\eta_2(t)=2\int_{\alpha_2}\eta_2(t),\end{gather*}
\begin{gather*}-\text{Im}\int_{\alpha_2}\Phi_{l_2,m_1}(t)=i\int_{\alpha_2}\Phi_{l_2,m_1}(t)=i\int_{\alpha_2}i(1+\cG_t^2)\eta_1(t)=-2\int_{\alpha_2}\eta_1(t).
\end{gather*}
Notice that by equation \eqref{Geta3alphaj} we can rewrite
\begin{gather*}
    -\text{Im}\int_{\alpha_1}\Phi_{l_3,m_3}(t)=-2\text{Im}\int_{\alpha_1}\cG_t\eta_3(t)=2i\int_{\alpha_1}\cG_t\eta_3(t),\end{gather*}
    \begin{gather*}\text{Re}\int_{\alpha_2}\Phi_{l_3,m_3}(t)=2\text{Re}\int_{\alpha_2}\cG_t\eta_3(t)=2\int_{\alpha_2}\cG_t\eta_3(t).
\end{gather*}

Using Lemma \ref{extralema} and Lemma \ref{periodlema5}, together with the previous equations, we have that in terms of the symmetric basis of Theorem \ref{Hchapelsymmetricbasis}, for all times $t\in (0,1]$, the Period Matrix is diagonal up to a reordering of its elements
\begin{gather*}
    \mathcal{P}er(\mathcal{G}_t)=\begin{pmatrix}
        0&0&0&0&2i\int_{\alpha_1}\eta_2(t)&0\\2i\int_{\alpha_1}\eta_1(t)&0&0&0&0&0\\0&0&0&0&0&2i\int_{\alpha_1}\cG_t\eta_3(t)\\0&2\int_{\alpha_2}\eta_2(t)&0&0&0&0\\ 0&0&0&-2\int_{\alpha_2}\eta_1(t)&0&0\\ 0&0&2\int_{\alpha_2}\cG_t\eta_3(t)&0&0&0
    \end{pmatrix}.
\end{gather*}

Therefore,
\begin{equation*}
    \det(\mathcal{P}er(\mathcal{G}_t))=64i\int_{\alpha_1}\eta_1(t)\cdot \int_{\alpha_1}\eta_2(t)\cdot \int_{\alpha_1}\cG_t \eta_3(t)\cdot \int_{\alpha_2}\eta_1(t)\cdot \int_{\alpha_2}\eta_2(t)\cdot \int_{\alpha_2}\cG_t\eta_3 (t).
\end{equation*}

\end{proof}

We finish this Section proving that the Period Matrix is non-singular for every $t \in (0,1]$
\begin{thm}\label{determinantperiodmatrixnonzero}
    The nullity of $\mathcal{P}er(\mathcal{G}_t)$ is zero for every $t\in (0,1]$.
\end{thm}
The proof of Theorem \ref{determinantperiodmatrixnonzero} will rely on the following two lemmas

\begin{lema}\label{lem:integralseta3nonzero}
    For every $t \in (0,1]$, $\int_{\alpha_1}\cG_t \eta_3 \neq 0 \text{ and } \int_{\alpha_2}\cG_t \eta_3 \neq 0.
    $ 
        
\end{lema}

\begin{proof}
Let
\begin{equation*}
R^{\alpha_j}_3(p,q,t)\coloneq (p+t)(pt-1)q(1+q^2)^2\left[p^2+2\gamma_{\alpha_j}(-1)^jp-1\right].
\end{equation*}
Using Definition \ref{Omegas} and applying Theorems \ref{theorem1}, \ref{theorem2} we obtain
\begin{gather}\label{integralalpha1Omega3}
    \int_{\alpha_1}\cG\eta_3=2i\left(\frac{r}{1+r^2}\right)^{\frac{1}{2}}F(k_{\alpha_1})\left(\frac{\sqrt{r}}{ab(1+a^2)(1+b^2)k_a(a,t)k_b(b,t)}\right)\left[R_3^{\alpha_1}(a,b,t)p_3(a,t)+R_3^{\alpha_1}(b,a,t)p_3(b,t)\right].
    \end{gather}
    \begin{gather}\label{integralalpha2Omega3}
    \int_{\alpha_2}\cG\eta_3=2\left(\frac{r}{1+r^2}\right)^{\frac{1}{2}}F(k_{\alpha_2})\left(\frac{\sqrt{r}}{ab(1+a^2)(1+b^2)k_a(a,t)k_b(b,t)}\right)\left[R_3^{\alpha_2}(a,b,t)p_3(a,t)+R_3^{\alpha_2}(b,a,t)p_3(b,t)\right].
\end{gather}

We claim that for all $t\in (0,1]$ it holds that
    \begin{equation}\label{p3positive}
          p_3(a,t)>0,\quad p_3(b,t)>0,\quad 
    \end{equation}

Indeed from the Definition \ref{specialpolynomials} we rewrite
    \begin{equation*}
        \frac{p_3(a_t,t)}{a_t}=c_{2D}\left(a_t-\frac{1}{a_t}\right)+c_{2N}=-c_{2D}y_t+c_{2N}
    \end{equation*}
and
\begin{equation*}
      \frac{p_3(b_t,t)}{b_t}=c_{2D}\left(b_t-\frac{1}{b_t}\right)+c_{2N}=-c_{2D}x_t+c_{2N}.
\end{equation*}
Since $a_t,b_t<0$ by Lemma \ref{lem:posicoesdosatbt}, and $c_{2D}(t)<0$ by Lemma \ref{propertiesconstants}, it follows that $p_3(a_t,t),\ p_3(b_t,t)>0$ is equivalent to
\begin{equation*}
    x_t,\ y_t <\frac{c_{2N(t)}}{c_{2D(t)}}.
\end{equation*}
We claim that for all $t\in (0,1]$
\begin{equation}\label{inequalitycconstants}
    c_{2N}(t)<5c_{2D}(t).
\end{equation}
In fact, 
\begin{equation*}
    5c_{2D}(t)-c_{2N}(t)=-1 - 15 r + 
 8 r^2 + (10 - 8 r + 10 r^2) t + (-25 r + 4 r^2) t^2 + (10 - 
    4 r) t^3 + t^4.
\end{equation*}
Using the estimate of Lemma \ref{rLopezNumerical} we see that
\begin{equation*}
    4r^2-25r=r(4r-25)<0,\quad 10-4r=2(5-2r)<0
\end{equation*}
and since $t\in (0,1]$ it follows that
\begin{equation*}
     5c_{2D}(t)-c_{2N}(t)>(-1-15r+8r^2)+(-25r+4r^2)t+(10-4r)t+(10-8r+10r^2)t,
\end{equation*}
that is
\begin{equation*}
     5c_{2D}(t)-c_{2N}(t)>(-1-15r+8r^2)+(20-37r+14r^2)t.
\end{equation*}
Using the estimates of Lemma \ref{rLopezNumerical} we obtain
\begin{equation*}
    -1-15r+8r^2>12>0,\quad (20-37r+14r^2)>15>0
\end{equation*}
and therefore $c_{2N}(t)<5c_{2D}(t)$. Finally, since $c_{2D}(t)<0$ by Lemma \ref{propertiesconstants} and applying Lemma \ref{xtytestimates} we conclude that
\begin{equation*}
    x_t<4<5<\frac{c_{2N}(t)}{c_{2D}(t)},\quad y_t<-1<5<\frac{c_{2N}(t)}{c_{2D}(t)}
\end{equation*}
which implies that $p_3(a_t,t),\ p_3(b_t,t)>0$ as we wanted to prove. 

 Moreover we claim that, for all $t\in (0,1]$
    \begin{equation}\label{R3positive}
      R^{\alpha_1}_3(a,b,t)<0,\quad R^{\alpha_1}_3(b,a,t)<0,\quad R^{\alpha_2}_3(a,b,t)<0,\quad R^{\alpha_2}_3(b,a,t)<0.
    \end{equation}

Indeed, using the relative positions of the points $a_t$, $b_t$ of Lemma \ref{lem:posicoesdosatbt}, these inequalities are equivalent to 
\begin{equation*}
    y_t<-2\gamma_{\alpha_1}<x_t,\quad y_t<2\gamma_{\alpha_2}<x_t.
\end{equation*}
These inequalities are verified by an application of Lemmas \ref{gammaestimatives}, \ref{xtytestimates}
\begin{equation*}
y_t<-1<-0.5<-2\gamma_{\alpha_1}<-0.3<2<x_t,\quad y_t<-1<0<2\gamma_{\alpha_2}<2<x_t,
\end{equation*}
implying that $R^{\alpha_1}_3(a,b,t)<0,\quad R^{\alpha_1}_3(b,a,t)<0$ and $R^{\alpha_2}_3(a,b,t)<0,\quad R^{\alpha_2}_3(b,a,t)<0$ as we wanted to prove.
Finally, substituting equations \eqref{p3positive} and \eqref{R3positive} into equations \eqref{integralalpha1Omega3} and \eqref{integralalpha2Omega3} we conclude that for every $t \in (0,1]$, $\int_{\alpha_1}\cG_t \eta_3 \neq 0 \text{ and } \int_{\alpha_2}\cG_t \eta_3 \neq 0.$
\end{proof}
And finally, the rest of the integrals are also nonzero for every $t \in (0,1]$.

\begin{lema}\label{lem:integralseta12nonzero}
  For every $t \in (0,1]$ and for each $j,k\in\{1,2\}$ it holds that $\int_{\alpha_j} \eta_k \neq 0
    $.
\end{lema}

\begin{proof}
    Consider the integral of $\eta_1$ along $\alpha_1$. Let

\begin{equation*}
    R_1^{\alpha_1}(p,q,t)\coloneq q(r+q)(-1+rq)(1+q^2)^2\left\lbrace
    \begin{aligned}
    &(t^2-\gamma_1r)p^4+\left[(r-\gamma_1)t^2-2(\gamma_1r+1)t-(r-\gamma_1)\right]p^3\\&-2\left[\gamma_1rt^2+2(r-\gamma_1)t-1\right]p^2\\&-\left[(r-\gamma_1)t^2-2(\gamma_1r+1)t-(r-\gamma_1)\right]p+(t^2-\gamma_1r)
    \end{aligned}
    \right\rbrace.
\end{equation*}
    Using Theorem \ref{theorem1} we have that
    \begin{equation*}
        \int_{\alpha_1}\eta_1=2i\left(\frac{r}{1+r^2}\right)^{\frac{1}{2}}F(k_{\alpha_1})\left(\frac{\sqrt{r}t}{a_t^2b_t^2(1+a_t^2)(1+b_t^2)k_a(a_t,t)k_b(b_t,t)k(-r,t)}\right)\left[
        \begin{aligned}
        &R_1^{\alpha_1}(a_t,b_t,t)p_1(a_t,t)\\&+R_1^{\alpha_1}(b_t,a_t,t)p_1(b_t,t)
        \end{aligned}
       \right].
    \end{equation*}

In the expression above, the factor $t/a_t$ is uniformly bounded away from zero by Lemma \ref{lem:limitt/a}. The factor $k_a(a_t,t)$ is also uniformly bounded away from zero which follows from Lemma \ref{lem:limitt/a}, in fact
\begin{align*}
    k_a(a_t,t) &= t (a_t - b_t) \left( a_t + \frac{1}{a_t} \right) \left( a_t + \frac{1}{b_t} \right)\\
    &=(a_t - b_t)\left( t a_t + \frac{t}{a_t} \right) \left( a_t + \frac{1}{b_t} \right).
\end{align*}
And similarly for $k_b(b_t,t)$ and $k(-r,t)$. Therefore, it only remains to prove
\begin{equation}\label{eq:nonzeroeta1alpha1}
    \frac{1}{a_t} (R^{\alpha_1}_1(a_t,b_t,t) p_1(a_t,t) + R^{\alpha_1}_1(b_t,a_t,t) p_1(b_t,t)) \neq 0.
\end{equation}
To prove this, we use the strategy presented in the Appendix \ref{section::EstimatesPolynomials}. We describe now how to transform \eqref{eq:nonzeroeta1alpha1} into an expression for which we can use the aforementioned strategy. The numerator in \eqref{eq:nonzeroeta1alpha1} is a polynomial in $a_t$, $b_t$, $t$, $r$ and $\gamma_1$ that we call $\text{Num}_{\alpha_1, \eta_1}(a_t, b_t, t, r, \gamma_1)$. 

One can verify that every monomial of $\text{Num}_{\alpha_1,\eta_1}(a_t, b_t, t, r, \gamma_1)$ without the variable $a_t$ has the variable $t$ (see the companion Mathematica notebook). That is, $\text{Num}_{\alpha_1,\eta_1}(0, b_t, t, r, \gamma_1) = t \cdot Q_{\alpha_1,\eta_1}(b_t, t, r, \gamma_1)$. Dividing by $a_t$ we obtain a polynomial $\tau_t \cdot Q_{\alpha_1,\eta_1}(b_t, t, r, \gamma_1) $ with the additional variable $\tau_t:= t/a_t$. We perform the division by $a_t$ in the monomials that contain $a_t$ as variable and add the resulting polynomial to $\tau_t \cdot Q_{\alpha_1,\eta_1}(b_t, t, r, \gamma_1)$ to define $\tilde P_{\alpha_1, \eta_1}(a_t, b_t, \tau_t,t,r,\gamma_1)$. Observe that for any $t \in (0,1]$, $\tilde{P}_{\alpha_1, \eta_1}(a_t,b_t,\tau_t,t,r,\gamma_1)$ is equal to the expression \eqref{eq:nonzeroeta1alpha1}. Finally, we define a new polynomial in order to use nonnegative variables
\begin{align*}   
P_{\alpha_1,\eta_1}(|a_t|, |b_t|, |\tau_t|, t, r, \gamma_1) := \tilde P_{\alpha_1, \eta_1}(-|a_t|, -|b_t|, -|\tau_t|,t, r, \gamma_1)= \tilde{P}_{\alpha_1,\eta_1}(a_t,b_t,\tau_t,t,r,\gamma_1).
\end{align*}
We prove that $P_{\alpha_1,\eta_1}$ is negative in $[0,1]$ using the strategy described in Appendix \ref{section::EstimatesPolynomials}, adapted to the case where we want to prove negative sign. We use the construction of piecewise linear functions estimating $a_t$, $b_t$ and $\tau_t$ from above and from below described in Lemma \ref{lem:estimatesforabtau}, and the explicit estimates for $r$ and $\gamma_1$ given by Lemmas \ref{rLopezNumerical} and \ref{gammaestimatives} to construct the $\Theta$ function for this problem, analogous to \eqref{eq:thetageral}, and prove that for some number of sub-intervals $n$, it holds
\begin{equation} \label{theta-alpha1-eta1}            \Theta(P_{\alpha_1,\eta_1};|a|^{\inf}_t,|a|^{\sup}_t,|b|^{\inf}_t,|b|^{\sup}_t,|\tau|^{\inf}_t,|\tau|^{\sup}_t,r^{\inf},r^{\sup},\gamma_1^{\inf},\gamma_1^{\sup},n) < 0.
\end{equation}
Therefore the integral of $\eta_1$ along $\alpha_1$ is nonzero. We refer to the companion Mathematica notebook for the explicit computations.

Consider the integral of $\Omega_1$ along $\alpha_2$. Let
\begin{equation*}
    R_1^{\alpha_2}(p,q,t)\coloneq q(r+q)(-1+rq)(1+q^2)^2\left\lbrace
    \begin{aligned}
    &(t^2+\gamma_2r)p^4+\left[(r+\gamma_2)t^2-2(-\gamma_2r+1)t-(r+\gamma_2)\right]p^3\\&-2\left[-\gamma_2rt^2+2(r+\gamma_2)t-1\right]p^2\\&-\left[(r+\gamma_2)t^2-2(-\gamma_2r+1)t-(r+\gamma_2)\right]p+(t^2+\gamma_2r)
    \end{aligned}
    \right\rbrace
\end{equation*}
Using Theorem \ref{theorem2} we have
\begin{equation*}
    \int_{\alpha_2}\eta_1=2\left(\frac{r}{1+r^2}\right)^{\frac{1}{2}}F(k_{\alpha_2})\left(\frac{\sqrt{r}t}{a_t^2b_t^2(1+a_t^2)(1+b_t^2)k_a(a_t,t)k_b(b_t,t)k(-r,t)}\right)\left[
    \begin{aligned}
         &R_1^{\alpha_2}(a_t,b_t,t)p_1(a_t,t)\\&+R_1^{\alpha_2}(b_t,a_t,t)p_1(b_t,t)
    \end{aligned}
   \right],
\end{equation*}  
Similarly, this integral being nonzero reduces to prove
\begin{equation}\label{eq:conditionalpha2eta1}
    \frac{1}{a_t} (R_1^{\alpha_2}(a_t,b_t,t) p_1(a_t,t) + R^{\alpha_2}_1(b_t,a_t,t) p_1(b_t,t)) \neq 0.
\end{equation}
For which we can follow the same procedure as before to conclude that left hand side of \eqref{eq:conditionalpha2eta1} is strictly negative for $t \in (0,1]$.
Thus, the integral of $\eta_1$ along $\alpha_2$ is nonzero for every $t \in (0,1]$. 

Consider the integral of $\Omega_2$ along $\alpha_1$. Let 
    \begin{equation*}
        R_2^{\alpha_1}(p,q,t)\coloneq (1+q^2)q(q+r)(-1+qr)\left\lbrace(\gamma_1r+t^2)p^2+\left[(r-\gamma_1)t^2+2(\gamma_1r-1)t+r-\gamma_1\right]p-(\gamma_1r+t^2)\right\rbrace
    \end{equation*}
Using Theorem \ref{theorem1} we have that
\begin{equation*}
\int_{\alpha_1}\eta_2=2i\left(\frac{r}{1+r^2}\right)^{\frac{1}{2}}F(k_{\alpha_1})\left(\frac{\sqrt{r}t}{a_t^2b_t^2k_a(a_t,t)k_b(b_t,t)k(-r,t)}\right)\left[R_2^{\alpha_1}(a_t,b_t,t)p_2(a_t,t)+R_2^{\alpha_1}(b_t,a_t,t)p_2(b_t,t)\right]
    \end{equation*}
    As before, this reduces to prove
\begin{equation}\label{eq:conditionperiodalpha1eta2}
    \frac{1}{a_t}(R^{\alpha_1}_2(a_t,b_t,t) p_2(a_t,t) + R^{\alpha_1}_2(b_t,a_t,t) p_2(b_t,t)) \neq 0.
\end{equation}
We follow the same procedure described in the last cases and prove that this expression is negative in $[0,t]$

Finally, consider the integral of $\Omega_2$ along $\alpha_2$. Let

  \begin{equation*}
        R_2^{\alpha_2}(p,q,t)\coloneq (1+q^2)q(q+r)(-1+qr)\left\lbrace(-\gamma_2r+t^2)p^2+\left[(r+\gamma_2)t^2-2(\gamma_2r+1)t+r+\gamma_2\right]p-(-\gamma_2r+t^2)\right\rbrace
    \end{equation*}
    Using Theorem \ref{theorem2} we have
 \begin{gather*}
\int_{\alpha_2}\eta_2=2\left(\frac{r}{1+r^2}\right)^{\frac{1}{2}}F(k_{\alpha_2})\left(\frac{\sqrt{r}t}{a^2b^2k_a(a_t,t)k_b(b_t,t)k(-r,t)}\right)\left[R_2^{\alpha_2}(a_t,b_t,t)p_2(a_t,t)+R_2^{\alpha_2}(b_t,a_t,t)p_2(b_t,t)\right]
    \end{gather*}
Then, following the same scheme as before, we can conclude that
\begin{equation*}
    \frac{1}{a_t} \left(R^{\alpha_2}_2(a_t,b_t,t) p_2(a_t,t) + R^{\alpha_2}_2(b_t,a_t,t) p_2(b,t)\right) \neq 0.
\end{equation*}
To conclude, we follow the same procedure as in the last integrals to conclude that this expression is positive in $(0,1]$. 
\end{proof}

We now use the previous lemmas to prove Theorem \ref{determinantperiodmatrixnonzero}.

\begin{proof}[Proof of Theorem \ref{determinantperiodmatrixnonzero}]
It follows from Lemmas \ref{lem:integralseta3nonzero}, \ref{lem:integralseta12nonzero} and Theorem \ref{determinantperiodmatrix} that the determinant of the Period Matrix $\text{Per}(\mathcal{G}_t)$ is nonzero for any $t \in (0,1]$, in particular, its nullity is zero in this interval.
\end{proof}

Therefore, we have solved the period problem. In conclusion, we have proved that $H(\mathcal{G}_t)=0$ for every $t \in (0,1]$.

\subsubsection{Nullity of the meromorphic function $\mathcal{G}_t$ for $t \in [0,1]$}

We are now in position to gather the information derived in the previous Sections to compute the nullity of the Schr\"odinger operators associated to each element of the family $\{\mathcal{G}_t\}_{t \in [0,1]}$ of meromorphic functions defined in equation \eqref{eq:deformapadegauss}. 
\begin{cor}\label{cor:nullityofallGt}
    The nullity of the meromorphic functions $\cG_t:\Sigma\to \bS^2$ is three, for all $t\in [0,1]$.
\end{cor}

\begin{proof}

As presented in Section \ref{sec:indexandgaussmap}, Montiel and Ros proved in \cite{MontielRos} that the nullity of a meromorphic function is always at least three, since the three coordinate functions of the corresponding map to the sphere are Jacobi functions. Finally, the nullity cannot be greater than three, since we proved in Theorem \ref{determinantperiodmatrixnonzero} that $H(\mathcal{G}_t) = 0$ for all $t \in (0,1]$, and this also holds for $t=0$, as follows from Lemma \ref{branchvaluesG0} and \cite[Corollary 15]{MontielRos}. 
\end{proof}

\subsection{Proof of Theorem \ref{thm-index-lopez}}\label{sec:proofB}
We collect the results of the previous Sections to give a streamlined proof of Theorem \ref{thm-index-lopez}.

\begin{proof}[Proof of Theorem \ref{thm-index-lopez}]

As concluded in Corollary \ref{cor:nullityofallGt}, all maps $\mathcal{G}_t$ for $t \in [0,1]$ have nullity equal to three. Recall that $\mathcal{G}_1$ corresponds to the Gauss map of the double cover of the López Klein bottle and that $\mathcal{G}_0$ has odd Morse index and nullity both equal to three, by Theorem \ref{thm:indexofnonorientable} and Lemma \ref{branchvaluesG0}. Thus, since the family $\mathcal{G}_t$ is continuous in the $C^1$ topology, we have that the eigenvalues associated to odd eigenfunctions of $-\Delta_{\mathcal{G}_t} - 2$ vary continuously. Therefore, since the nullity, and in particular the odd nullity of each operator is constant, the number of negative eigenvalues of $-\Delta_{\mathcal{G}_t} - 2$ associated to odd eigenfunctions must be constant. This means that the odd Morse index and the odd nullity must be the same as for $\mathcal{G}_0$, from which the theorem follows.
\end{proof}
\subsection{Proof of Corollary \ref{coro-index-8pi}}\label{indexthreecorollary}We conclude this Section with a proof of Corollary \ref{coro-index-8pi}.
\begin{proof}[Proof of Corollary \ref{coro-index-8pi}]
    Let $X:M\to \bR^3$ be a complete minimal immersion of finite total curvature $8\pi$. We have two cases to consider.
    
    If $M$ is orientable, then $M$ is either a genus zero surface \cite[Theorem 3, 4, 5, 6]{Lopezorientable8pi} or the Chen-Gackstatter surface, by the López classification result \cite[Corollary 1]{Lopezorientable8pi}. Notice that the Gauss map of all these minimal surfaces has degree two. On one hand, for the genus zero minimal surfaces, the maximum number of branch points of their Gauss map is at most two, by the Riemann-Hurwitz Theorem \ref{RiemannHurwitz}, so the branch values of the Gauss map belong to an equator of the sphere. On the other hand, it is known that the Chen-Gackstatter surface has Morse index three (see \cite{ChodoshMaximotopandindexII,MontielRos}).

    If $M$ is nonorientable, then $M$ is the López minimal Klein bottle by the characterization result of López  \cite{LopezUniqueness}. Therefore the Morse index of $M$ is three by Theorem \ref{thm-index-lopez}.
\end{proof}

\section{Minimal surfaces with Morse index two immersed in Euclidean three-space}\label{section::Minimal surfaces with Morse index two immersed in Euclidean three-space}

  The main goal of this Section is to restrict the topology and geometry of a complete minimal surface with index $2$ immersed in $\mathbb{R}^3$ and establish Theorem \ref{thm-classification-index-two}. Our proof builds strongly on the work of Chodosh-Maximo \cite{ChodoshMaximotopandindexII}, as we use their results and the ideas behind the proof of their Theorem 1.5. 

\begin{proof}[Proof of Theorem \ref{thm-classification-index-two}]

 Chodosh and Maximo proved the non-existence of orientable complete minimal surfaces immersed in $\mathbb{R}^3$ with Morse index 2 \cite{ChodoshMaximotopandindexII}. It is then enough to consider nonorientable complete minimal immersions. Suppose $X: M \to \mathbb{R}^3$ is a complete, nonorientable minimal immersion with Morse index two, where $M=\Sigma\setminus \{p_1,\ldots, p_r\}$ and $\Sigma$ compact. We denote by $g$ the genus of the orientable double cover. Applying the estimate \eqref{eq:chodosh-maximo2} of Chodosh-Maximo, we have:
\begin{equation}\label{eq:chodosh2forindex2}
    \frac{8 \pi}{3} \leq \int_{M} (-\kappa) \leq 10 \pi.
\end{equation}
Recall that the total curvature of a complete nonorientable immersion is an integer multiple of $2 \pi$, and by \cite{Meeks} and  \cite{LopezUniqueness}, the only nonorientable examples with total curvature at most $8 \pi$ are the Meeks minimal M\"obius band and the López minimal Klein bottle which have Morse index two and three, respectively, by Theorems \ref{thm-Meeks-indice} and \ref{thm-index-lopez}. Thus, by \eqref{eq:chodosh2forindex2} it remains to study the case of total curvature $10 \pi$. Now, applying the inequality \eqref{eq:chodosh-maximo4} of Chodosh-Maximo, we have:
\begin{equation}\label{inequalitycontradiction}
    10 -g \geq 2 \sum_{i=1}^r(d_j+1) \geq 4 r \geq 4.
\end{equation}
On the other hand, applying the Jorge-Meeks formula \eqref{jogemeeksformulaonesided} in the nonorientable case, we have
\begin{equation}\label{eq:jorgemeekscurvatura10pi}
    g - 1 + \sum_{i=1}^r (d_i + 1) = \frac{\int_{\Sigma} (-\kappa)}{2 \pi} = 5.
\end{equation}
Thus, we have
\begin{equation*}
    6 - g \geq \sum_{i=1}^r (d_i+1) \geq 2r \geq 2.
\end{equation*}
Then, we conclude that the genus satisfies $0 \leq g \leq 4$. First, consider the genus four case. Using the Jorge-Meeks formula \eqref{eq:jorgemeekscurvatura10pi} we conclude:
\begin{equation*}
    \sum_{i=1}^r (d_i+1) = 2.
\end{equation*}
The last equation is only possible if $r=1$ and $d_1 = 1$, which means that the orientable double cover has two ends with multiplicity one, which means they are embedded, see \cite{jorge1983topology}, \cite{MartinThesis}. Then by a theorem of Schoen \cite{schoen1983uniqueness} the orientable double cover must be the catenoid, which is impossible. For the genus three case, using the Meeks equation \eqref{eq:meeksmod2} we would have $5 \equiv -2\ (\text{mod } 2)$, which gives a contradiction.

For the genus two case, we have by Jorge-Meeks formula \eqref{eq:jorgemeekscurvatura10pi} that:
\begin{equation*}
    \sum_{i=1}^r(d_i+1) = 4.
\end{equation*}
The case $r=2$ and $d_1 = d_2 = 1$ is impossible by a theorem of Kusner \cite{kusner1987conformal} saying that there is no complete, nonorientable minimal surface in $\mathbb{R}^3$ with two embedded ends. The other possibility is $r=1$ and $d_1 = 3$, which is compatible with equations \eqref{eq:meeksmod2} and \eqref{eq:chodosh2forindex2} .

If the genus is one, then we have from the Jorge Meeks formula:
\begin{equation*}
    \sum_{i=1}^r (d_i + 1) = 5,
\end{equation*}
which contradicts equation \eqref{inequalitycontradiction}.

Finally, the remaining case is having genus zero, from which we obtain again a contradiction putting equation \eqref{eq:jorgemeekscurvatura10pi} into \eqref{inequalitycontradiction}. This completes the proof of Theorem \eqref{thm-classification-index-two}.
\end{proof}

\section{The index of a minimal oriented double cover immersed in Euclidean three-space} \label{section::The index of a minimal oriented double cover immersed in Euclidean three-space}

We investigate the sharpest lower bound on the index of a minimal oriented double cover. The goal of this Section is to prove Theorem \ref{thm-classification-oriented-double-cover}. As in Section \ref{section::Minimal surfaces with Morse index two immersed in Euclidean three-space}, our approach builds strongly on the work of Chodosh and Maximo. We divide the proof of Theorem \ref{thm-classification-oriented-double-cover} into three theorems, which rely on different ideas and techniques. Our first theorem is a direct application of the ideas of Montiel and Ros \cite{MontielRos}.

\begin{thm}\label{indexdoublecoverMeeks}
     The oriented double cover of the Meeks minimal M\"obius band has Morse index $5$.
\end{thm}

\begin{proof}
    The Gauss map of the double cover of Meeks minimal M\"obius band has its branch values contained in an equator of $\bS^2$, as shown in the proof of Theorem \ref{thm-Meeks-indice} (see Section \ref{section::The Meeks minimal M\"obius band and the Oliveira family of minimal M\"obius bands}). Hence, we conclude that the double cover of the Meeks minimal M\"obius band has index $5$ as an application of Corollary 15 in \cite{MontielRos}. 
\end{proof}
\begin{thm}\label{indexdoublecoverLopez}
    The oriented double cover of the López minimal Klein bottle has Morse index $7$.
\end{thm}
\begin{proof}
    The Gauss map of the oriented double cover of the López minimal Klein bottle $\cG_1$ can be connected by a $C^1$-continuous one-parameter family of meromorphic functions $\{\cG_t\}_{t\in[0,1]}$, as in Lemma \ref{deformationpath}, in such a way that $\cG_0$ is also a meromorphic function of the same degree, $4$, but has branch values contained on an equator (see Lemma \ref{branchvaluesG0}). Hence, $\cG_0$ has Morse index $7$ and nullity $3$ by \cite[Corollary 15]{MontielRos}. 
    
    By Theorem \ref{determinantperiodmatrixnonzero} the nullity of $\cG_t$ is constant along the path $\{\cG_t\}_{t\in[0,1]}$, therefore the oriented double cover of the López minimal Klein bottle also has Morse index $7$.
\end{proof}

We now study more general minimal oriented double covers.

\begin{thm}
    Every complete minimal oriented double cover immersed in Euclidean three-space has Morse index at least $5$.
\end{thm}
\begin{proof}
    Let us consider a nonorientable complete minimal surface $\Sigma$ immersed in $\mathbb{R}^3$ and let $\tilde \Sigma$ denote its oriented double cover with genus $g$. Let $r$ be the number of ends of $\Sigma$ and $\tilde r$ the number of ends of $\tilde \Sigma$. According to \cite[Theorem 1.1]{ChodoshMaximotopandindexII} we have
\begin{equation}\label{estimateoneindex}
    \text{index}(\tilde \Sigma) \ge \frac{1}{3}(2g+ 2 \sum_{j=1}^{\tilde r} (d_j + 1) - 5 ) = \frac{1}{3}(2g+ 4 \sum_{j=1}^{ r} (d_j + 1) - 5 )
\end{equation}

\noindent where we have used that every end of $\Sigma$ is repeated in $\tilde \Sigma$ twice with the same multiplicity \cite{MartinThesis}. 

Now we use the index to estimate the total curvature of $\tilde \Sigma$. According to \cite[Theorem 1.10]{ChodoshMaximotopandindexII} and since $$\int_{\tilde \Sigma} (-\kappa) = 2 \int_\Sigma (-\kappa) $$ we find

\begin{equation}\label{estimatetwoindex}
    \frac{1}{3} + \frac{1}{3\pi} \int_{\Sigma} (-\kappa) \le \text{index}(\tilde \Sigma) \le -3 + \frac{3}{\pi} \int_{ \Sigma} (-\kappa).
\end{equation}

Note that orientable complete minimal surfaces immersed in Euclidean three-space with Morse index at most two are completely classified, as mentioned in the introduction (\cite{DoCarmo}, \cite{fischerschoen1980structure},\cite{pogorelov1981stability},\cite{LopezRosWeakly}). These minimal surfaces are not minimal oriented double covers, since the catenoid is embedded and the Enneper surface and the plane only have one end. Moreover there does not exist a complete orientable minimal immersion in $\bR^3$ with index two \cite[Theorem 1.5]{ChodoshMaximotopandindexII}

We claim that $\text{index}(\tilde \Sigma)\neq 3$. Indeed, suppose by contradiction that $\text{index}(\tilde \Sigma)= 3$, then using \eqref{estimatetwoindex} we obtain
\begin{equation*}
      2\pi \le \int_\Sigma (-\kappa) \le 8\pi.
\end{equation*}

Using the classification results of Meeks \cite[Section 4]{Meeks} and López \cite[Section 3]{LopezUniqueness} we find that $\Sigma$ is either the Meeks M\"obius band, or the López minimal Klein bottle. However, the oriented double cover of the Meeks minimal M\"obius band has index $5$ by Theorem \ref{indexdoublecoverMeeks} and the oriented double cover of the López minimal Klein bottle has index $7$ by Theorem \ref{indexdoublecoverLopez}. Therefore $\text{index}(\tilde \Sigma)\neq 3$.

Finally, in order to conclude the proof, we claim that $\text{index}(\tilde \Sigma)\neq 4$. In fact, suppose by contradiction that $\text{index}(\tilde \Sigma)= 4$. On one hand, evaluating equation \eqref{estimatetwoindex} we obtain
\begin{equation*}
      \frac{7}{3}\pi \le \int_\Sigma (-\kappa) \le 11\pi.
\end{equation*}
Therefore, using Meeks Theorem \cite[Section 4]{Meeks} and the fact that the total curvature must be even, we have
\begin{equation*}
    6\pi \le \int_\Sigma (-\kappa) \le 10\pi.
\end{equation*}
By the same reasoning when we excluded the index three case, it remains to analyze the case when 
\begin{equation*}
    \int_{\Sigma}(-\kappa)=10\pi
\end{equation*}
Notice that evaluating equation \eqref{estimateoneindex} we obtain after simplification that
\begin{equation}\label{restricao topologica para double cover of index 5}
    8 \ge g+ 2 \sum_{j=1}^{ r} (d_j + 1)
\end{equation}
On the other hand, applying the Jorge-Meeks formula for the nonorientable case \eqref{jogemeeksformulaonesided} we obtain
  \begin{equation}\label{Jorge-Meeks para one-sided de curvatura 10pi}
    6 = g +\sum_{j=1}^r(d_j+1).
\end{equation}
Putting together equations \eqref{restricao topologica para double cover of index 5} and \eqref{Jorge-Meeks para one-sided de curvatura 10pi} we have

\begin{equation}
    8 \ge g+ 2 \sum_{j=1}^{ r} (d_j + 1) = 6 + \sum_{j=1}^r(d_j+1)
\end{equation}
so 
\begin{equation*}
    \sum_{j=1}^r(d_j+1) \le 2
\end{equation*}

\noindent which implies $r=1$ and $d_1 = 1$. But this is excluded by R. Schoen's characterization of the catenoid \cite{schoen1983uniqueness}.  We conclude that $\text{index}(\tilde \Sigma)\neq 4$.

Therefore $\text{index}(\tilde \Sigma)\geq 5$, as we wanted to prove.
\end{proof}

It is clear now that Theorem \ref{thm-classification-oriented-double-cover} follows directly from the three results established in this Section. Theorem \ref{indexdoublecoverMeeks} shows that our estimate on the index of minimal double covers is sharp.

\appendix
\section{Riemann surfaces results} \label{section::Riemann surfaces results}

\begin{defn}[{\cite[Definition I.1.5-6]{MartinThesis}}]\label{multiplicity}
\normalfont{
   Let $F:M\to N$ be a non-constant holomorphic map between two Riemann surfaces. Consider charts $(U,\psi)$ of $M$ and $(V,\phi)$ of $N$ such that $\psi(0)=P\in M$ and $\phi(0)=F(P)\in N$, which we call centered at $P$ and $F(P)$ respectively. Consider the holomorphic map
   \begin{equation*}
       \hat{F}\coloneqq \phi^{-1}\circ F\circ \psi:U\subset \bC\to V\subset \bC.
   \end{equation*} 
   Since $\hat{F}(0)=0$, there exists a natural number $n\in \bN$ and a holomorphic function $H(z)$ with $H(0)\neq 0$ such that
\begin{equation*}
   \hat{F}(z)=z^nH(z).
\end{equation*}
The integer $n$ does not depend on the coordinate charts and it is called \textit{the multiplicity} $mult_{P}(F)$ of the function $F$ at the point $P$. We define \textit{the branch order} of $F$ at $P$ as $r_{P}(F)=mult_{P}(F)-1$. If $r_{P}(F)\geq 1$ then $P$ is called a \textit{branch point} of $F$, and its image $F(P)$ a \textit{branch value}.
}
\end{defn}

\begin{thm}[{Riemann-Hurwitz}]\cite[Theorem 1.76]{RiemannSurfacesGirondo}\label{RiemannHurwitz}
    Let $F:M_1\to M_2$ be a non-constant holomorphic map between Riemann surfaces. Then
    \begin{equation*}
        \cX(M_1)=\deg F\cdot \cX(M_2)-\sum_{P_1\in M_1}r_{P_1}(F).
    \end{equation*}
\end{thm}
The following proposition follows using the same argument as \cite[Proposition 4.4.7]{CarlosThesis}
\begin{prop}\label{BranchPointsComeInPairs}
   Let $\phi:\Sigma\to \bS^2$ be a holomorphic map on a compact Riemann surface $\Sigma$. Suppose that there exists an anti-holomorphic involution $\tau:\Sigma\to \Sigma$ such that $\phi\circ \tau=-\phi$. Then 
    \begin{equation*}
       \forall p\in \Sigma,\quad  r_p(g)=r_{\tau(p)}(g).
    \end{equation*}
\end{prop}

\begin{prop}\label{residuetool}
    Let $\beta_i\neq \beta_j\ \forall i\neq j$ and $Q_j$ polynomials with roots different than the $\beta_i$'s and consider the one-form
    \begin{equation*}
        \alpha\coloneq \frac{\Pi_{j=1}^m Q_j(z)}{\Pi_{j=1}^n(z-\beta_j)^2}dz
    \end{equation*}
    Then for any $k$
    \begin{equation*}
        \text{Res}_{\beta_k}\alpha=\frac{\Pi_{j=1}^m Q_j(\beta_k)}{\Pi_{j=1,\ j\neq k}^n(\beta_k-\beta_j)^{2}}\left(\sum_{j=1}^m\frac{Q_j'(\beta_k)}{Q_j(\beta_k)}-\sum_{j=1,\ j\neq k}^n\frac{2}{\beta_k-\beta_j}\right)
    \end{equation*}
\end{prop}
\begin{proof}
  Consider the function $$F(z)=\prod_{j=1}^nf_j(z)^{a_j}$$ where $f_j(z)$ $j=1,\ldots n$ are complex functions and $a_j\in \bZ$. Then applying Leibniz's rule recursively we obtain
\begin{equation}\label{logderivative}
F'(z)=\sum_{j=1}^nf_1(z)^{a_1}\ldots a_jf_j(z)^{a_j-1}f_j'(z)\ldots f_n(z)^{a_n}=F(z)\sum_{j=1}^na_j\frac{f'_j(z)}{f_j(z)}
\end{equation}
  Notice that the $\alpha$ one-form has a pole of order two at each $\beta_k$ since $Q_j(\beta_k)\neq 0$ by hypothesis. Therefore we obtain
  $$\operatorname{Res}_{z=\beta_k}\alpha=\eval{\frac{d}{dz}}_{\beta_k}\prod_{j=1}^mQ_j(z)\prod_{\substack{j=1\\j\neq k}}^n(z-\beta_j)^{-2}$$
  by an application of the residue formula \cite[Theorem 1.4]{Stein}. We conclude the proof by evaluating equation \eqref{logderivative} at $z=\beta_k$ for the function
$$F(z)=\prod_{j=1}^mQ_j(z)\prod_{\substack{j=1\\j\neq k}}^n(z-\beta_j)^{-2}.$$
\end{proof}
The following Lemma is a useful tool that we learned from Matthias Weber.
\begin{lema}\label{polynomialidentity}
   Let $\beta_1,\ldots \beta_n$ be distinct complex numbers. Consider the polynomial $$K(z)=\prod_{k=1}^n(z-\beta_k).$$ Then for all $1\leq k\leq n$
    \begin{equation*}
    K''(\beta_k)=2K'(\beta_k)\sum_{\substack{j=1\\j\neq k}}^n\frac{1}{(\beta_k-\beta_j)}.
    \end{equation*}
\end{lema}
\begin{proof}
    By an application of \eqref{logderivative} we find that
    \begin{equation*}
        K'(z)=\frac{K(z)}{(z-\beta_j)}+K(z)\sum_{\substack{k=1\\k\neq j}}^n\frac{1}{(z-\beta_k)},
    \end{equation*}
    since the polynomial $K(z)$ has simple zeros by hypothesis. Computing one more derivative and applying a Taylor expansion around $\beta_j$ we have that
    \begin{equation*}
         K''(z)=\frac{1}{2}K''(\beta_j)+(z-\beta_j)Q(z)+K'(z)\sum_{\substack{k=1\\k\neq j}}^n\frac{1}{(z-\beta_k)}-K(z)\sum_{\substack{k=1\\k\neq j}}^n\frac{1}{(z-\beta_k)^2},
    \end{equation*}
     where $Q(z)$ is a polynomial. The result then follows by evaluation at the point $z=\beta_j$.
\end{proof}
\begin{lema}\cite[Lemma 4.2.11]{CarlosThesis}\label{Residuerelatedpoints}
    Let $M$ be a Riemann surface and $\gamma$ an anti-holomorphic involution on $M$. Then for any meromorphic one-form $\omega$ in $M$ and any point $P\in M$,
    \begin{equation*}
        \eval{Res}_{P}\overline{\gamma^{*}\omega}=\overline{\eval{Res}_{\gamma(P)}\omega}.
    \end{equation*}
\end{lema}
\begin{lema}\label{holomorphiconeforms}(\textit{Cf.} \cite[Lemma I.1.2]{MartinThesis})
    Let $\cH$ be the space of holomorphic one-forms on a compact Riemann surface $M$ and fix $\Gamma\in H_1(M;\bZ)$. If for all $\omega \in\cH$ it holds that $\int_\Gamma\omega=0$, then $\Gamma=0$.
\end{lema}
 We introduce the spaces that appear in the Riemann-Roch theorem (see \cite[Section III.4]{Farkas}). Fixing a divisor $D$ in a compact Riemann surface $\Sigma$, we define the vector space of meromorphic functions 
\begin{equation*}
    \cL(D)\coloneq \left\lbrace f:\Sigma\to \overline{\bC}\mid  \operatorname{div}(f)+D\geq 0\right\rbrace\cup \{0\}
\end{equation*}
Notice that for a holomorphic map $\phi:\Sigma\to \bS^2$, there is a clear isomorphism between $\cL(2K_{\Sigma}+R(\phi))$ and the space $H^{0,2}(\phi)$ defined on Section \ref{sec:indexandgaussmap} given by
\begin{equation}\label{isomorphismLspaceH0space}
    f\in \cL(2K_{\Sigma}+R(\phi))\longrightarrow f\left(\frac{dz}{w}\right)^2\in H^{0,2}(\phi)
\end{equation}
We also define the following vector space of meromorphic one-forms on $\Sigma$
\begin{equation*}
     \cJ(D)\coloneq \left\lbrace \omega\ \text{meromorphic 1-forms}\mid  \operatorname{div}(\omega)\geq D\right\rbrace\cup \{0\}
\end{equation*}
We denote by $l(D)$ and $i(D)$ the dimensions of the spaces $\cL(D)$ and $\cJ(D)$, respectively.

The $\cJ(D)$ space is trivial when the divisor $D$ has high degree

\begin{lema}\label{highdegreedivisor}
Let $M$ be a compact Riemann surface of genus $g$. If the divisor $D$ has degree $deg(D)>2g-2$ then $i(D)=0$.
\end{lema}
\begin{proof}
    Indeed suppose by contradiction that there is $\omega\in \cJ(D)\setminus\{0\}$. By definition
    \begin{equation*}
        \operatorname{div}(\omega)\geq D.
    \end{equation*}
    The degree of any meromorphic one-form $\eta$ in $M$ is
    \begin{equation*}
        \deg(\operatorname{div}(\eta))=2g-2
    \end{equation*}
    Therefore, we arrive at the contradiction
    \begin{equation*}
        2g-2=deg(\operatorname{div}(\omega))\geq deg(D)>2g-2,
    \end{equation*}
    which shows that $\cJ(D)=\{0\}$ is the trivial space.
\end{proof}

We conclude this appendix with the classical Riemann-Roch theorem
\begin{thm}[{Riemann-Roch}]\cite[\textit{Cf}. Section III.4.8]{Farkas}\label{RiemannRoch}
Let $D$ be any divisor in a compact Riemann surface of genus $g$. Then
    \begin{equation*}
        l(D)-i(D)=1-g+deg(D)
    \end{equation*}
\end{thm}

\section{Auxiliary computations with elliptic integrals} \label{section::Auxiliary computations with elliptic integrals}

We calculate the six fundamental integrals
\begin{equation*}
    \int_{\alpha_1}\eta_1,\ \int_{\alpha_1}\eta_2,\ \int_{\alpha_1}\cG \eta_3,\ \int_{\alpha_2}\eta_1,\ \int_{\alpha_2}\eta_2,\ \int_{\alpha_2}\cG\eta_3 
\end{equation*}
that appear in the determinant of the Period Matrix of Theorem \ref{determinantperiodmatrix}.
\begin{defn}\cite[Eq. 110.02-0.3]{handbook}\label{ellipticintegrals}
\normalfont
    We introduce the definition of the \textit{elliptic integrals of first and second kind}
    \begin{equation*}
        F(\varphi,k)\coloneqq \int_0^{\varphi}\frac{d\theta}{\sqrt{1-k^2\sin^2\theta}},\quad E(\varphi,k)\coloneqq \int_0^\varphi \sqrt{1-k^2\sin^2\theta}d\theta
    \end{equation*}
    In the special case when $\varphi=\frac{\pi}{2}$ the elliptic integrals are said to be complete, we denote them by
    \begin{equation*}
        F(k)\equiv F\left(\frac{\pi}{2},k\right),\quad E(k)\equiv E\left(\frac{\pi}{2},k\right)
    \end{equation*}
    Moreover when $\varphi=0$ we have from \cite[Eq. 111.00]{handbook}
    \begin{equation*}
        F(0,k)=E(0,k)=0
    \end{equation*}
\end{defn}
\begin{remark}\label{monotonicityellipticintegrals}
\normalfont
    The complete elliptic integral of the first kind $F(k)$ is positive and monotonically increasing in $k$ while the complete integral of the second kind $E(k)$ is positive and monotonically decreasing in $k$.
\end{remark}
\begin{defn}\label{gammaconstants}
\normalfont
   We define the following positive real constants which only depend on the parameter $r$ controlling the conformal structure of the López minimal Klein bottle
    \begin{gather*}
        k_{\alpha_1}=\frac{1}{\sqrt{1+r^2}},\quad k_{\alpha_2}=\frac{r}{\sqrt{1+r^2}},\quad \gamma_{\alpha_1}\coloneqq \frac{k_{\alpha_2}}{F(k_{\alpha_1})}\frac{dF(k_{\alpha_1})}{dk_{\alpha_1}},\quad \gamma_{\alpha_2}\coloneqq \frac{k_{\alpha_1}}{F(k_{\alpha_2})}\frac{dF(k_{\alpha_2})}{dk_{\alpha_2}}
    \end{gather*}
\end{defn}

\begin{defn}
\normalfont
    For all $t\in (0,1]$ we define the set $S_t\coloneq\{a_t,b_t,-\frac{1}{a_t},-\frac{1}{b_t}\}$. 
\end{defn}

\begin{defn}\label{Omegas}
\normalfont
    Consider the meromorphic one-forms
    \begin{gather*}
       \Omega_j= \begin{cases}
             \eta_j,\ \quad j=1,2\\ \cG\eta_3,\ \quad j=3
        \end{cases}
    \end{gather*}
\end{defn}
Recall Definition \ref{specialpolynomials}.
\begin{defn}\label{Dconstants}
\normalfont
    For all $t\in(0,1]$ and $p\in S_t$, we introduce the following useful quantities
    \begin{gather*}
     D_p^1=\frac{w(B_p)}{\cG_t(B_p)}\frac{P_1(p,t)}{k_p^2(p,t)},\quad  D_p^2=\frac{w(B_p)}{\cG_t(B_p)}\frac{P_2(p,t)}{k_p^2(p,t)},\quad  D_p^3= \cG_t(B_p)w(B_p)\frac{P_3(p,t)}{k_p^2(p,t)}
    \end{gather*}
    where $k_p(z,t)$ denotes the polynomial $k_t(z)$ where the factor $(z-p)$ has been deleted.
\end{defn}
\begin{remark}
\normalfont
    The quantities $D_p^j$ can be rewritten as,
    \begin{equation*}
        D_p^1=\frac{1}{\sqrt{r}}\frac{(tp-1)^2(p+r)p_1(p,t)}{k_p^2(p,t)},\quad   D_p^2=\frac{1}{\sqrt{r}}\frac{(tp-1)^2(p+r)p_2(p,t)}{k_p^2(p,t)},\quad D_p^3=\sqrt{r}\frac{(p+t)(tp-1)p_3(p,t)}{k_p^2(p,t)}q(p)
    \end{equation*}
\end{remark}

We next state the two main results of this Section

\begin{thm}[{Integrals along $\alpha_1$}]\label{theorem1}
    For all $j=1,2,3$
    \begin{gather*}
    \int_{\alpha_1}\Omega_j=2i\left(\frac{r}{1+r^2}\right)^{\frac{1}{2}}F(k_{\alpha_1})\sum_{p\in S_t} (\gamma_{\alpha_1}-p)\frac{D_p^j}{q(p)}
    \end{gather*}
\end{thm}

\begin{thm}[{Integrals along $\alpha_2$}]\label{theorem2}
    For all $j=1,2,3$
    \begin{gather*}
    \int_{\alpha_2}\Omega_j=-2\left(\frac{r}{1+r^2}\right)^{\frac{1}{2}}F(k_{\alpha_2})\sum_{p\in S_t}(\gamma_{\alpha_2}+p)\frac{D_p^j}{q(p)}
    \end{gather*}
\end{thm}
For the proof of these results we introduce some preliminary lemmas.

\begin{lema}[{\cite[Eq. 710.00]{handbook}}]\label{lema0}
    \begin{gather*}
        \frac{1}{k_{\alpha_1}}\frac{dF(k_{\alpha_1})}{dk_{\alpha_1}}=\frac{E(k_{\alpha_1})-k_{\alpha_2}^2F(k_{\alpha_1})}{k_{\alpha_1}^2k_{\alpha_2}^2},\quad \frac{1}{k_{\alpha_2}}\frac{dF(k_{\alpha_2})}{dk_{\alpha_2}}=\frac{E(k_{\alpha_2})-k_{\alpha_1}^2F(k_{\alpha_2})}{k_{\alpha_1}^2k_{\alpha_2}^2}
    \end{gather*}
\end{lema}
\begin{lema}\label{lema1}
The following integrals are useful for $\alpha_1$
    \begin{itemize}
        \item[I.] $ \int_0^{\frac{1}{r}}\frac{dz}{\sqrt{-q(z)}}=2\left(\frac{r}{1+r^2}\right)^{\frac{1}{2}}F(k_{\alpha_1})$
        \item[II.] $ \int_0^{\frac{1}{r}}\frac{zdz}{\sqrt{-q(z)}}=2\left(\frac{r}{1+r^2}\right)^{\frac{3}{2}}\frac{1}{k_{\alpha_1}}\frac{dF(k_{\alpha_1})}{dk_{\alpha_1}}$
    \end{itemize}
\end{lema}
\begin{proof}
    For item I) we use \cite[Eq. 235.00]{handbook} with $(a,b,c,y)=\left(\frac{1}{r},0,-r,\frac{1}{r}\right)$. For item II), we use \cite[Eq. 235.06]{handbook} with the same choice and use additionally \cite[Eq. 318.02, Eq. 122.01-0.2]{handbook}, Lemma \ref{lema0}.
\end{proof}

\begin{lema}\label{lema2}
The following integrals are useful for $\alpha_2$
    \begin{itemize}
        \item[I.] $\int_{-r}^{0}\frac{dz}{\sqrt{q(z)}}=2\left(\frac{r}{1+r^2}\right)^{\frac{1}{2}}F(k_{\alpha_2})$
        \item[II.] $ \int_{-r}^{0}\frac{zdz}{\sqrt{q(z)}}=-2\left(\frac{r}{1+r^2}\right)^{\frac{3}{2}}\frac{1}{k_{\alpha_2}}\frac{dF(k_{\alpha_2})}{dk_{\alpha_2}}$
    \end{itemize}
\end{lema}
\begin{proof}
    For item I) we use \cite[Eq. 234.00]{handbook} with $(a,b,c,y)=\left(\frac{1}{r},0,-r,-r\right)$. For item II), we use \cite[Eq. 234.07]{handbook} with the same choice and use additionally \cite[Eq. 318.02, Eq. 122.01-0.2]{handbook}, Lemma \ref{lema0}.
\end{proof}

\begin{lema}\label{lema3}
The following integrals are useful for $\alpha_1$
    \begin{itemize}
        \item[I.] $\int_{-\infty}^{-r}\frac{dz}{\sqrt{-q(z)}}=\int_{0}^{\frac{1}{r}}\frac{dz}{\sqrt{-q(z)}}=2\left(\frac{r}{1+r^2}\right)^{\frac{1}{2}}F(k_{\alpha_1})$
        \item[II.] $ \int_{-\infty}^{-r}\frac{dz}{z\sqrt{-q(z)}}=-\int_{0}^{\frac{1}{r}}\frac{zdz}{\sqrt{-q(z)}}=-2\left(\frac{r}{1+r^2}\right)^{\frac{3}{2}}\frac{1}{k_{\alpha_1}}\frac{dF(k_{\alpha_1})}{dk_{\alpha_1}}$
    \end{itemize}
\end{lema}

\begin{proof}
    Indeed, notice that
    \begin{equation*}
        q\left(-\frac{1}{z}\right)=\left(-\frac{1}{z}\right)\left(-\frac{1}{z}-\frac{1}{r}\right)\left(-\frac{1}{z}+r\right)=\frac{z(z+r)\left(z-\frac{1}{r}\right)}{z^4}=\frac{q(z)}{z^4}
    \end{equation*}
    Then the change of variable $z\to -\frac{1}{z}$ gives
    \begin{gather*}
        \int_{-\infty}^{-r}\frac{dz}{\sqrt{-q(z)}}=\int_0^{\frac{1}{r}}\frac{d\left(-\frac{1}{z}\right)}{\sqrt{-q\left(-\frac{1}{z}\right)}}=\int_0^{\frac{1}{r}}\frac{dz/z^2}{\sqrt{-q(z)/z^4}}=\int_0^{\frac{1}{r}}\frac{dz}{\sqrt{-q(z)}}
    \end{gather*}
    \begin{gather*}
        \int_{-\infty}^{-r}\frac{dz}{z\sqrt{-q(z)}}=\int_0^{\frac{1}{r}}\frac{dz/z^2}{\left(-\frac{1}{z}\right)\sqrt{-q(z)/z^4}}=-\int_0^{\frac{1}{r}}\frac{zdz}{\sqrt{-q(z)}}
    \end{gather*}
    The rest of the proof follows from Lemma \ref{lema1}
\end{proof}

\begin{lema}\label{lema4}
The following integrals are useful for $\alpha_2$
    \begin{itemize}
        \item[I.] $\int_{\frac{1}{r}}^{\infty}\frac{dz}{\sqrt{q(z)}}=\int_{-r}^{0}\frac{dz}{\sqrt{q(z)}}=2\left(\frac{r}{1+r^2}\right)^{\frac{1}{2}}F(k_{\alpha_2})$
        \item[II.] $ \int_{\frac{1}{r}}^{\infty}\frac{dz}{z\sqrt{q(z)}}=-\int_{-r}^{0}\frac{zdz}{\sqrt{q(z)}}=2\left(\frac{r}{1+r^2}\right)^{\frac{3}{2}}\frac{1}{k_{\alpha_2}}\frac{dF(k_{\alpha_2})}{dk_{\alpha_2}}$
    \end{itemize}
\end{lema}
\begin{proof}
    It follows the same strategy of proof of Lemma \ref{lema3}.
\end{proof}

\begin{prop}[{Useful for $\alpha_1$}]\label{prop1}
    For all $p\in S_t\setminus \{-\frac{1}{b_t}\}$
    \begin{equation*}
        \int_0^{\frac{1}{r}}\frac{dz}{(z-p)^2\sqrt{-q(z)}}+\frac{q'(p)}{2q(p)}\int_0^{\frac{1}{r}}\frac{dz}{(z-p)\sqrt{-q(z)}}=\frac{1}{q(p)}\left(\frac{r}{1+r^2}\right)^{\frac{3}{2}}\frac{1}{k_{\alpha_1}}\frac{dF(k_{\alpha_1})}{dk_{\alpha_1}}-\frac{p}{q(p)}\left(\frac{r}{1+r^2}\right)^{\frac{1}{2}}F(k_{\alpha_1})
    \end{equation*}
\end{prop}

\begin{proof}
    From \cite[Eq. 230.02]{handbook} with the choices
    \begin{equation*}
        (a_0,r_1,r_2,r_3,m,y_1,y)=\left(-1,r,0,-\frac{1}{r},2,0,\frac{1}{r}\right),\quad p\neq -\frac{1}{b}
    \end{equation*}
    \begin{equation*}
        \int_0^{\frac{1}{r}}\frac{dz}{(z-p)^2\sqrt{-q(z)}}=\frac{1}{2(-q(p))}\left\lbrace-\int_0^{\frac{1}{r}}\frac{(z-p)dz}{\sqrt{-q(z)}}+\left[3p^2+2p\left(r-\frac{1}{r}\right)-1\right]\int_0^{\frac{1}{r}}\frac{dz}{(z-p)\sqrt{-q(z)}}\right\rbrace
    \end{equation*}
    Since $q(z)=z^3+\left(r-\frac{1}{r}\right)z^2-z$, $q'(z)=3z^2+2\left(r-\frac{1}{r}\right)z-1$, so we have by Lemma \ref{lema1}
    \begin{align*}
    \begin{split}
        \int_0^{\frac{1}{r}}\frac{dz}{(z-p)^2\sqrt{-q(z)}}+\frac{q'(p)}{2q(p)}\int_0^{\frac{1}{r}}\frac{dz}{(z-p)\sqrt{-q(z)}}&=\frac{1}{2q(p)}\int_0^{\frac{1}{r}}\frac{(z-p)dz}{\sqrt{-q(z)}}\\&=\frac{1}{q(p)}\left(\frac{r}{1+r^2}\right)^{\frac{3}{2}}\frac{1}{k_{\alpha_1}}\frac{dF(k_{\alpha_1})}{dk_{\alpha_1}}-\frac{p}{q(p)}\left(\frac{r}{1+r^2}\right)^{\frac{1}{2}}F(k_{\alpha_1})
    \end{split}
    \end{align*}
\end{proof}

\begin{prop}[{Useful for $\alpha_1$}]\label{prop2}
    \begin{align*}
        \int_{-\infty}^{-r}\frac{dz}{\left(z+\frac{1}{b}\right)^2\sqrt{-q(z)}}+&\frac{q'\left(-\frac{1}{b}\right)}{2q\left(-\frac{1}{b}\right)}\int_{-\infty}^{-r}\frac{dz}{\left(z+\frac{1}{b}\right)\sqrt{-q(z)}}\\&=\frac{1}{q\left(-\frac{1}{b}\right)}\left(\frac{r}{1+r^2}\right)^{\frac{3}{2}}\frac{1}{k_{\alpha_1}}\frac{dF(k_{\alpha_1})}{dk_{\alpha_1}}-\frac{\left(-\frac{1}{b}\right)}{q\left(-\frac{1}{b}\right)}\left(\frac{r}{1+r^2}\right)^{\frac{1}{2}}F(k_{\alpha_1})
    \end{align*}
\end{prop}

\begin{proof}
    From \cite[Eq. 230.02]{handbook} with the choices 
    \begin{equation*}
        \left(a_0,r_1,r_2,r_3,m,y_1,y\right)=\left(-1,r,0,-\frac{1}{r},2,y_1,-r\right),\quad p=-\frac{1}{b}
    \end{equation*}

    \begin{equation*}
        \int_{y_1}^{-r}\frac{dz}{\left(z+\frac{1}{b}\right)^2\sqrt{-q(z)}}=-\frac{1}{2\left(q\left(-\frac{1}{b}\right)\right)}\left\lbrace \frac{2\sqrt{-q(y_1)}}{y_1+\frac{1}{b}}-\int_{y_1}^{-r}\frac{\left(z+\frac{1}{b}\right)dz}{\sqrt{-q(z)}}+q'\left(-\frac{1}{b}\right)\int_{y_1}^{-r}\frac{dz}{\left(z+\frac{1}{b}\right)\sqrt{-q(z)}} \right\rbrace
    \end{equation*}
    \begin{equation*}
        \int_{y_1}^{-r}\frac{dz}{\left(z+\frac{1}{b}\right)^2\sqrt{-q(z)}}+\frac{q'\left(-\frac{1}{b}\right)}{2q\left(-\frac{1}{b}\right)}\int_{y_1}^{-r}\frac{dz}{\left(z+\frac{1}{b}\right)\sqrt{-q(z)}}=-\frac{\sqrt{-q(y_1)}}{q\left(-\frac{1}{b}\right)\left(y_1+\frac{1}{b}\right)}+\frac{1}{2q\left(-\frac{1}{b}\right)}\int_{y_1}^{-r}\frac{\left(z+\frac{1}{b}\right)dz}{\sqrt{-q(z)}}
    \end{equation*}

    From \cite[Eq. 230.01]{handbook} with the choices
    \begin{equation*}
        (a_0,r_1,r_2,r_3,m,y_1,y)=\left(-1,r,0,-\frac{1}{r},1,y_1,-r\right)
    \end{equation*}
    \begin{equation*}
        \int_{y_1}^{-r}\frac{zdz}{\sqrt{-q(z)}}=-\left\lbrace -2\frac{\sqrt{-q(y_1)}}{y_1}+\int_{y_1}^{-r}\frac{dz}{z\sqrt{-q(z)}}\right\rbrace=\frac{2\sqrt{-q(y_1)}}{y_1}-\int_{y_1}^{-r}\frac{dz}{z\sqrt{-q(z)}}
    \end{equation*}
    \begin{align*}
        \begin{split}
            &\int_{y_1}^{-r}\frac{dz}{\left(z+\frac{1}{b}\right)^2\sqrt{-q(z)}}+\frac{q'\left(-\frac{1}{b}\right)}{2q\left(-\frac{1}{b}\right)}\int_{y_1}^{-r}\frac{dz}{\left(z+\frac{1}{b}\right)\sqrt{-q(z)}}\\&=-\frac{\sqrt{-q(y_1)}}{q\left(-\frac{1}{b}\right)\left(y_1+\frac{1}{b}\right)}+\frac{\sqrt{-q(y_1)}}{q\left(-\frac{1}{b}\right)y_1}-\frac{1}{2q\left(-\frac{1}{b}\right)}\int_{y_1}^{-r}\frac{dz}{z\sqrt{-q(z)}}-\frac{\left(-\frac{1}{b}\right)}{2q\left(-\frac{1}{b}\right)}\int_{y_1}^{-r}\frac{dz}{\sqrt{-q(z)}}
        \end{split}
    \end{align*}
    Since
    \begin{equation*}
        \lim_{y_1\to -\infty}\left[\frac{-\sqrt{-q(y_1)}}{\left(y_1+\frac{1}{b}\right)}+\frac{\sqrt{-q(y_1)}}{y_1}\right]= \lim_{y_1\to -\infty}\frac{\left(\frac{1}{b}\right)\sqrt{-q(y_1)}}{y_1\left(y_1+\frac{1}{b}\right)}=0
    \end{equation*}
    we have by Lemma \ref{lema3}
    \begin{align*}
    \begin{split}
        \int_{-\infty}^{-r}\frac{dz}{\left(z+\frac{1}{b}\right)^2\sqrt{-q(z)}}+\frac{q'\left(-\frac{1}{b}\right)}{2q\left(-\frac{1}{b}\right)}&\int_{-\infty}^{-r}\frac{dz}{\left(z+\frac{1}{b}\right)\sqrt{-q(z)}}\\&=-\frac{1}{2q\left(-\frac{1}{b}\right)}\int_{-\infty}^{-r}\frac{dz}{z\sqrt{-q(z)}}-\frac{\left(-\frac{1}{b}\right)}{2q\left(-\frac{1}{b}\right)}\int_{-\infty}^{-r}\frac{dz}{\sqrt{-q(z)}}\\&=\frac{1}{q\left(-\frac{1}{b}\right)}\left(\frac{r}{1+r^2}\right)^{\frac{3}{2}}\frac{1}{k_{\alpha_1}}\frac{dF(k_{\alpha_1})}{dk_{\alpha_1}}-\frac{\left(-\frac{1}{b}\right)}{q\left(-\frac{1}{b}\right)}\left(\frac{r}{1+r^2}\right)^{\frac{1}{2}}F(k_{\alpha_1})
    \end{split}
    \end{align*}
\end{proof}

\begin{prop}[{Useful for $\alpha_2$}]\label{prop3}
    For all $p\in S_t\setminus \{a_t\}$
    \begin{equation*}
        \int_{-r}^{0}\frac{dz}{(z-p)^2\sqrt{q(z)}}+\frac{q'(p)}{2q(p)}\int_{-r}^{0}\frac{dz}{(z-p)\sqrt{q(z)}}=-\frac{1}{q(p)}\left(\frac{r}{1+r^2}\right)^{\frac{3}{2}}\frac{1}{k_{\alpha_2}}\frac{dF(k_{\alpha_2})}{dk_{\alpha_2}}-\frac{p}{q(p)}\left(\frac{r}{1+r^2}\right)^{\frac{1}{2}}F(k_{\alpha_2})
    \end{equation*}
\end{prop}

\begin{proof}
    From \cite[Eq. 230.02]{handbook} with the choices
    \begin{equation*}
        (a_0,r_1,r_2,r_3,m,y_1,y)=\left(+1,r,0,-\frac{1}{r},2,-r,0\right),\quad p\neq a_t
    \end{equation*}
    \begin{equation*}
        \int_{-r}^{0}\frac{dz}{(z-p)^2\sqrt{q(z)}}=\frac{1}{2q(p)}\left\lbrace\int_{-r}^{0}\frac{(z-p)dz}{\sqrt{q(z)}}-q'(p)\int_{-r}^{0}\frac{dz}{(z-p)\sqrt{q(z)}}\right\rbrace
    \end{equation*}
    So we have by Lemma \ref{lema2}
    \begin{align*}
    \begin{split}
        \int_{-r}^{0}\frac{dz}{(z-p)^2\sqrt{q(z)}}+\frac{q'(p)}{2q(p)}\int_{-r}^{0}\frac{dz}{(z-p)\sqrt{q(z)}}&=\frac{1}{2q(p)}\int_{-r}^{0}\frac{(z-p)dz}{\sqrt{q(z)}}\\&=-\frac{1}{q(p)}\left(\frac{r}{1+r^2}\right)^{\frac{3}{2}}\frac{1}{k_{\alpha_2}}\frac{dF(k_{\alpha_2})}{dk_{\alpha_2}}-\frac{p}{q(p)}\left(\frac{r}{1+r^2}\right)^{\frac{1}{2}}F(k_{\alpha_2})
    \end{split}
    \end{align*}
\end{proof}

\begin{prop}[{Useful for $\alpha_2$}]\label{prop4}
    \begin{equation*}
        \int_{\frac{1}{r}}^{+\infty}\frac{dz}{\left(z-a\right)^2\sqrt{q(z)}}+\frac{q'\left(a\right)}{2q\left(a\right)}\int_{\frac{1}{r}}^{+\infty}\frac{dz}{\left(z-a\right)\sqrt{q(z)}}=-\frac{1}{q\left(a\right)}\left(\frac{r}{1+r^2}\right)^{\frac{3}{2}}\frac{1}{k_{\alpha_2}}\frac{dF(k_{\alpha_2})}{dk_{\alpha_2}}-\frac{a}{q\left(a\right)}\left(\frac{r}{1+r^2}\right)^{\frac{1}{2}}F(k_{\alpha_2})
    \end{equation*}
\end{prop}

\begin{proof}
    From \cite[Eq. 230.02]{handbook} with the choices 
    \begin{equation*}
        \left(a_0,r_1,r_2,r_3,m,y_1,y\right)=\left(+1,r,0,-\frac{1}{r},2,\frac{1}{r},y\right),\quad p=a_t
    \end{equation*}

    \begin{equation*}
        \int_{\frac{1}{r}}^{y}\frac{dz}{\left(z-a\right)^2\sqrt{q(z)}}=\frac{1}{2q\left(a\right)}\left\lbrace -\frac{2\sqrt{q(y)}}{y-a}+\int_{\frac{1}{r}}^{y}\frac{\left(z-a\right)dz}{\sqrt{q(z)}}-q'\left(a\right)\int_{\frac{1}{r}}^{y}\frac{dz}{\left(z-a\right)\sqrt{q(z)}} \right\rbrace
    \end{equation*}
    \begin{equation*}
        \int_{\frac{1}{r}}^{y}\frac{dz}{\left(z-a\right)^2\sqrt{q(z)}}+\frac{q'\left(a\right)}{2q\left(a\right)}\int_{\frac{1}{r}}^{y}\frac{dz}{\left(z-a\right)\sqrt{q(z)}}=-\frac{\sqrt{q(y)}}{q\left(a\right)\left(y-a\right)}+\frac{1}{2q\left(a\right)}\int_{\frac{1}{r}}^{y}\frac{\left(z-a\right)dz}{\sqrt{q(z)}}
    \end{equation*}

    From \cite[Eq. 230.01]{handbook} with the choices
    \begin{equation*}
        (a_0,r_1,r_2,r_3,m,y_1,y)=\left(+1,r,0,-\frac{1}{r},1,\frac{1}{r},y\right)
    \end{equation*}
    \begin{equation*}
        \int_{\frac{1}{r}}^{y}\frac{zdz}{\sqrt{q(z)}}=\frac{2\sqrt{q(y)}}{y}-\int_{\frac{1}{r}}^{y}\frac{dz}{z\sqrt{q(z)}}
    \end{equation*}
    \begin{align*}
        \begin{split}
            &\int_{\frac{1}{r}}^{y}\frac{dz}{\left(z-a\right)^2\sqrt{q(z)}}+\frac{q'\left(a\right)}{2q\left(a\right)}\int_{\frac{1}{r}}^{y}\frac{dz}{\left(z-a\right)\sqrt{q(z)}}\\&=\frac{-\sqrt{q(y)}}{q\left(a\right)\left(y-a\right)}+\frac{\sqrt{q(y)}}{q\left(a\right)y}-\frac{1}{2q\left(a\right)}\int_{\frac{1}{r}}^{y}\frac{dz}{z\sqrt{q(z)}}-\frac{a}{2q\left(a\right)}\int_{\frac{1}{r}}^{y}\frac{dz}{\sqrt{q(z)}}
        \end{split}
    \end{align*}
    Since
    \begin{equation*}
        \lim_{y\to +\infty}\frac{\sqrt{q(y)}}{y(y-a)}=0
    \end{equation*}
    we have by Lemma \ref{lema4}
    \begin{equation*}
        \int_{\frac{1}{r}}^{+\infty}\frac{dz}{\left(z-a\right)^2\sqrt{q(z)}}+\frac{q'\left(a\right)}{2q\left(a\right)}\int_{\frac{1}{r}}^{+\infty}\frac{dz}{\left(z-a\right)\sqrt{q(z)}}=-\frac{1}{q\left(a\right)}\left(\frac{r}{1+r^2}\right)^{\frac{3}{2}}\frac{1}{k_{\alpha_2}}\frac{dF(k_{\alpha_2})}{dk_{\alpha_2}}-\frac{a}{q\left(a\right)}\left(\frac{r}{1+r^2}\right)^{\frac{1}{2}}F(k_{\alpha_2})
    \end{equation*}
\end{proof}

\begin{lema}\label{lema5}
    For each $j=1,2,3$ there exists a polynomial $Q_j(z,t)$ with $\deg(Q_j)\leq7$ in the $z$ variable such that
    \begin{equation*}
        \Omega_j=\frac{Q_j(z,t)}{k_t^2(z)}\frac{dz}{w}
    \end{equation*}
\end{lema}
\begin{proof}
    Fix $j=1,2$. We know that
    \begin{equation*}
        \Omega_j=\eta_j=\frac{1}{\cG_t}\frac{P_j}{k_t^2}dz
    \end{equation*}
    Recall that for $j=1,2$ $P_j=(tz-1)(z+t)p_j(z,t)$ where $\deg(p_j)\leq4$ in the $z$ variable. Then
    \begin{equation*}
        \forall j=1,2\quad \Omega_j=\frac{(tz-1)(z+r)}{\sqrt{r}(z+t)w}\frac{(tz-1)(z+t)p_j}{k_t^2}dz=\frac{Q_j}{k_t^2}\frac{dz}{w}
    \end{equation*}
    with 
    \begin{equation*}
        Q_j(z,t)=\frac{1}{\sqrt{r}}(tz-1)^2(z+r)p_j(z,t)\implies \deg(Q_j)\leq7\ \text{in the $z$ variable.}
    \end{equation*}
    Now consider $j=3$,
    \begin{equation*}
        \Omega_3=\cG_t\eta_3=\cG_t\frac{P_3}{k_t^2}dz.
    \end{equation*}
    Recall
    \begin{equation*}
        P_3(z,t)=(tz-1)^2(z+r)p_3(z,t),\ \text{with}\, \deg(p_3)\leq2\ \text{in the $z$ variable.}
    \end{equation*}
    Then
    \begin{equation*}
        \Omega_3=\sqrt{r}\frac{(z+t)w}{(tz-1)(z+r)}\frac{(tz-1)^2(z+r)p_3}{k_t^2}dz=\frac{Q_3}{k_t^2}\frac{dz}{w},
    \end{equation*}
    with 
    \begin{equation*}
        Q_3(z,t)=\sqrt{r}(z+t)(tz-1)z\left(z-\frac{1}{r}\right)(z+r)p_3(z,t),\ \text{and } \deg(Q_3)\leq 7\ \text{in the $z$ variable}.
    \end{equation*}
\end{proof}

\begin{lema}\label{lema6}
    For all $j=1,2,3$ we have the partial fraction decomposition
    \begin{equation*}
        \frac{Q_j(z,t)}{k_t^2(z)}=\sum_{p\in S_t}\frac{D_p^j}{(z-p)^2}+\frac{E_p^j}{(z-p)},\quad E_p^j\coloneq \frac{D_p^j}{2}\frac{q'(p)}{q(p)},
    \end{equation*}
    where $D_p^j$ was defined in Definition \ref{Dconstants}.
\end{lema}

\begin{proof}
    Notice that by Lemma \ref{lema5}, $Q_j(z,t)$ is a polynomial with degree at most $7$ in $z$ which is strictly less than the degree of $k_t^2(z)$ in $z$. Then by the well-known theorem of partial fraction decomposition we have that
    \begin{equation*}
        \frac{Q_j(z,t)}{k_t^2(z)}=\sum_{p\in S_t}\frac{\Tilde{D}_p^j}{(z-p)^2}+\frac{\tilde{E}_p^j}{(z-p)}
    \end{equation*}
    We now Taylor expand $\Omega_j=\frac{Q_j}{k_t^2}\frac{dz}{w}$ about the branch point $B_p$ in two different equivalent ways. On one hand for all $j=1,2,3$
    \begin{align}\label{firstway}
        \begin{split}
            \frac{Q_j}{k_t^2}\frac{dz}{w}&=\left[\frac{\Tilde{D}_p^j}{(z-p)^2}+\frac{\tilde{E}_p^j}{(z-p)}+\text{holomorphic}\right]\left[\frac{1}{w(B_p)}-\frac{w'(B_p)}{w^2(B_p)}(z-p)+O\left((z-p)^2\right)\right]dz\\&=\left[\frac{\left(\frac{\Tilde{D}_p^j}{w(B_p)}\right)}{(z-p)^2}+\frac{\left(-\frac{\Tilde{D}_p^jw'(B_p)}{w^2(B_p)}+\frac{\tilde{E}_p^j}{w(B_p)}\right)}{(z-p)}+\text{holomorphic}\right]dz
        \end{split}
    \end{align}
    On the other hand we have
    \begin{itemize}
        \item[I.] $j=1,2$. In this case
        \begin{equation}\label{secondway}
            \Omega_j=\frac{1}{\cG_t}\frac{P_j}{k_t^2}dz=\left[\frac{\left(\frac{1}{\cG_t(B_p)}\frac{P_j(p)}{k_p^2(p,t)}\right)}{(z-p)^2}+\text{holomorphic}\right]dz
        \end{equation}
        Comparing equations \eqref{firstway}, \eqref{secondway} and using Definition \ref{Dconstants} we obtain
        \begin{equation*}
            \Tilde{D}_p^j=\frac{w(B_p)}{\cG_t(B_p)}\frac{P_j(p)}{k_p^2(p,t)}=D_p^j,\quad \tilde{E}_p^j=\Tilde{D}_p^j\frac{w'(B_p)}{w(B_p)}=D_p^j\frac{w'(B_p)}{w(B_p)}=E^j_p
        \end{equation*}
        We know that $w^2=q(z)$, then $2w'(B_p)w(B_p)=q'(p)$, so
        \begin{equation*}
            E_p^j=D_p^j\frac{w'(B_p)}{w(B_p)}=\frac{D_p^j}{2}\frac{q'(p)}{w^2(B_p)}=\frac{D_p^j}{2}\frac{q'(p)}{q(p)}
        \end{equation*}
        \item[II.] $j=3.$ In this case we have
        \begin{equation}\label{secondwayj3}
            \Omega_3=\frac{\cG_tP_3}{k_t^2}dz=\left[\frac{\left(\frac{\cG_t(B_p)P_3(p)}{k_p^2(p,t)}\right)}{(z-p)^2}+\text{holomorphic}\right]dz
        \end{equation}
        Comparing equations \eqref{firstway}, \eqref{secondwayj3} and using Definition \ref{Dconstants} we obtain
        \begin{equation*}
            \Tilde{D}_p^3=\cG_t(B_p)w(B_p)\frac{P_3(p)}{k_p^2(p,t)}=D_p^3,\quad \tilde{E}_p^3=\Tilde{D}_p^3\frac{w'(B_p)}{w(B_p)}=\frac{D_p^3}{2}\frac{q'(p)}{q(p)} = E^3_p
        \end{equation*}
    \end{itemize}
\end{proof}

We are ready to prove the main results
\begin{proof}[{Proof of Theorem \ref{theorem1}}]
  From Lemma \ref{lema6} we have
  \begin{equation*}
      \int_{\alpha_1}\Omega_j=\int_{\alpha_1}\left[\frac{D_{-\frac{1}{b}}^j}{\left(z+\frac{1}{b_t}\right)^2}+\frac{E_{-\frac{1}{b}}^j}{\left(z+\frac{1}{b_t}\right)}\right]\frac{dz}{w}+\sum_{p\in S_t\setminus \left\lbrace-\frac{1}{b}\right\rbrace}\int_{\alpha_1}\left[\frac{D_p^j}{(z-p)^2}+\frac{E_p^j}{(z-p)}\right]\frac{dz}{w}
  \end{equation*}

Using Lemma \ref{lema7} items $I$ and $II$ we obtain
\begin{align*}
    \begin{split}
        \int_{\alpha_1}\Omega_j&=2iD_{-\frac{1}{b}}^j\int_{-\infty}^{-r}\frac{dz}{\left(z+\frac{1}{b_t}\right)^2\sqrt{-q(z)}}+2iE_{-\frac{1}{b}}^j\int_{-\infty}^{-r}\frac{dz}{\left(z+\frac{1}{b_t}\right)\sqrt{-q(z)}}\\&+\sum_{p\in S_t\setminus\left\lbrace -\frac{1}{b_t} \right\rbrace}2iD_{p}^j\int_{0}^{\frac{1}{r}}\frac{dz}{\left(z-p\right)^2\sqrt{-q(z)}}+2iE_{p}^j\int_{0}^{\frac{1}{r}}\frac{dz}{\left(z-p\right)\sqrt{-q(z)}}
    \end{split}
\end{align*}
By Lemma \ref{lema6} we can rewrite the integral as
\begin{align*}
    \begin{split}
        \int_{\alpha_1}\Omega_j&=2iD_{-\frac{1}{b}}^j\left[\int_{-\infty}^{-r}\frac{dz}{\left(z+\frac{1}{b_t}\right)^2\sqrt{-q(z)}}+\frac{q'\left(-\frac{1}{b_t}\right)}{2q\left(-\frac{1}{b_t}\right)}\int_{-\infty}^{-r}\frac{dz}{\left(z+\frac{1}{b_t}\right)\sqrt{-q(z)}}\right]\\&+\sum_{p\in S_t\setminus\left\lbrace -\frac{1}{b_t} \right\rbrace}2iD_{p}^j\left[\int_{0}^{\frac{1}{r}}\frac{dz}{\left(z-p\right)^2\sqrt{-q(z)}}+\frac{q'(p)}{2q(p)}\int_{0}^{\frac{1}{r}}\frac{dz}{\left(z-p\right)\sqrt{-q(z)}}\right]
    \end{split}
\end{align*}
Finally, using Propositions \ref{prop1} and \ref{prop2}
\begin{align*}
    \begin{split}
        \int_{\alpha_1}\Omega_j&=2iD_{-\frac{1}{b}}^j\left[\frac{1}{q\left(-\frac{1}{b}\right)}\left(\frac{r}{1+r^2}\right)^{\frac{3}{2}}\frac{1}{k_{\alpha_1}}\frac{dF(k_{\alpha_1})}{dk_{\alpha_1}}-\frac{\left(-\frac{1}{b}\right)}{q\left(-\frac{1}{b}\right)}\left(\frac{r}{1+r^2}\right)^{\frac{1}{2}}F(k_{\alpha_1})\right]\\&+\sum_{p\in S_t\setminus\left\lbrace -\frac{1}{b_t} \right\rbrace}2iD_{p}^j\left[\frac{1}{q(p)}\left(\frac{r}{1+r^2}\right)^{\frac{3}{2}}\frac{1}{k_{\alpha_1}}\frac{dF(k_{\alpha_1})}{dk_{\alpha_1}}-\frac{p}{q(p)}\left(\frac{r}{1+r^2}\right)^{\frac{1}{2}}F(k_{\alpha_1})\right]\\&=2i\left(\frac{r}{1+r^2}\right)^{\frac{3}{2}}\frac{1}{k_{\alpha_1}}\frac{dF(k_{\alpha_1})}{dk_{\alpha_1}}\sum_{p\in S_t} \left(\frac{D_{p}^j}{q(p)}\right)-2i\left(\frac{r}{1+r^2}\right)^{\frac{1}{2}}F(k_{\alpha_1})\left(\sum_{p\in S_t}\frac{pD_p^j}{q(p)}\right)
    \end{split}
\end{align*}
where the conclusion follows from an application of Definition \ref{gammaconstants}.
\end{proof}

In exactly the same analogous way we can prove the corresponding result for integration on $\alpha_2$:
\begin{proof}[{Proof of Theorem \ref{theorem2}}]
     From Lemma \ref{lema6} we have
  \begin{equation*}
      \int_{\alpha_2}\Omega_j=\int_{\alpha_2}\left[\frac{D_{a}^j}{\left(z-a_t\right)^2}+\frac{E_{a}^j}{\left(z-a_t\right)}\right]\frac{dz}{w}+\sum_{p\in S_t\setminus \left\lbrace a \right\rbrace}\int_{\alpha_2}\left[\frac{D_p^j}{(z-p)^2}+\frac{E_p^j}{(z-p)}\right]\frac{dz}{w}
  \end{equation*}

Using Lemma \ref{lema7} items $III$ and $IV$ we obtain
\begin{align*}
    \begin{split}
        \int_{\alpha_2}\Omega_j&=2D_{a}^j\int_{\frac{1}{r}}^{\infty}\frac{dz}{\left(z-a_t\right)^2\sqrt{q(z)}}+2E_{a}^j\int_{\frac{1}{r}}^{\infty}\frac{dz}{\left(z-a_t\right)\sqrt{q(z)}}\\&+\sum_{p\in S_t\setminus\left\lbrace a_t \right\rbrace}2D_{p}^j\int_{-r}^{0}\frac{dz}{\left(z-p\right)^2\sqrt{q(z)}}+2E_{p}^j\int_{-r}^{0}\frac{dz}{\left(z-p\right)\sqrt{q(z)}}.
    \end{split}
\end{align*}
By Lemma \ref{lema6} we can rewrite the integral as
\begin{align*}
    \begin{split}
        \int_{\alpha_2}\Omega_j&=2D_{a}^j\left[\int_{\frac{1}{r}}^{\infty}\frac{dz}{\left(z-a_t\right)^2\sqrt{q(z)}}+\frac{q'\left(a\right)}{2q\left(a\right)}\int_{\frac{1}{r}}^{\infty}\frac{dz}{\left(z-a_t\right)\sqrt{q(z)}}\right]\\&+\sum_{p\in S_t\setminus\left\lbrace a_t \right\rbrace}2D_{p}^j\left[\int_{-r}^{0}\frac{dz}{\left(z-p\right)^2\sqrt{q(z)}}+\frac{q'(p)}{2q(p)}\int_{-r}^{0}\frac{dz}{\left(z-p\right)\sqrt{q(z)}}\right].
    \end{split}
\end{align*}
Finally, using Propositions \ref{prop3} and \ref{prop4}
\begin{align*}
    \begin{split}
        \int_{\alpha_2}\Omega_j&=2D_{a}^j\left[-\frac{1}{q\left(a\right)}\left(\frac{r}{1+r^2}\right)^{\frac{3}{2}}\frac{1}{k_{\alpha_2}}\frac{dF(k_{\alpha_2})}{dk_{\alpha_2}}-\frac{a}{q\left(a\right)}\left(\frac{r}{1+r^2}\right)^{\frac{1}{2}}F(k_{\alpha_2})\right]\\&+\sum_{p\in S_t\setminus\left\lbrace a_t \right\rbrace}2D_{p}^j\left[-\frac{1}{q(p)}\left(\frac{r}{1+r^2}\right)^{\frac{3}{2}}\frac{1}{k_{\alpha_2}}\frac{dF(k_{\alpha_2})}{dk_{\alpha_2}}-\frac{p}{q(p)}\left(\frac{r}{1+r^2}\right)^{\frac{1}{2}}F(k_{\alpha_2})\right]\\&=-2\left(\frac{r}{1+r^2}\right)^{\frac{3}{2}}\frac{1}{k_{\alpha_2}}\frac{dF(k_{\alpha_2})}{dk_{\alpha_2}}\sum_{p\in S_t} \left(\frac{D_{p}^j}{q(p)}\right)-2\left(\frac{r}{1+r^2}\right)^{\frac{1}{2}}F(k_{\alpha_2})\left(\sum_{p\in S_t}\frac{pD_p^j}{q(p)}\right),
    \end{split}
\end{align*}
where the conclusion follows from an application of Definition \ref{gammaconstants}.
\end{proof}

\section{Explicit estimates}\label{section::Explicit estimates}

In this appendix we present explicit estimates that are relevant to our work. We have divided the appendix into two subsections. In Section \ref{est-elliptic-integrals} we present estimates for some elliptic integrals. Section \ref{subsection::Explicit estimates for the positions of the branch points} contains explicit estimates for the positions of the branch points of the meromorphic maps $\mathcal{G}_t$ defined in \eqref{eq:deformapadegauss}.

\subsection{Elliptic integrals} \label{est-elliptic-integrals}
We now explain how to obtain an estimate for the conformal parameter $r$ of the López minimal Klein bottle.
\begin{lema}\label{rLopez}
  The conformal parameter $r$ of the López minimal Klein bottle is the reciprocal of the unique zero of the real valued function $h:[0,+\infty)\to \bR$ defined by
    \begin{equation}\label{hfunction}
        h(t)\coloneq t^2F\left(\frac{1}{\sqrt{1+t^2}}\right)-(t^2+3t-1)E\left(\frac{1}{\sqrt{1+t^2}}\right).
    \end{equation}
\end{lema}
\begin{proof}
   In the proof of \cite[Existence Theorem]{LopezConstruction}, F. López introduced the real function $f:[0,+\infty)\to \bR$ defined by
   \begin{equation*}
    f(t)=\int_{-1}^0\left[(t^2+3t-1)u+t^2(2-3t)\right]\frac{du}{\sqrt{u(u+1)(u-t^2)}}
   \end{equation*}
   and he proved \cite[a), b),c), d)]{LopezConstruction} that there exists a unique $t_{*}\in (0,\frac{1}{2})$ in $[0,+\infty)$ such that $f(t_*)=0$. Moreover López proves \cite[Eq. 5]{LopezConstruction} that $r=\frac{1}{t_*}$ is the unique value of the conformal parameter which solves the period problem for the López minimal Klein bottle. We now prove that 
   \begin{equation}\label{identity}
       f(t)=2\sqrt{1+t^2}h(t).
   \end{equation}
In fact, we can rewrite
\begin{equation}\label{rewriting}
    f(t)=(t^2+3t-1)\int_{-1}^0\frac{udu}{\sqrt{u(u+1)(u-t^2)}}+t^2(2-3t)\int_{-1}^0\frac{du}{\sqrt{u(u+1)(u-t^2)}}.
\end{equation}
From \cite[Eq. 234.00]{handbook} with the choice $(c,y,b,a)=(-1,-1,0,t^2)$ we have
\begin{equation}\label{firstintegral}
    \int_{-1}^0\frac{du}{\sqrt{u(u+1)(u-t^2)}}=\frac{2}{\sqrt{1+t^2}}F\left(\frac{1}{\sqrt{1+t^2}}\right).
\end{equation}
From \cite[Eq. 234.07]{handbook} and \cite[Eq. 318.02]{handbook} with the choice $(c,y,b,a)=(-1,-1,0,t^2)$ we have
\begin{equation}\label{secondintegral}
    \int_{-1}^0\frac{udu}{\sqrt{u(u+1)(u-t^2)}}=-\int_{-1}^0\frac{\sqrt{-u}du}{\sqrt{(u+1)(t^2-u)}}=-2\sqrt{1+t^2}\left[E\left(\frac{1}{\sqrt{1+t^2}}\right)-\frac{t^2}{1+t^2}F\left(\frac{1}{\sqrt{1+t^2}}\right)\right].
\end{equation}
Substituting equations \eqref{firstintegral} and \eqref{secondintegral} into \eqref{rewriting} proves the identity \eqref{identity} and concludes the proof of Lemma \ref{rLopez}.
\end{proof}
 From Definitions \cite[17.3.1, 17.3.3]{HandbookStegun} and equations \cite[17.3.11, 17.3.12]{HandbookStegun}, we have according to the notation of Definition \ref{ellipticintegrals} that for all $k\in (0,1)$ 
   \begin{equation*}
       F(k)=\frac{\pi}{2}\left[1+\sum_{j=1}^\infty a_jk^{2j}\right],\quad E(k)=\frac{\pi}{2}\left[1-\sum_{j=1}^{\infty}b_jk^{2j}\right]
   \end{equation*}
   where, for all $j\geq 1$
   \begin{equation*}
       a_j=\left(\frac{1}{2}\cdot \frac{3}{4}\ldots \frac{2j-1}{2j}\right)^2,\quad b_j=\frac{a_j}{2j-1}.
   \end{equation*}
   \begin{defn}\label{truncationelliptic}
   \normalfont
       Let $N$ be a natural number and $k\in (0,1)$. We define the \textit{$N$-truncations of the first and second kind of complete elliptic integrals $F_N(k)$ and $E_N(k)$} as the finite sums
       \begin{equation}\label{explicittruncationsellipticintegrals}
            F_N(k)=\frac{\pi}{2}\left[1+\sum_{j=1}^N a_jk^{2j}\right],\quad E_N(k)=\frac{\pi}{2}\left[1-\sum_{j=1}^{N}b_jk^{2j}\right].
       \end{equation}
   \end{defn}
\begin{lema}\label{CalculationEllipticIntegrals}
    Fix $k\in (0,1)$ and $\epsilon>0$. For all $N\in \bN$ sufficiently large such that 
    \begin{equation*}
        \frac{\pi}{2}\frac{k^{2(N+1)}}{1-k^2}<\epsilon
    \end{equation*}
    it holds that
    \begin{equation*}
        \abs{F(k)-F_N(k)}<\epsilon,\quad \abs{E(k)-E_N(k)}<\epsilon.
    \end{equation*}
\end{lema}
\begin{proof}
    Let $k\in (0,1)$. Notice that 
    \begin{equation*}
F(k)=F_N(k)+R^F_N(k),\quad E(k)=E_N(k)+R^E_N(k),
    \end{equation*}
    where we have defined the convergent series of positive coefficients $a_j,b_j\in \bR^+$
    \begin{equation*}
    R_N^F(k)=\frac{\pi}{2}\sum_{j=N+1}^\infty a_jk^{2j},\quad -R_N^E(k)=\frac{\pi}{2}\sum_{j=N+1}^\infty b_jk^{2j}.
    \end{equation*}
    Since $a_j,\ b_j\in (0,1)$ by Definition \ref{truncationelliptic}, it follows by comparison with the well-known geometric series that
    \begin{equation*}
       0< R_N^F(k)\leq\frac{\pi}{2}\sum_{j=N+1}^{\infty}k^{2j}=\frac{\pi}{2}\frac{k^{2(N+1)}}{1-k^2},\quad  0< -R_N^E(k)\leq\frac{\pi}{2}\sum_{j=N+1}^{\infty}k^{2j}=\frac{\pi}{2}\frac{k^{2(N+1)}}{1-k^2}
    \end{equation*}
    concluding the proof.
\end{proof}

\begin{lema}\label{rLopezNumerical}
    The value of the conformal parameter $r$ of the López minimal Klein bottle is estimated as
    \begin{equation*}
        2.544702667\leq r\leq 2.544702669.
    \end{equation*}
\end{lema}
\begin{proof}
Consider the function
\begin{equation*}
    k(t)=\frac{1}{\sqrt{1+t^2}}.
\end{equation*}
Implementing Lemma \ref{CalculationEllipticIntegrals} with $N=200$ we find that for
\begin{equation*}
 t_1=\frac{1}{2.544702669}<t_2=\frac{1}{2.544702667},
\end{equation*}
the complete integrals of the first kind and second kind are estimated by
\begin{equation}\label{estimativet}
    \abs{F(k(t_j))-F_{N}(k(t_j))}< \epsilon ,\quad \abs{E(k(t_j))-E_N(k(t_j))}< \epsilon,\quad j=1,2,
\end{equation}
where $\epsilon=10^{-11}$ and $F_N(k(t_j)),\ E_N(k(t_j))$ are rational multiples of $\pi$, obtained from Definition \ref{truncationelliptic}.

It then follows by equations \eqref{hfunction}, \eqref{explicittruncationsellipticintegrals} and \eqref{estimativet} that
    \begin{equation*}
      h(t_1)\in \bR^+,\quad h(t_2)\in  \bR^-
    \end{equation*}
    This means that the unique zero $t_0$ of $h$ in the interval $[0,+\infty)$ satisfies $$t_0\in \left(t_1,t_2\right)$$
    The result follows from an application of Lemma \ref{rLopez}.
\end{proof}
\begin{remark}
\normalfont
Since the López minimal Klein bottle has a uniqueness property \cite{LopezUniqueness}, the fixed number $r$ which controls its conformal structure is an important constant. Even though the precision of the $r$ parameter that we need for the main results of this work is of the order $10^{-9}$ as stated in Lemma \ref{rLopezNumerical}, by increasing the number of iterations to $N=300$, we are able to estimate it with sufficiently high accuracy
\begin{equation}\label{preciserlopez}
    \abs{r-2.544702668431974} <10^{-15}
\end{equation}
\end{remark}
\begin{lema}\label{gammaestimatives}
    The constants $\gamma_{\alpha_1},\ \gamma_{\alpha_2}$ defined in Definition \ref{gammaconstants}, which only depend on the conformal parameter $r$ of the López minimal Klein bottle are estimated as
    \begin{equation}
        0.1925\leq \gamma_{\alpha_1}\leq 0.1933,\quad  0.9678\leq \gamma_{\alpha_2}\leq 0.9682
    \end{equation}
\end{lema}
\begin{proof}
     By Definition \ref{gammaconstants} and Lemma \ref{lema0} we have an alternative way to determine the constants $\gamma_{\alpha_1}$ and $\gamma_{\alpha_2}$ in terms of the complete elliptic integrals of the first and second kind
    \begin{equation}\label{alternativeformulagamma}
        \gamma_{\alpha_1}=\frac{1}{F(k_{\alpha_1})}\left(\frac{E(k_{\alpha_1})-k_{\alpha_2}^2F(k_{\alpha_1})}{k_{\alpha_1}k_{\alpha_2}}\right),\quad \gamma_{\alpha_2}=\frac{1}{F(k_{\alpha_2})}\left(\frac{E(k_{\alpha_2})-k_{\alpha_1}^2F(k_{\alpha_2})}{k_{\alpha_1}k_{\alpha_2}}\right).
    \end{equation}
    Observe that the functions $k_{\alpha_1}(r),\ k_{\alpha_2}(r)$ of Definition \ref{gammaconstants} are monotone decreasing and increasing respectively. Therefore, since in particular $ r\in (2.5446,2.5448)$ by Lemma \ref{rLopezNumerical}, we have that
\begin{equation}\label{estimateskappa}
    k_{\alpha_1}\in (k_{\alpha_1}^{\text{min}},k_{\alpha_1}^{\text{max}}),\quad k_{\alpha_2}\in (k_{\alpha_2}^{\text{min}},k_{\alpha_2}^{\text{max}}),
\end{equation}
where
\begin{equation}\label{minmaxestimates0}
   k_{\alpha_1}^{\text{min}}= 0.365733,\quad k_{\alpha_1}^{\text{max}}=0.365760,\quad k_{\alpha_2}^{\text{min}}=0.93070,\quad k_{\alpha_2}^{\text{max}}=0.93072.
\end{equation}
As an application of Remark \ref{monotonicityellipticintegrals} we have for $k_{\alpha_1}$
\begin{equation*}
     F(k_{\alpha_1})\in (F(k_{\alpha_1}^{\text{min}}),F(k_{\alpha_1}^{\text{max}})),\quad E(k_{\alpha_1})\in (E(k_{\alpha_1}^{\text{max}}),E(k_{\alpha_1}^{\text{min}})),
\end{equation*}
while for $k_{\alpha_2}$
\begin{equation*}
     F(k_{\alpha_2})\in (F(k_{\alpha_2}^{\text{min}}),F(k_{\alpha_2}^{\text{max}})),\quad E(k_{\alpha_2})\in (E(k_{\alpha_2}^{\text{max}}),E(k_{\alpha_2}^{\text{min}})). 
\end{equation*}
Applying Lemma \ref{CalculationEllipticIntegrals} for  $k_{\alpha_1}^{\text{min}}$ and $k_{\alpha_1}^{\text{max}}$ with $N=7$, we obtain that
\begin{equation*}
    F(k_{\alpha_1})\in (F_N(k_{\alpha_1}^{\text{min}})-\epsilon,F_N(k_{\alpha_1}^{\text{max}})+\epsilon),\quad E(k_{\alpha_1})\in(E_N(k_{\alpha_1}^{\text{max}})-\epsilon,E_N(k_{\alpha_1}^{\text{min}})+\epsilon),
\end{equation*}
where $\epsilon=10^{-6}$ and $F_N(k_{\alpha_1}^{\text{min}}),\ F_N(k_{\alpha_1}^{\text{max}}),\ E_N(k_{\alpha_1}^{\text{min}}),\ E_N(k_{\alpha_1}^{\text{max}})$ are rational multiples of $\pi$, obtained from Definition \ref{truncationelliptic}. 
Therefore, we find that
\begin{equation}\label{estimativesalpha1}
    F(k_{\alpha_1})\in (F_{\text{min}}^{k_{\alpha_1}},F_{\text{max}}^{k_{\alpha_1}}),\quad  E(k_{\alpha_1})\in (E_{\text{min}}^{k_{\alpha_1}},E_{\text{max}}^{k_{\alpha_1}}),
\end{equation}
where
\begin{equation}\label{minmaxquantities1}
    F_{\text{min}}^{k_{\alpha_1}}=1.62767,\quad  F_{\text{max}}^{k_{\alpha_1}}=1.62780,\quad  E_{\text{min}}^{k_{\alpha_1}}=1.51685,\quad  E_{\text{max}}^{k_{\alpha_1}}=1.51688.
\end{equation}
Similarly, applying Lemma \ref{CalculationEllipticIntegrals} for each $k_{\alpha_2}^{\text{min}}$ and $k_{\alpha_2}^{\text{max}}$ with $N=120$, we obtain that
\begin{equation*}
    F(k_{\alpha_2})\in (F_N(k_{\alpha_2}^{\text{min}})-\epsilon,F_N(k_{\alpha_2}^{\text{max}})+\epsilon),\quad E(k_{\alpha_2})\in(E_N(k_{\alpha_2}^{\text{max}})-\epsilon,E_N(k_{\alpha_2}^{\text{min}})+\epsilon),
\end{equation*}
where $\epsilon=10^{-6}$ and $F_N(k_{\alpha_2}^{\text{min}}),\ F_N(k_{\alpha_2}^{\text{max}}),\ E_N(k_{\alpha_2}^{\text{min}}),\ E_N(k_{\alpha_2}^{\text{max}})$ are rational multiples of $\pi$, obtained from Definition \ref{truncationelliptic}. 
Therefore, we find that
\begin{equation}\label{estimativesalpha2}
 F(k_{\alpha_2})\in (F_{\text{min}}^{k_{\alpha_2}},F_{\text{max}}^{k_{\alpha_2}}),\quad  E(k_{\alpha_2})\in (E_{\text{min}}^{k_{\alpha_2}},E_{\text{max}}^{k_{\alpha_2}}),
\end{equation}
where
\begin{equation}\label{minmaxquantities2}
F_{\text{min}}^{k_{\alpha_2}}=2.44194,\quad  F_{\text{max}}^{k_{\alpha_2}}=2.44210,\quad  E_{\text{min}}^{k_{\alpha_2}}=1.13130,\quad  E_{\text{max}}^{k_{\alpha_2}}=1.13135.
\end{equation}
Therefore, we may estimate $\gamma_{\alpha_1}$ and $\gamma_{\alpha_2}$ using equations \eqref{alternativeformulagamma}, \eqref{estimateskappa},  \eqref{estimativesalpha1} and \eqref{estimativesalpha2} 
\begin{equation}\label{estimativegamma1}
\frac{E_{\text{min}}^{k_{\alpha_1}}}{F_{\text{max}}^{k_{\alpha_1}}k_{\alpha_1}^{\text{max}}k_{\alpha_2}^{\text{max}}}-\frac{k_{\alpha_2}^{\text{max}}}{k_{\alpha_1}^{\text{min}}}  \leq \gamma_{\alpha_1}\leq \frac{E_{\text{max}}^{k_{\alpha_1}}}{F_{\text{min}}^{k_{\alpha_1}}k_{\alpha_1}^{\text{min}}k_{\alpha_2}^{\text{min}}}-\frac{k_{\alpha_2}^{\text{min}}}{k_{\alpha_1}^{\text{max}}},
\end{equation}
\begin{equation}\label{estimativegamma2}
\frac{E_{\text{min}}^{k_{\alpha_2}}}{F_{\text{max}}^{k_{\alpha_2}}k_{\alpha_1}^{\text{max}}k_{\alpha_2}^{\text{max}}}-\frac{k_{\alpha_1}^{\text{max}}}{k_{\alpha_2}^{\text{min}}}  \leq \gamma_{\alpha_2}\leq \frac{E_{\text{max}}^{k_{\alpha_2}}}{F_{\text{min}}^{k_{\alpha_2}}k_{\alpha_1}^{\text{min}}k_{\alpha_2}^{\text{min}}}-\frac{k_{\alpha_1}^{\text{min}}}{k_{\alpha_2}^{\text{max}}}.
\end{equation}
Finally, substituting equations \eqref{minmaxestimates0}, \eqref{minmaxquantities1} and \eqref{minmaxquantities2} into equations \eqref{estimativegamma1} and \eqref{estimativegamma2} we conclude that
\begin{equation*}
    \gamma_{\alpha_1}\in (0.1925,0.1933),\quad \gamma_{\alpha_2}\in (0.9678,0.9682)
\end{equation*}
as we wanted to prove.
\end{proof}

\subsection{Explicit estimates for the positions of the branch points} \label{subsection::Explicit estimates for the positions of the branch points}

The next lemma provides information about the position of $a_t$ and $b_t$ in the real line, for every $t \in (0,1]$, when compared to $t$ and $r$. 

\begin{lema}\label{lem:posicoesdosatbt}

The following inequalities hold  for every $t \in (0,1]$, $$b_t < -r< -t<a_t<0.$$
    
\end{lema}

\begin{proof}

We use the intermediate value theorem to verify these estimates. First, observe that since
\begin{equation*}
    k_t(-r) = -(1 + r^2) (r - t + r^2 t - r t^2) < 0,
\end{equation*}
it follows that $b_t < -r$, as $k_t$ is a fourth order polynomial with positive leading coefficient. Now,
\begin{equation*}
k_t(-t) = \frac{2 t (1 + t^2) (t(1 - r^2) + r (-1 + t^2))}{r} < 0,
\end{equation*}
therefore, since $k_t(0) = t > 0$, we must have $-t < a_t < 0$.
\end{proof}
The following estimate will be useful when studying the period problem for the López minimal Klein bottle.
\begin{lema}\label{xtytestimates}
    For all $t\in (0,1]$, $x_t\in [2,4]$ and $y_t<-1.$
\end{lema}
\begin{proof}
    Since the discriminant $\delta$ is positive by equation \eqref{discriminant}, it follows from equation \eqref{xtyt} that for each $t\in (0,1]$, $x_t$ is the largest solution, while $y_t$ is the smallest solution to the second degree polynomial
    \begin{equation*}
      F_t(\lambda)=\lambda^2-\left(\frac{L_1(t)}{t}\right)\lambda +\left(\frac{L_2(t)}{t}+2\right).
    \end{equation*}
    We claim that $F_t(2)<0$ and $F_t(4)>0$, and as a consequence, $x_t\in [2,4]$ by the concavity of the polynomial $F_t(z).$

    After simplification, we obtain
    \begin{equation*}
        F_t(2)=-\frac{2f(t)}{rt},\quad f(t)=(-1-r)t^2+(2r^2-2r)t+(2 r^2-3r-1),
    \end{equation*}
    moreover, using that $r>2$, it holds that
    \begin{equation*}
        f(0)=r(2r-3)-1>0,\quad f(1)=2r(2r-3)-2>0
    \end{equation*}
Hence, we obtain that for all $t\in (0,1]$ $f(t)>0$, since the polynomial $f(t)$ is concave down and therefore $F_t(2)<0.$ Finally we prove that $F_t(4)>0$. Indeed,
\begin{equation*}
    F_t(4)=-\frac{2h(t)}{rt},\quad h(t)=-1 - 6 r + 2 r^2 + (-8 r + 4 r^2) t + (-1 - 2 r) t^2.
\end{equation*}
The discriminant $\delta_h$ of the polynomial $h(t)$ for each $t\in (0,1]$ is given by
\begin{equation*}
    \delta_h=4 (-1 - 8 r + 6 r^2 - 12 r^3 + 4 r^4).
\end{equation*}
Using the improved estimate for the conformal parameter $r$ of the López minimal Klein bottle of Lemma \ref{rLopezNumerical} we find that for all $t\in (0,1]$
\begin{equation*}
    \delta_h<-39<0,
\end{equation*}
which implies that the polynomial $h(t)$ does not have real roots, and then $h(t)<0$ since it is concave down. As a consequence we conclude that $F_t(4)>0$. Therefore $x_t\in [2,4]$. We now prove that $F_t(-1)<0$ which immediately implies that $y_t<-1$ by the concavity of $F_t(\lambda)$ and the property of $y_t$ as the smallest zero of the polynomial $F_t(\lambda)$. By direct evaluation we find that
\begin{equation*}
    F_t(-1)=\frac{u(t)}{rt},\quad u(t)=2 - 3 r - 4 r^2 + (r + 2 r^2) t + (2  - r) t^2.
\end{equation*}
Therefore using $r>2$
\begin{equation*}
    u(t)<2-3r-4r^2+r+2r^2=2(1-r-r^2)<-10<0
\end{equation*}
which proves that $F_t(-1)<0$, concluding the proof.
\end{proof}

For the following lemma, recall the continuous real valued functions $a_t$ and $b_t$ defined for $t \in (0,1]$ described in Lemma \ref{lem:defkt} and consider $\tau_t := t/a_t$. 
\begin{lema}\label{lem:estimatesforabtau}
    There exist piecewise linear functions with rational coefficients $\lvert a \lvert^{\inf}_t$, $\lvert a \lvert^{\sup}_t$, $\lvert b \lvert^{\inf}_t$, $\lvert b \lvert^{\sup}_t$, $\lvert \tau \lvert^{\inf}_t$, and $\lvert \tau \lvert^{\sup}_t$ defined on $(0,1]$, estimating $|a_t|, |b_t|$ and $|\tau_t|$ as
    \begin{equation}
    0 \leq \lvert a \lvert^{\inf}_t \leq \lvert a_t \lvert  \leq \lvert a \lvert_t^{\sup}, \quad 0\leq 
    \lvert b \lvert^{\inf}_t \leq \lvert b_t \lvert \leq \lvert b \lvert_t^{\sup}, \quad \text{ and } \quad
    0\leq \lvert \tau \lvert_t^{\inf} \leq \lvert \tau_t \lvert \leq \lvert \tau \lvert_t^{\sup}
    \end{equation}
    \noindent such that equations \eqref{thetaforresidues} and \eqref{theta-alpha1-eta1} hold.

\end{lema}
\begin{proof}

From the definition of $k_t$ in equation \eqref{eq:kt} and its factorization in Lemma \ref{lem:defkt}, we know that for any $t \in (0,1]$ the polynomial $k_t$ has exactly four distinct roots $a_t$, $b_t$, $-\frac{1}{a_t}$, and $-\frac{1}{b_t}$ that satisfy
\begin{equation*}
    b_t < a_t < 0 < - \frac{1}{b_t} < - \frac{1}{a_t}.
\end{equation*}

Since $k_t$ is a polynomial of degree four with a positive leading coefficient, any pair of negative numbers $a^{\inf}<a^{\sup}$ satisfying $k_t(a^{\inf}) < 0 < k_t(a^{\sup})$ is such that $a^{\inf} < a_t < a^{\sup}$. Similarly, for any pair of negative numbers $b^{\inf}<b^{\sup}$ such that $k_t(b^{\sup})<0<k_t(b^{\inf})$, it follows that $b^{\inf}<b_t<b^{\sup}$. These are the facts that we explore to construct non-positive piecewise linear functions with rational coefficients $a_t^{\inf}$, $a_t^{\sup}$, $b_t^{\inf}$ and $b_t^{\sup}$ which bound from below and from above the functions $a_t$ and $b_t$, respectively. For instance, we construct $a_t^{\inf}$ and $a_t^{\sup}$ using an ansatz and prove that $k_t(a_t^{\inf}) <0< k_t(a_t^{\sup})$ holds for every $t \in (0,1]$. The strategy above to estimate $a_t$ also indicates a way to estimate $\tau_t=t/a_t$ with analogous computations. 

As we saw, the strategy we follow here depends on the possibility of checking if the piecewise linear functions constructed produce a positive (or negative) function after composition with $k_t$. This sign verification relies on the strategy described in Appendix \ref{section::EstimatesPolynomials}, using the estimates for $r$ given by Lemma \ref{rLopezNumerical}. Finally, we mention that the functions we constructed are close enough to the exact functions $|a_t|, |b_t|$ and $|\tau_t|$ so that equations \eqref{thetaforresidues} and \eqref{theta-alpha1-eta1} hold. We describe these computations explicitly in the companion Mathematica notebook.

\end{proof}

\begin{lema}\label{lem:limitt/a}The continuous function $\frac{a_t}{t}$ defined for all $t\in (0,1]$ has a well-defined limit at $t=0$ 
\begin{equation*}
    \lim_{t\to 0^{+}}\frac{a_t}{t}=-\frac{1}{3}.
\end{equation*} 
\end{lema}
\begin{proof}
    Using equations \eqref{eq:kt}, \eqref{discriminant}, \eqref{xtyt} and \eqref{eq:atbt} we obtain after rationalization that
    \begin{equation*}
        \frac{a_t}{t}=-\frac{4}{((L_1(t)-\sqrt{\delta_t})^2+16t^2)^{1/2}-(L_1(t)-\sqrt{\delta_t})},
    \end{equation*}
    where
    \begin{equation*}
        L_1(t)=-3 + 2rt - t^2,\quad \delta_t=9+\left(4r-\frac{8}{r}\right)t+(4r^2+6)t^2-\left(4r+\frac{8}{r}\right)t^3+t^4>0.
    \end{equation*}
    The conclusion then follows from direct evaluation.
\end{proof}

\section{A framework for quantitative estimates for a polynomial in several variables
}\label{section::EstimatesPolynomials}

We consider a polynomial in several variables $P(u_1,...,u_k,t,\sigma_1,...,\sigma_l)$ with coefficients in the field of rational numbers $\mathbb{Q}$. In this appendix we study the map
\begin{equation}\label{eq:generalsettingpolnonzero}                  t \in [0,1] \mapsto P(u_1(t),...,u_k(t),t,\sigma_1,...,\sigma_l)
\end{equation}
    where $u_i(t) \geq 0$ is a list of non-negative continuous functions of $t$, $t \in [0,1]$, and $\sigma_j > 0$ is a list of positive real parameters. We want to guarantee that \eqref{eq:generalsettingpolnonzero} never vanishes and we now explain a sufficient condition that relies on a finite number of calculations with rational numbers. For simplicity, we explain the method for the case where we want to prove that \eqref{eq:generalsettingpolnonzero} is positive, because the case where we want to prove that \eqref{eq:generalsettingpolnonzero} is negative is analogous. The method depends mainly on the possibility of finding good approximations for the functions $u_i(t)$ using piecewise linear functions with rational coefficients. We also assume that one can find good rational approximations for the positive numbers $\sigma_j$. 
    
 The basic idea is to divide $[0,1]$ into smaller intervals, say $n \in \mathbb{N}$ intervals of equal length, $n\ge 2$, and to bound \eqref{eq:generalsettingpolnonzero} from below on each of these intervals by a positive rational number depending on the interval.

We denote the aforementioned approximations of $u_i$ and $\sigma_j$ by

\begin{equation*}
        0 \leq u_i^{\inf}(t) \leq u_i(t) \leq u_i^{\sup}(t), \quad \text{ and } \quad 0 \leq \sigma_j^{\inf} \leq \sigma_j \leq \sigma_j^{\sup},
\end{equation*}
\noindent for any $1 \leq i \leq k$ and $1 \leq j \leq l$.

We now define a new polynomial of $2k+2+2l$ variables $P^{\inf}$ that will be used to bound $P$ from below. We introduce new variables $u_{i,+},u_{i,-},t_+,t_-,\sigma_{j,+},\sigma_{j,-}$ to define $P^{\inf}(u_{i,-},u_{i,+},t_-,t_+,\sigma_{j,-},\sigma_{j,+})$ in the following way. A monomial of $P$ is of the form $ m_{I,J,d} := c_{I,J,d} \cdot u_1^{i_1}\cdots u_k^{i_k} \sigma_1^{j_1}\cdots \sigma_l^{j_l} t^d$, $I=(i_1,\dots,i_k)$, $J=(j_1,\dots,j_l)$. On the one hand, if $c_{I,J,d}\ge0$, then we use the monomial $$ m^{\inf}_{I,J,d} = c_{I,J,d} \cdot (u_{1,-})^{i_1}\cdots (u_{k,-})^{i_k} (\sigma_{1,-})^{j_1}\cdots (\sigma_{l,-})^{j_l} t_-^d$$ in $P^{\inf}$. On the other hand, if $c_{I,J,d}<0$, then we use the monomial $$ m^{\inf}_{I,J,d} = c_{I,J,d} \cdot (u_{1,+})^{i_1}\cdots (u_{k,+})^{i_k} (\sigma_{1,+})^{j_1}\cdots (\sigma_{l,+})^{j_l} t_+^d$$ in $P^{\inf}$. These are all the monomials of $P^{\inf}$. Note that, in particular, 

\begin{equation}
    P(u_i(t), t, \sigma_j) \geq P^{\inf}(u_i^{\inf}(t), u_i^{\sup}(t), T_-(t), T_+(t), \sigma_j^{\inf}, \sigma_j^{\sup}), \quad \text{ for every } t \in [0,1],
\end{equation}
\noindent as long as $0 \le T_{-}(t) \le t \le T_+(t)$ holds for every $t \in [0,1]$.

We now conclude the description of the method. First, we divide the interval $[0,1]$ into $n$ sub-intervals $I_s = [\frac{s}{n}, \frac{s+1}{n}]$, for $0 \leq s \leq n-1$. Then, in each sub-interval $I_s$, we estimate $u_i^{\inf}$ from below by its minimum $\min_{I_s} u_i^{\inf}$, $u_i^{\sup}$ from above by its maximum $\max_{I_s} u_i^{\sup}$, and $t$ from above and below by $\frac{s}{n} \le t \le \frac{s+1}{n}$. Observe that $\min_{I_s} u_i^{\inf}$ and $\max_{I_s} u_i^{\sup}$ are rational numbers, by construction. The construction gives $$P(u_i(t),t,\sigma_j) \ge P^{\inf}\left(\min_{I_s} u_i^{\inf}, \max_{I_s} u_i^{\sup}, \frac{s}{n}, \frac{s+1}{n}, \sigma_j^{\inf}, \sigma_j^{\sup}\right), \quad \text{for every } t \in I_s.$$

We then compute the rational number
\begin{equation}\label{eq:thetageral}
    \Theta(P; u_i^{\inf}, u_i^{\sup}, \sigma_j^{\inf}, \sigma_j^{\sup},n) := \min_{0 \leq s \leq n-1} P^{\inf}\left(\min_{I_s} u_i^{\inf}, \max_{I_s} u_i^{\sup}, \frac{s}{n}, \frac{s+1}{n}, \sigma_j^{\inf}, \sigma_j^{\sup}\right) \in \mathbb{Q}.
\end{equation}

\noindent Therefore,
$$P(u_i(t),t,\sigma_j) \ge \Theta(P; u_i^{\inf}, u_i^{\sup}, \sigma_j^{\inf}, \sigma_j^{\sup},n) \quad \text{for every } t \in [0,1].$$ 

In conclusion, if $\Theta(P; u_i^{\inf}, u_i^{\sup}, \sigma_j^{\inf}, \sigma_j^{\sup},n) > 0$, then \eqref{eq:generalsettingpolnonzero} is positive, as desired. This is conditioned to the approximations $u_i^{\inf}, u_i^{\sup} \sigma_j^{\inf}, \sigma_j^{\sup}$ being good enough.

\begin{exam}\label{ex:examplec2D}
\normalfont
    Here, we use the method described in this appendix in order to prove one of the assertions of Lemma \ref{propertiesconstants}. More specifically, we prove that $c_{2D}(t) < 0$ for $t \in [0,1]$. For this, we consider the polynomial in two variables
    \begin{equation*}
        P_{c_{2D}}(t,r) = 2 t^3 - 5 r t^2 + (2 + 2 r^2) t - 3r.
    \end{equation*}
    We use a number $n = 7$ of sub-intervals and we consider $r^{\inf} = 2.54470$ and $r^{\sup} = 2.54471$ (see Lemma \ref{rLopezNumerical}) to prove our claim. We construct the approximation from above as we have described and this gives us a new polynomial $P^{\sup}_{c2D}$ such that for $t \in [\frac{s}{n}, \frac{s+1}{n}]$ and $0 \leq s \leq n-1$, we have
    \begin{align*}
        P_{c_{2D}}(t,r) &\leq P^{\sup}_{c_{2D}}(\frac{s}{n}, \frac{s+1}{n},r^{\inf}, r^{\sup})\\ 
        &= 2 \left( \frac{s+1}{n} \right)^3 - 5 r^{\inf} \left( \frac{s}{n} \right)^2 + \left(2 + 2 (r^{\sup})^2 \right) \frac{s+1}{n} - 3 r^{\inf}.  
    \end{align*}
    Finally, introducing the notation $P^{\sup}_{c_{2D}}[s]:=P^{\sup}_{c_{2D}}(\frac{s}{n}, \frac{s+1}{n},r^{\inf}, r^{\sup})$ we compute
    \[
\begin{array}{c|c|c|c|c|c|c|c}
  s & 0 & 1 & 2 & 3 & 4 & 5 & 6  \\ \hline
   P^{\sup}_{c_{2D}}[s] & r_0 & r_1 & r_2 & r_3 & r_4 & r_5 & r_6
\end{array}
\]
    Here, $r_0$ are explicit rational numbers such that $r_0 \leq -5.4923$, $r_1 \leq -3.5753$, $r_2 \leq -2.1077$, $r_3 \leq -1.0544$, $r_4 \leq -0.3804$, $r_5 \leq -0.0509$ and $r_6 \leq -0.0308$.
    We then conclude that $$\Theta(P_{c2D};r^{\inf},r^{\sup},7)<0$$ \noindent using the $\Theta$ function defined in equation \eqref{eq:thetageral}, and this proves that $c_{2D}(t) < 0$ for $t \in [0,1]$.
    \end{exam}

\printbibliography

\end{document}